\documentclass[12pt,reqno]{amsart}

\usepackage{color}
\usepackage{amsmath, amsthm, amssymb, amsxtra}
\usepackage{amsfonts}
\usepackage[utf8]{inputenc}
\usepackage[dvips]{epsfig}
\usepackage{graphicx}
\usepackage[english]{babel}
\usepackage{hyperref}

\usepackage{graphicx}
\usepackage{latexsym}
\usepackage{mathtools}
\usepackage{esint}
\usepackage{float}

\theoremstyle{plain}
\newtheorem{thm}{Theorem}[section]
\newtheorem{cor}[thm]{Corollary}
\newtheorem{lem}[thm]{Lemma}
\newtheorem{prop}[thm]{Proposition}

\theoremstyle{definition}
\newtheorem{defn}[thm]{Definition}

\theoremstyle{remark}
\newtheorem{obs}[thm]{Remark}

\numberwithin{equation}{section}

\newcommand{\average}{{\mathchoice {\kern1ex\vcenter{\hrule height.4pt
width 6pt depth0pt} \kern-9.7pt} {\kern1ex\vcenter{\hrule
height.4pt width 4.3pt depth0pt} \kern-7pt} {} {} }}

\newcommand{\R}{\mathbb R}
\newcommand{\N}{\mathbb N}
\renewcommand{\L}{\mathcal L}

\newcommand{\p}{\partial}

\newcommand{\comment}[1]{}
\newcommand{\Mplus}{\mathcal M^+}
\newcommand{\Mminus}{\mathcal M^-}
\newcommand{\loc}{\mathrm{loc}}

\begin{document}

\title[Parabolic boundary Harnack inequalities with right-hand side]{Parabolic boundary Harnack inequalities\\ with right-hand side}
\author{Clara Torres-Latorre}
\address{Universitat de Barcelona, Departament de Matem\`atiques i Inform\`atica, Gran Via de les Corts Catalanes 585, 08007 Barcelona, Spain.}
\email{\tt claratorreslatorre@ub.edu}
\begin{abstract}
We prove the parabolic boundary Harnack inequality in parabolic flat Lipschitz domains by blow-up techniques, allowing for the first time a non-zero right-hand side. Our method allows us to treat solutions to equations driven by non-divergence form operators with bounded measurable coefficients, and a right-hand side $f \in L^q$ for $q > n+2$. In the case of the heat equation, we also show the optimal $C^{1-\varepsilon}$ regularity of the quotient.

As a corollary, we obtain a new way to prove that flat Lipschitz free boundaries are $C^{1,\alpha}$ in the parabolic obstacle problem and in the parabolic Signorini problem.
\end{abstract}

\thanks{The author has received funding from the European Research Council (ERC) under the Grant Agreement No 801867, and from AEI project PID2021-125021NAI00 (Spain).}
\maketitle

\section{Introduction}\label{sect:intro}
The well known elliptic boundary Harnack inequality asserts that the rate at which positive harmonic functions approach zero Dirichlet boundary conditions depends only on the geometry of the domain. Quantitatively, if $u$ and $v$ are positive harmonic functions in $\Omega$ that vanish on $\p\Omega$, then the quotient $u/v$ is bounded near the boundary. This is also known as the Carleson estimate.

An important corollary of the boundary Harnack is that $u/v$ is not only bounded, but also Hölder continuous. In the elliptic case, the $C^{0,\alpha}$ regularity of the quotient follows from the Carleson estimate by a standard iteration technique (see \cite{CS05}). However, in the parabolic setting, the question is much more delicate due to the time delay in the interior Harnack inequality. The first proof of the Hölder regularity of the quotient for solutions to the heat equation appeared in \cite{ACS96}, two decades after the first proof of the Carleson estimate for caloric functions \cite{Kem72b}.

When the domains are $C^{1,\mathrm{Dini}}$ or smoother, a combination of the Hopf lemma and the Lipschitz regularity of solutions implies that all solutions to elliptic and parabolic equations decay linearly as they approach zero Dirichlet boundary conditions. However, in less regular domains where the Hopf lemma and Lipschitz continuity do not hold, the fact that the quotient of solutions is bounded is far from trivial and needs to be studied separately.

In this work, we provide a new approach to boundary Harnack inequalities with right-hand side, extending the previous results of Allen and Shahgholian \cite{AS19}, and Ros-Oton and the author \cite{RT21} to the parabolic setting. Our proof relies on comparison and scaling arguments, as in \cite{RT21}, as well as a blow-up argument inspired by \cite{RS17a}. We will also consider the applications of our result to free boundary problems, such as the parabolic obstacle problem and the parabolic Signorini problem. Using the boundary Harnack, we can prove $C^{1,\alpha}$ regularity of free boundaries when we know they are flat Lipschitz.

\begin{obs}
    In flat Lipschitz and $C^1$ domains, the boundary Harnack and the boundary regularity of solutions do not follow from each other. Nevertheless, they are intimately related, and both can be proved by a contradiction-compactness argument where the Hölder exponent is determined by a Liouville theorem in the half-space.
    
    In this regard, we will prove that, if $u$ and $v$ are positive harmonic functions with zero Dirichlet boundary conditions, $u/v \in C^{1-\varepsilon}$ if the boundary is sufficiently flat, matching the known regularity up to the boundary of $u$ and $v$, and the fact that the only harmonic function with sublinear growth in a half-space and zero Dirichlet boundary conditions is zero.
\end{obs}

\subsection{Main results}

In the following, $\mathcal{L}$ will denote a non-divergence form elliptic operator with bounded measurable coefficients,
\begin{equation}\label{eq:non-divergence_operator}
\mathcal{L}u = \sum\limits_{i,j = 1}^na_{ij}(x)\p^2_{ij}u, \quad \text{with} \quad \lambda I \leq A(x) \leq \Lambda I,
\end{equation}
with $0 < \lambda \leq \Lambda$.

We denote by $\alpha_0(\lambda,\Lambda) \in (0,1)$ a universal constant (only dependent on the dimension and the ellipticity constants), which is defined as the minimum of the following:
\begin{itemize}
\item The $C^{1,\alpha_0}$ boundary regularity estimate in \cite[Theorem 2.1]{Wan92b}.
    \item The interior $C^{0,\alpha_0}$ regularity estimate in \cite[Lemma 5.1]{CKS00}.
\end{itemize}
We will define $\alpha_0(1,1) := 1$ instead if the operator is the Laplacian.

Our main result is the following boundary Harnack inequality, which extends the main result in \cite{ACS96} to general non-divergence form operators and equations with right-hand side. Here, $C^{0,\gamma}_p$ is the parabolic Hölder space defined in Section \ref{sect:previ}.

\begin{thm}\label{thm:main}
Let $q > n+2$, $0 < \gamma < \min\{\alpha_0,1-\frac{n+2}{q}\}$, $m \in (0,1]$, and let $\L$ be a non-divergence form operator as in \eqref{eq:non-divergence_operator}. There exists $c_0 \in (0,1)$, only depending on $q$, $\gamma$, the dimension and the ellipticity constants, such that the following holds.

Let $\Omega \subset \R^{n+1}$ be a parabolic Lipschitz domain in $Q_1$ in the sense of Definition \ref{defn:lipschitz_domain} with Lipschitz constant $L \leq c_0$. Let $u$ and $v$ be solutions to
$$\left\{\begin{array}{rclll}
u_t - \L u & = & f_1 & \text{in} & \Omega\\ 
u & = & 0 & \text{on} & \p_\Gamma\Omega
\end{array}\right.\quad\text{and}\quad\left\{\begin{array}{rclll}
v_t - \L v & = & f_2 & \text{in} & \Omega\\ 
v & = & 0 & \text{on} & \p_\Gamma\Omega,
\end{array}\right.$$
and assume that $\|u\|_{L^\infty(Q_1)} \leq 1$, $\|v\|_{L^\infty(Q_1)} = 1$, $v > 0$, $v\left(\frac{e_n}{2},-\frac{3}{4}\right) \geq m$, ${\|f_1\|_{L^q(Q_1)} \leq 1}$ and $\|f_2\|_{L^q(Q_1)} \leq c_0m$.

Then,
$$\left\|\frac{u}{v}\right\|_{C^{0,\gamma}_p(\Omega\cap Q_{1/2})} \leq C,$$ where $C$ depends only on $q$, $m$, $\gamma$, the dimension and the ellipticity constants.
\end{thm}

\begin{obs}\label{obs:main_alpha}
We will actually prove Theorem \ref{thm:main} under the more general assumption that $f_i = g_i + h_i$ with
$$\|d^{1-\alpha}g_1\|_{L^\infty(Q_1)}+\|d^{-1/(n+1)-\alpha}h_1\|_{L^{n+1}(Q_1)} \leq 1$$
and
$$\|d^{1-\alpha}g_2\|_{L^\infty(Q_1)}+\|d^{-1/(n+1)-\alpha}h_2\|_{L^{n+1}(Q_1)} \leq c_0m,$$
where $d(x',x_n,t) = x_n - \Gamma(x',t)$, and $\alpha \in (\gamma,\min\{\alpha_0,1-\frac{n+2}{q}\})$  (see Proposition \ref{prop:dist_interpolation}). In this case, $C$ depends also on $\alpha$.
\end{obs}

\begin{obs}
The result is sharp in the following sense:
    \begin{itemize}
        \item If the domain is $C^1$, $\L = \Delta$ and $q = \infty$, we may take $\gamma \rightarrow 1$.
        \item If the Lipschitz constant is not small, the result fails, even for $\L = \Delta$ and $q = \infty$.
        \item If the norm of the right-hand side is big, the result fails.
        \item If $q = n+2$, the result fails for any $c_0 > 0$, even for $\L = \Delta$.
    \end{itemize}
    Counterexamples can be constructed by a straightforward adaptation of \cite[Section 6]{RT21} to the parabolic setting. See also \cite{AS19}.
\end{obs}

Assuming that both solutions are positive and the right-hand sides of the equations are small, we can use symmetry to deduce the Carleson estimate.

\begin{cor}\label{cor:main_carleson}
Let $q > n+2$, $m \in (0,1]$, and let $\L$ be a non-divergence form operator as in \eqref{eq:non-divergence_operator}. There exists $c_0 \in (0,1)$, only depending on $q$, the dimension and the ellipticity constants, such that the following holds.

Let $\Omega \subset \R^{n+1}$ be a parabolic Lipschitz domain in $Q_1$ in the sense of Definition \ref{defn:lipschitz_domain} with Lipschitz constant $L \leq c_0$. Let $u$ and $v$ be positive solutions to
$$\left\{\begin{array}{rclll}
u_t - \L u & = & f_1 & \text{in} & \Omega\\ 
u & = & 0 & \text{on} & \p_\Gamma\Omega
\end{array}\right.\quad\text{and}\quad\left\{\begin{array}{rclll}
v_t - \L v & = & f_2 & \text{in} & \Omega\\ 
v & = & 0 & \text{on} & \p_\Gamma\Omega,
\end{array}\right.$$
and assume that $\|u\|_{L^\infty(Q_1)} = \|v\|_{L^\infty(Q_1)} = 1$, $v\left(\frac{e_n}{2},-\frac{3}{4}\right) \geq m$, $u\left(\frac{e_n}{2},-\frac{3}{4}\right) \geq m$ and $\|f_i\|_{L^q(Q_1)} \leq c_0m$.

Then,
$$\frac{1}{C} \leq \frac{u}{v} \leq C \quad \text{in} \ \Omega\cap Q_{1/2},$$
where $C$ depends only on $q$, $m$, the dimension and the ellipticity constants.
\end{cor}

We will also deal with solutions to the heat equation in slit domains, that appear naturally when studying the parabolic Signorini problem. Our result gives an alternative proof to the boundary Harnack in \cite{PS14} when the Lipschitz constant of the domain is small, relaxing the condition on the right-hand side from $L^\infty$ to $L^q$, and providing the optimal Hölder regularity of the quotient.

\begin{thm}\label{thm:slit_main}
Let $q > n+2$, $0 < \gamma < \min\{1,\frac{3}{2}-\frac{n+3}{q}\}$, and $m \in (0,1]$. There exists $c_0 \in (0,\frac{1}{8})$, only depending on $q$, $\gamma$ and the dimension such that the following holds.

Let $\Omega \subset \R^{n+2}$ be a parabolic slit domain in $Q_1$ in the sense of Definition \ref{defn:slit_domain} with Lipschitz constant $L \leq c_0$. Let $u$ and $v$ be solutions to
$$\left\{\begin{array}{rclll}
u_t - \Delta u & = & f_1 & \text{in} & \Omega\\ 
u & = & 0 & \text{on} & \p_\Gamma\Omega
\end{array}\right.\quad\text{and}\quad\left\{\begin{array}{rclll}
v_t - \Delta v & = & f_2 & \text{in} & \Omega\\ 
v & = & 0 & \text{on} & \p_\Gamma\Omega,
\end{array}\right.$$
and assume that $u$ and $v$ are even in $x_{n+1}$, $\|u\|_{L^\infty(Q_1)} \leq 1$, $\|v\|_{L^\infty(Q_1)} = 1$, $v > 0$, $v\left(\frac{e_n}{2},-\frac{3}{4}\right) \geq m$, and that $\|f_1\|_{L^q(Q_1)} \leq 1$ and $\|f_2\|_{L^q(Q_1)} \leq c_0m$.

Then,
$$\left\|\frac{u}{v}\right\|_{C^{0,\gamma}_p(\Omega\cap Q_{1/2})} \leq C,$$ where $C$ depends only on $m$, $q$, $\gamma$ and the dimension.
\end{thm}

\begin{obs}
    Our method for studying non-divergence form operators relies on constructing homogeneous barriers. While we can create sub- and supersolutions with near-linear homogeneity in almost-flat parabolic Lipschitz domains, this method fails in the case of slit domains when the operator is not the Laplacian. Specifically, there is a gap between the homogeneity of sub- and supersolutions, which is expected when the equation is driven by an operator with coefficients \cite{RS17b, CFR23}. However, in the case of the Laplacian, we can still construct barriers with homogeneity close to $\frac{1}{2}$, as shown in Proposition \ref{prop:cone_solns}.
\end{obs}

\subsection{Known results}

The boundary Harnack is a fundamental tool in the realm of analysis and partial differential equations that has had significant impact over the past 50 years. While there is a vast array of literature on this topic and its applications, we have compiled a representative sample of the most noteworthy advancements, exclusively focusing on elliptic and parabolic equations in domains less regular than $C^{1,\mathrm{Dini}}$.

\subsubsection{Elliptic boundary Harnack}
The first proof of the classical case for harmonic functions in Lipschitz domains was given by Kemper in \cite{Kem72a}. Caffarelli, Fabes, Mortola and Salsa considered operators in divergence form in Lipschitz domains, while the case of operators in non-divergence form was treated by Fabes, Garofalo, Marin-Malave and Salsa \cite{CFMS81,FGMS88}. Jerison and Kenig extended the same result to NTA domains for divergence form operators \cite{JK82}, and the case of non-divergence operators in Hölder domains with $\alpha > 1/2$ was treated with probabilistic techniques in \cite{BB94} by Bass and Burdzy. A simple and unified proof of these previous results was recently presented by De Silva and Savin \cite{DS20, DS22a}.

The boundary Harnack inequality also holds for solutions to elliptic equations with right-hand side. 
Allen and Shahgholian investigated operators in divergence form in Lipschitz domains with a right-hand side in a weighted $L^\infty$ space \cite{AS19}. In a subsequent work with Kriventsov, they developed a general theory to derive boundary Harnack inequalities for equations with right-hand side based on the boundary Harnack for homogeneous equations \cite{AKS22}. Ros-Oton and the author studied non-divergence and divergence form operators in Lipschitz domains with small Lipschitz constant and small right-hand side in $L^q$ with $q > n$ \cite{RT21}.

\subsubsection{Parabolic boundary Harnack}
The presence of a waiting time in the parabolic interior Harnack inequality exacerbates the complexity of the parabolic problem, as it renders several approaches to the elliptic setting inapplicable.

Kemper first established the parabolic boundary Harnack inequality for the heat equation \cite{Kem72b}. Fabes, Garofalo, and Salsa extended it to equations in divergence form in Lipschitz cylinders \cite{Sal81, FGS86}. Athanasopoulos, Caffarelli, and Salsa proved the Hölder continuity of quotients of positive solutions for divergence form operators \cite{ACS96}, and Fabes, Safonov, and Yuan extended this result to non-divergence form equations \cite{FS97, FSY99}. Bass and Burdzy used probabilistic techniques to handle Hölder cylindrical domains \cite{BB92}, while Hoffman, Lewis, and Nyström treated unbounded parabolically Reifenberg flat domains \cite{HLN04}, and Petrosyan and Shi dealt with Lipschitz slit domains, where they allowed for a $L^\infty$ right-hand side in the equation \cite{PS14}. Recently, De Silva and Savin developed a unified and simplified approach to prove the Carleson estimate (but not the $C^{0,\alpha}$ regularity of the quotient) for both divergence and non-divergence equations in various settings \cite{DS22b}.

\subsection{Parabolic obstacle problems}

Boundary Harnack inequalities are a key tool in regularity theory for obstacle problems. They are used to establish $C^{1,\alpha}$ regularity of free boundaries from Lipschitz regularity, following from an original idea of Athanasopoulos and Caffarelli \cite{AC85}. For an introduction to this strategy, see also \cite[Section 6.2]{PSU12} and \cite[Section 5.6]{FR22}.

Let us briefly sketch how the technique works. In the elliptic setting, if $u$ is a solution to the obstacle problem
\begin{equation*}
\left\{
\begin{array}{rcl}
\Delta u & = & f\chi_{\{u > 0\}}\\
u & \geq & 0,
\end{array}
\right.
\end{equation*}
then we call $\{u = 0\}$ the contact set and $\p\{u > 0\}$ the free boundary. The derivatives of $u$ are solutions of
\begin{equation*}
\left\{
\begin{array}{rcll}
\Delta (\p_e u) & = & \p_e f & \text{in } \{u > 0\}\\
\p_e u & = & 0 & \text{on } \p\{u > 0\},
\end{array}
\right.
\end{equation*}
where $e$ is an unit vector. Then, applying the boundary Harnack to the derivatives of $u$ yields $u_i/u_n \in C^{0,\alpha}$ for every coordinate $i$, implying that the normal vector to $\p\{u > 0\}$ is $C^{0,\alpha}$ and hence $\p\{u > 0\}$ is $C^{1,\alpha}$.

Previously, this approach was restricted to the case where $f$ was constant because the known boundary Harnack inequalities applied only to equations without a right-hand side. Consequently, alternative methods such as those found in \cite{Bla01, ALS13, AK23} were used to establish $C^1$ or $C^{1,\alpha}$ regularity of free boundaries where the direct application of boundary Harnack was not feasible. However, new boundary inequalities that accommodate a right-hand side have been developed, enabling simpler proofs and reduced regularity assumptions on the obstacle \cite{AS19, RT21, AKS22}.

In the parabolic obstacle problem with a smooth obstacle, it is well known that the free boundary is $C^{1,\alpha}$ at the points where it is flat Lipschitz \cite{Caf77, AK23}. Consider the parabolic obstacle problem
\begin{equation}\label{eq:parabolic_obstacle}
\left\{
\begin{array}{rcl}
\p_t u - \Delta u & = & f\chi_{\{u > 0\}}\\
u & \geq & 0.
\end{array}
\right.
\end{equation}

Then, we have the following improvement of flatness result for the free boundary. We will state the hypotheses that are typically assumed in regular points as part of the full program to prove free boundary regularity, but we will present a standalone result that is simpler to state and prove.

\begin{cor}\label{cor:classic_parabolic_obstacle}
There exists a dimensional constant $L(n)$ such that the following holds. Let $u \in C^{1,1}_x\cap C^1_t(Q_1)$ be a solution to the parabolic obstacle problem \eqref{eq:parabolic_obstacle} with $f \in W^{1,q}$, where $q > n+2$. Assume that $(0,0) \in \p\{u > 0\}$ and that $\{u > 0\}$ is a parabolic Lipschitz domain in $Q_1$ in the sense of Definition \ref{defn:lipschitz_domain} with Lipschitz constant $L(n)$. Additionally, assume that $\p_nu \geq cd$ in $\{u > 0\}$, where $c > 0$ and $d$ is the distance to $\p\{u > 0\}$.

Then, $\p\{u > 0\}$ is a $C^{1,\alpha}$ graph in a neighbourhood of $(0,0)$ for some $\alpha > 0$.
\end{cor}

We will now examine the no-sign parabolic obstacle problem, expressed as:
\begin{equation}\label{eq:nosign_parabolic_obstacle}
\left\{
\begin{array}{rcl}
\p_t u - \Delta u & = & f\chi_{\{u \neq 0\}}\\
u & \geq & 0,
\end{array}
\right.
\end{equation}
This problem was first studied in \cite{CPS04} for $f \equiv 1$. In \cite{ALS13}, it was established that $f$ must belong to $C^{\mathrm{Dini}}$ to guarantee free boundary regularity, and that the free boundary is $C^1$ with respect to the parabolic metric at regular points. Furthermore, assuming $f \in W^{1,q}$ with $q > n + 2$, we can derive $C^{1,\alpha}$ regularity of the free boundary in space and time at those points using a similar technique to the proof of Corollary~\ref{cor:classic_parabolic_obstacle}.

Our next example is the fully nonlinear parabolic obstacle problem, which has been studied in \cite{AK23}. The problem is expressed as follows:
\begin{equation}\label{eq:FN}
    \left\{\begin{array}{rclll}
    \p_tu - F(D^2u,x) & = & f(x)\chi_{\{u > 0\}} & \text{in} & Q_1\\
    u \geq 0, & & u_t \geq 0 & \text{in} & Q_1,
    \end{array}\right.
\end{equation}
where we assume that $F$ satisfies the following conditions:
\begin{itemize}
    \item[\textit{(H1)}] $F$ is uniformly elliptic and $F(0,\cdot) \equiv 0$.
    \item[\textit{(H2)}] $F$ is convex and $C^1$ in the first variable.
    \item[\textit{(H3)}] $F$ is $W^{1,q}$ in the second variable for some $q > n+2$.
\end{itemize}
In \cite{AK23}, it is shown that the free boundary is $C^\infty$ at regular points under the assumption that $F$ and $f$ are smooth. Prior studies of the fully nonlinear parabolic obstacle problem in various settings have established the $C^1_p$ regularity of the free boundary at regular points \cite{FS15, IM16}. We believe that the technique used in \cite[Corollary 1.7]{RT21} for the elliptic case can be adapted to the parabolic problem to derive $C^{1,\alpha}$ regularity of the free boundary when $f \in W^{1,q}$ with $q > n + 2$.

\subsection{The parabolic Signorini problem}
The parabolic thin obstacle problem, also known as the parabolic Signorini problem, has been extensively studied in \cite{DGPT17}. In $\R^{n+2}$, it can be formulated as follows, with $Q_1^+ := Q_1\cap\{x_{n+1} > 0\}$:
\begin{equation}\label{eq:parabolic_signorini}
\left\{
\begin{array}{rclll}
\p_t u - \Delta u & = & f & \text{in} & Q_1^+\\
u \ \geq \ 0, \ -\p_{n+1}u \ \geq \ 0, \ u\p_{n+1}u & = & 0 & \text{on} & \{x_{n+1} = 0\}
\end{array}
\right.
\end{equation}
After an even extension in the $x_{n+1}$ variable, solutions to \eqref{eq:parabolic_signorini} also satisfy
\begin{equation}\label{eq:parabolic_signorini_even}
\left\{
\begin{array}{rclll}
\p_t u - \Delta u & = & f & \text{in} & Q_1 \setminus \Lambda(u)\\
u & = & 0 & \text{on} & \Lambda(u),
\end{array}
\right.
\end{equation}
where $\Lambda(u) \subset {x_{n+1} = 0}$ and $u(x',x_n,x_{n+1},t) = u(x',x_n,-x_{n+1},t)$.

Smoothness of the free boundary near regular points has been established in the case where $f$ is smooth \cite{BSZ17}. For obstacles with lower regularity, it has been proven in \cite{PZ19} that $u_t$ is continuous at regular free boundary points, implying that the free boundary is locally a $C^{1,\alpha}$ graph in both space and time near regular points, under the assumption that $f \in C^2_p$.

If $f$ is independent of time, corresponding to a stationary obstacle, the arguments presented in \cite{PZ19} hold. In this case, it is expected that the condition for the parabolic Lipschitz regularity of the free boundary, currently established with the assumption that $f \in C^{3/2}_p$ in \cite{DGPT17}, can be relaxed. On the other hand, once the free boundary is known to be Lipschitz, the weaker assumption that $f \in W^{1,q}$ for some $q > n + 2$ is sufficient to deduce that it is $C^{1,\alpha}$, thanks to Theorem \ref{thm:slit_main}. In the following result, we assume the nondegeneracy condition that is expected to hold at regular points, similar to Corollary \ref{cor:classic_parabolic_obstacle}.

\begin{cor}\label{cor:parabolic_signorini}
There exists a dimensional constant $L(n)$ such that the following holds. Let $u \in C^{3/2}_x\cap C^1_t(Q_1)$ be a solution to \eqref{eq:parabolic_signorini_even} with $f \in W^{1,q}$, where $q > n + 3$. Assume that $(0,0) \in \p\{u > 0\}$ and that $\{u > 0\}$ is a parabolic Lipschitz slit domain in $Q_1$ in the sense of Definition \ref{defn:slit_domain} with Lipschitz constant $L(n)$. Additionally, assume that $\p_nu \geq cd^{1/2}$ in $\{u > 0\}\cap\{x_{n+1} = 0\}$, where $c > 0$ and $d$ is the distance to $\p\{u > 0\}$.

Then, $\p\{u > 0\}$ is a $C^{1,\alpha}$ graph in a neighbourhood of $(0,0)$ for some $\alpha > 0$ (in the relative topology of $\{x_{n+1} = 0\}$).
\end{cor}

Determining the optimal regularity for the obstacle to ensure that the free boundary is $C^{1,\alpha}$ at regular points remains an open problem.

\subsection{Optimal boundary regularity for the quotient of harmonic functions}

We begin by recalling the results for the regularity of harmonic functions in different types of domains (see \cite[Section 2.6]{FR22} and the references therein). If $\Delta u = 0$ in $\Omega$, then the regularity of $u$ depends on the regularity of $\Omega$ as follows:
\begin{itemize}
    \item If $\Omega$ is a $C^{1,\alpha}$ domain, $u \in C^{1,\alpha}(\overline{\Omega})$.
    \item If $\Omega$ is a $C^1$ domain, $u \in C^{1-\varepsilon}(\overline{\Omega})$ for all $\varepsilon > 0$.
    \item If $\Omega$ is a Lipschitz domain with constant $L$, $u \in C^{0,\gamma}(\overline{\Omega})$, with $\gamma$ depending only on the dimension and $L$. Moreover, $\gamma \nearrow 1$ as $L \searrow 0$.
\end{itemize}
After flattening the boundary with a change of variables, we obtain another angle to see the same phenomenon. If $\L u = 0$ in a half space, where $\L u := \operatorname{Div}(A(x)\nabla u)$ is an elliptic operator in divergence form,
\begin{itemize}
    \item If $A(x) \in C^{0,\alpha}$, then $u \in C^{1,\alpha}$ by Schauder theory.
    \item If $A(x) \in C^0$, then $u \in C^{1-\varepsilon}$ for all $\varepsilon > 0$ by Cordes-Nirenberg.
    \item If $A(x) \in L^\infty$, then $u \in C^{0,\gamma}$ for a small $\gamma$ by De Giorgi. Moreover, $\gamma \nearrow 1$ as $A(x) \rightarrow I$ by Cordes-Nirenberg.
\end{itemize}
On the other hand, the regularity of the quotient of two harmonic functions with zero Dirichlet boundary data goes as follows. If $\Delta u = \Delta v = 0$ in $\Omega$ and $u = v = 0$ on $\p\Omega$, then
\begin{itemize}
    \item If $\Omega$ is a $C^{1,\alpha}$ domain, $u/v \in C^{1,\alpha}(\overline{\Omega})$, \cite{DS15}.
    \item If $\Omega$ is a $C^1$ domain, $u/v \in C^{1-\varepsilon}(\overline{\Omega})$ for all $\varepsilon > 0$, \cite{Kuk22}.
    \item If $\Omega$ is a Lipschitz domain with constant $L$, $u/v \in C^{0,\gamma}$.
\end{itemize}

The Hölder exponent in Boundary Harnack inequalities for Lipschitz domains is often suboptimal due to its dependence on an iteration scheme. However, our Theorem \ref{thm:main} provides a way to bridge the gap between Lipschitz and $C^1$ domains by showing that as the Lipschitz constant $L$ approaches zero, the exponent $\gamma$ can approach $1$.

\begin{cor}\label{cor:optimal_quotient_regularity}
Let $\varepsilon > 0$. There exists $L_0 > 0$, only depending on the dimension and $\varepsilon$, such that the following holds.

Let $\Omega$ be a Lipschitz domain in $B_1$ in the sense of Definition \ref{defn:elliptic_domains} with Lipschitz constant $L_0$. Let $u$ and $v$ be solutions to
$$\left\{\begin{array}{rclll}
\Delta u & = & 0 & \text{in} & \Omega\\ 
u & = & 0 & \text{on} & \p_\Gamma\Omega,
\end{array}\right.\quad\text{and}\quad\left\{\begin{array}{rclll}
\Delta v & = & 0 & \text{in} & \Omega\\ 
v & = & 0 & \text{on} & \p_\Gamma\Omega.
\end{array}\right.$$

Then, $u/v \in C^{1 - \varepsilon}(B_{1/2}\cap\overline{\Omega})$.
\end{cor}

The proof follows immediately from Theorem \ref{thm:main}.

\subsection{A general Hopf lemma for parabolic equations with right-hand side}
In this section, we present a Hopf lemma for parabolic equations in parabolic $C^{1,\mathrm{Dini}}$ domains. As far as we know, this is the first result of this kind in the parabolic setting that allows for equations with right-hand side. While this is a relatively unexplored area in the literature, similar results have been established previously in the elliptic case \cite{BM18,RT21}.

\begin{cor}\label{cor:hopf_parabolic_rhs}
Let $\alpha \in (0,\alpha_0)$, and let $\L$ be a non-divergence form operator as in \eqref{eq:non-divergence_operator}. There exists $c_0 > 0$, only depending on $\alpha$, the dimension, and the ellipticity constants, such that the following holds.

Let $\Omega$ be a parabolic Lipschitz domain in $Q_1$ in the sense of Definition \ref{defn:lipschitz_domain} with Lipschitz constant $L \leq c_0$, and assume that it satisfies the interior $C^{1,\mathrm{Dini}}$ condition at $0$ in the sense of Definition \ref{defn:interior_C1_Dini}. Let $u$ be a positive solution to
$$\left\{\begin{array}{rclll}
u_t - \L u & = & f & \text{in} & \Omega\\ 
u & = & 0 & \text{on} & \p_\Gamma\Omega,
\end{array}\right.$$
and assume that $f = g + h$ , with
$$\|d^{1-\alpha}g\|_{L^\infty(Q_1)}+\|d^{-1/(n+1)-\alpha}h\|_{L^{n+1}(Q_1)} \leq c_0u\left(\frac{e_n}{2},-\frac{1}{2}\right),$$
where $d(x',x_n,t) = x_n - \Gamma(x',t)$. Then, for all $r \in (0,\delta)$,
$$u(re_n,0) \geq cr,$$
for some small $c, \delta > 0$.
\end{cor}
\vspace{-0.15cm}
\begin{obs}
    If the boundary of $\Omega$ is a $C^1$ graph, one can obtain a Lipschitz constant as small as necessary by scaling.
\end{obs}
\vspace{-0.15cm}
\subsection{Plan of the paper}

The paper is organized as follows.

We begin in Section \ref{sect:previ} by recalling some classical tools that we will use throughout the paper, such as interior regularity estimates and a Liouville theorem.

In Section \ref{sect:bdry_growth_reg}, we derive precise growth estimates near the boundary for solutions of parabolic equations. This allows us to construct a special solution in Section \ref{sect:special_soln} that is almost proportional to the distance to the boundary. Using this special solution, we prove our main result, Theorem \ref{thm:main}, in Section \ref{sect:proof_BH}. Our proof relies crucially on the growth properties of the special solution.

In Section \ref{sect:slit}, we apply the same strategy used in Sections \ref{sect:bdry_growth_reg}-\ref{sect:special_soln}-\ref{sect:proof_BH} to slit domains and prove Theorem \ref{thm:slit_main}. Furthermore, in Section \ref{sect:FB}, we prove our free boundary regularity results, Corollaries \ref{cor:classic_parabolic_obstacle} and \ref{cor:parabolic_signorini}.

In Section \ref{sect:elliptic}, we explain how to apply the ideas of the paper to elliptic equations, leading to the elliptic version of our main theorem, Theorem \ref{thm:elliptic_main}. Finally, in Section \ref{sect:hopf}, we prove Corollary \ref{cor:hopf_parabolic_rhs}.

\vspace{-0.35cm}
\section{Preliminaries}\label{sect:previ}
\subsection{Setting}

Throughout the paper, given $x \in \R^n$, we will denote $x' = (x_1,\ldots,x_{n-1})$. $B_r(x)$ will denote the ball of radius $r$ of $\R^n$, centered at $x$, and $B'_r(x')$ will be the one of $\R^{n-1}$, with center at $x'$. We also introduce the parabolic cylinders
$$Q_r(x,t) := B_r'(x')\times(x_n-r,x_n+r)\times(t-r^2,t) \subset \R^{n+1},$$
and we will write $Q_r := Q_r(0,0)$.

We define the parabolic Lipschitz ($\alpha = 1$) and Hölder ($\alpha \in (0,1)$) seminorms of a function $g : \Omega \subset \R^n\times\R \to \R$ as follows:
$$[g]_{C^{0,\alpha}_p(\Omega)} = \sup\frac{|g(x,t)-g(y,s)|}{(|x-y|+|t-s|^{1/2})^\alpha},$$
where the supremum is taken over all pairs of different points $(x,t) \neq (y,s)$ in $\Omega$, and we define correspondingly the parabolic Lipschitz and Hölder norms
$$\|g\|_{C^{0,\alpha}_p(\Omega)} := \|g\|_{L^\infty(\Omega)} + [g]_{C^{0,\alpha}_p(\Omega)}.$$
We will omit the domain where there is no case for confusion.

\vspace{0.1cm}

We will denote by
$$\mathcal{M^-}(D^2u) := \inf\limits_{\lambda I \leq A \leq \Lambda I}\operatorname{Tr}(AD^2u), \quad \mathcal{M^+}(D^2u) := \sup\limits_{\lambda I \leq A \leq \Lambda I}\operatorname{Tr}(AD^2u)$$
the Pucci extremal operators; see \cite{CC95} or \cite{FR22} for their properties.

In our work, we will consider the following notion of solutions.
\begin{defn}\label{defn:strong_solns}
Let $f \in L^{n+1}_{\mathrm{loc}}$. We say $u \in C^0\cap L^{n+1}_{\mathrm{loc}}$ is a \textit{strong} solution to
$$u_t - \L u = f$$
if $D^2_xu, \p_tu \in L^{n+1}_{\mathrm{loc}}$, and the equation holds almost everywhere. The condition $u \in C^0$ is redundant, but we write it to fix ideas, see \cite[Theorem 10.4]{BIN78}.
\end{defn}

\vspace{0.1cm}

We will consider parabolic Lipschitz domains of the following form.
\begin{defn}\label{defn:lipschitz_domain}
We say $\Omega$ is a parabolic Lipschitz domain in $Q_R$ with Lipschitz constant $L$ if $\Omega$ is the epigraph of a parabolic Lipschitz function ${\Gamma : B_R'\times[-R^2,0] \to \R}$, with $\Gamma(0,0) = 0$:
$$\Omega = \big\{(x,t) \in Q_R \ | \ x_n > \Gamma(x',t) \big\}, \quad \|\Gamma\|_{C_p^{0,1}} \leq L.$$
In this context, we will denote the lateral boundary
$$\p_\Gamma\Omega := \big\{(x,t) \in Q_R \ | \ x_n = \Gamma(x',t) \big\},$$
and the parabolic boundary
$$\p_p\Omega := \p_\Gamma\Omega\cup\big(\overline{\Omega}\cap\p Q_R\cap\{t < 0\}\big).$$
\end{defn}

\subsection{Technical tools} Let us start with the parabolic interior Harnack.
\vspace{0.1cm}

\begin{thm}[\protect{\cite[Theorem 4.18]{Wan92a}}]\label{thm:interior_Harnack}
Let $\L$ be a non-divergence form operator as in \eqref{eq:non-divergence_operator}, and let $u$ be a strong solution to $u_t - \L u = 0$ in $Q_r$, with $u \geq 0$.

Then,
$$\sup\limits_{Q_{r/2}\left(0,-\frac{r}{2}\right)}u \leq C\inf\limits_{Q_{r/2}}u,$$
where $C$ depends only on the dimension and the ellipticity constants.
\end{thm}

Then, we recall the Alexandrov-Bakelman-Pucci-Krylov-Tso estimate; see \cite{Kry76,Tso85}.

\begin{thm}\label{thm:ABPKT}
Let $\L$ be a non-divergence form operator as in \eqref{eq:non-divergence_operator} and let $u$ be a strong solution to $u_t - \L u = f$ in $Q_r$, with $f \in L^{n+1}(Q_r)$.

Then,
$$\sup\limits_{Q_r} u \leq \sup\limits_{\p_p Q_r} u^+ + Cr^{n/(n+1)}\|f\|_{L^{n+1}(Q_r)},$$
where $C$ depends only on the dimension and the ellipticity constants.
\end{thm}

We will state together the interior regularity estimates for the heat equation and for equations with bounded measurable coefficients.

\begin{thm}\label{thm:interior_Calpha}
Let $\L$ be a non-divergence form operator as in \eqref{eq:non-divergence_operator}, let $\alpha \in (0,\alpha_0)$, and let $u$ be a strong solution to $u_t - \L u = f$ in $Q_r$, with $f \in L^{n+1}(Q_r)$.

Then,
$$[u]_{C^{0,\alpha}_p(Q_{r/2})} \leq C(r^{-\alpha}\|u\|_{L^\infty(Q_r)} + r^{n/(n+1)-\alpha}\|f\|_{L^{n+1}(Q_r)}),$$
where $C$ depends only on $\alpha$, the dimension and the ellipticity constants.
\end{thm}

\begin{proof}
    Let us write the proof for $r = 1$, as the dependence on $r$ follows by scaling.

    First, when $\L$ is the Laplacian, we apply the Calderón-Zygmund estimates in \cite[Theorem 6]{Wan03} to obtain:
    $$\|D^2u\|_{L^{n+1}(Q_{1/2})} + \|u_t\|_{L^{n+1}(Q_{1/2})} \leq C(\|u\|_{L^\infty(Q_1)} + \|f\|_{L^{n+1}(Q_1)}).$$
    The conclusion follows via the parabolic Sobolev embedding in \cite[Theorem 1.4.1 (ii)]{WYW06}.

    In the case of operators with coefficients, this result is \cite[Lemma 5.1]{CKS00}.\footnote{The result is actually stated for $L^p$-viscosity solutions, but it automatically extends to strong solutions. See the introduction of \cite{CKS00}.}
\end{proof}

We also need the following covering result  (cf. \cite[Lemma B.2]{Kuk22}).

\begin{lem}\label{lem:teo_b2}
Let $\alpha \in (0,1)$, and let $\Omega$ be a parabolic Lipschitz domain in $Q_1$ with Lipschitz constant $\frac{1}{8}$, in the sense of Definition \ref{defn:lipschitz_domain}.

Assume that $u : \Omega\cap Q_1 \to \R$ satisfies
$$[u]_{C^{0,\alpha}_p(Q_r(x_0,t_0))} \leq C_0$$
whenever $Q_{2r}(x_0,t_0) \subset \Omega\cap Q_1$. Then, for any $0 < \sigma < 1$,
$$[u]_{C^{0,\alpha}_p(\Omega\cap Q_\sigma)} \leq CC_0.$$
The constant $C$ depends only on $\alpha$ and $\sigma$.
\end{lem}

The proof follows by a standard covering argument (see \cite[Appendix A]{FR22}). In addition, the following convergence result for limits of solutions will be useful to make contradiction-compactness arguments.
\begin{prop}[\protect{\cite[Theorem 6.1]{CKS00}}]\label{prop:viscosity_limit}
Let $f_k \in L^{n+1}(\Omega)$ and let $u_k$ be solutions to
$$(\p_t - \Mplus)u_k \leq f_k \leq (\p_t - \Mminus)u_k \quad \text{in} \ \Omega,$$
such that $u_k \rightarrow u_0$ locally uniformly, and that for every $K \subset\subset \Omega$, $\|f_k\|_{L^{n+1}(K)} \rightarrow 0$.

Then,
$$(\p_t - \Mplus)u_0 \leq 0 \leq (\p_t - \Mminus)u_0$$
in the viscosity sense.
\end{prop}

Finally, a key step in our proofs is to classify certain solutions in special domains. We will do so with the following Liouville theorem for the half-space, that follows from standard arguments from the boundary regularity estimate \mbox{\cite[Theorem 2.1]{Wan92b}}.
\begin{thm}\label{thm:liouville}
Let $\L$ be a non-divergence form operator as in \eqref{eq:non-divergence_operator}, let $\alpha \in (0,\alpha_0)$ and let $u$ be a solution to
$$\left\{\begin{array}{rccclll}
(\p_t - \Mplus)u & \leq & 0 & \leq & (\p_t - \Mminus)u & \text{in} & \{x_n > 0\}\\
u & = & 0 &&& \text{in} & \{x_n \leq 0\}
\end{array}\right.$$
with the growth condition
$$\|u\|_{L^\infty(Q_R)} \leq C(1+R^{1+\alpha}), \ \forall R \geq 1.$$
Then, $u = k(x_n)_+$ for some $k \in \R$.
\end{thm}

\begin{proof}
First, observe that, for every $r \geq 1$,
$$u_r(x,t) := \frac{u(rx,r^2t)}{r^{1+\alpha}}$$
also satisfies the hypotheses.

Now, from \cite[Theorem 2.1]{Wan92b} and interior regularity estimates (in the case of the heat equation, by $C^2$ boundary regularity estimates) it follows that
$$[u]_{C^{1,\alpha_0}(Q_r)} = r^{\alpha-\alpha_0}[u_r]_{C^{1,\alpha_0}(Q_1)} \leq Cr^{\alpha-\alpha_0}\|u_r\|_{L^\infty(Q_2)} \leq Cr^{\alpha-\alpha_0},$$
and then letting $r \rightarrow \infty$ we deduce that
$$[u]_{C^{1,\alpha_0}(\R^{n+1})} = 0,$$
and therefore $u$ is a linear function. From the boundary conditions, we deduce that $u = k(x_n)_+$.
\end{proof}

\section{Boundary growth and regularity estimates}\label{sect:bdry_growth_reg}
The main goal of this section is to prove that for sufficiently flat domains, solutions that vanish on the boundary of the domain grow like $d^{1\pm\varepsilon}$ and are $C^{0,\gamma}_p$.

\begin{prop}\label{prop:bdry_Calpha}
Let $\gamma \in (0,\alpha_0)$. There exists $L_0 > 0$, depending only on $\gamma$, the dimension and the ellipticity constants, such that the following holds.

Let $\Omega$ be a parabolic Lipschitz domain in $Q_1$ as in Definition \ref{defn:lipschitz_domain} with Lipschitz constant $L \leq L_0$. Let $d(x',x_n,t) = x_n - \Gamma(x',t)$, and let $u\in C(Q_1)$ be a viscosity solution to
\[ u_t - \Mminus u \ge -K_0(d^{\gamma-2}+f) \quad \mbox{and}\quad u_t - \Mplus u \le K_0(d^{\gamma-2}+f) \quad \mbox{in }\Omega\cap Q_1,\]
\[ u= 0 \quad \mbox{in }\p_\Gamma\Omega\cap Q_1.\]
Assume that $\|d^{-(\gamma-n/(n+1))_+}f\|_{L^{n+1}(Q_1)} \leq 1$. Then,
\[\|u\|_{C^{0,\gamma}_p(Q_{1/2})}\leq C(\|u\|_{L^\infty(Q_1)} + K_0).\]
The constant $C$ depends only on $n$, $\gamma$ and the ellipticity constants.
\end{prop}

We start by introducing the regularized distance \cite{Lie85, Lie96}, a technical tool that will be useful to construct barriers.

\begin{lem}\label{lem:regularized_distance}
Let $\Omega$ be a parabolic Lipschitz domain with Lipschitz constant $L \leq 1$, in the sense of Definition \ref{defn:lipschitz_domain}. Then, there exists a function $d : \Omega \cap Q_1 \to \R$ satisfying the following:
\begin{gather*}
    \frac{1}{2}(x_n - \Gamma(x',t)) \leq d \leq \frac{3}{2}(x_n - \Gamma(x',t))\\
    \frac{2}{3} \leq |\nabla_x d| \leq C_1\\
    |\p_td| + |D^2_xd| \leq C_2Ld^{-1},
\end{gather*}
where $C_1$ and $C_2$ depend only on the dimension.
\end{lem}

We will sketch the proof in Appendix \ref{sect:app}.

As a direct consequence, we obtain the following scaling property.
\begin{lem}\label{lem:dist_scaling}
Let $\Omega$ be a parabolic Lipschitz domain in $Q_R$, in the sense of Definition \ref{defn:lipschitz_domain}. Let $r \in (0,1)$ and let
$$\tilde\Omega := \{(x,t) \in Q_R \ | \ (rx,r^2t) \in \Omega\}.$$
Let $d$ and $\tilde d$  be the regularized distances in $\Omega$ and $\tilde\Omega$, respectively. Then,
$$\frac{1}{3}\tilde d(x,t) \leq \frac{1}{r}d(rx,r^2t) \leq 3\tilde d(x,t).$$
\end{lem}

\begin{proof} It follows from Lemma \ref{lem:regularized_distance}.
\end{proof}

The following barriers are constructed in a very similar way to \cite[Lemmas 3.2 and 3.3]{Kuk22}. We start with a supersolution.

\begin{lem}\label{lem:barrier_super}
Let $\varepsilon \in (0,1)$. There exist sufficiently small $\eta > 0$ and sufficiently large $K > 0$, only depending on the dimension, $\varepsilon$ and the ellipticity constants, such that the following holds. 

Let $u$ be a solution to
\begin{equation*}
\left\{
\begin{array}{rclll}
u_t - \Mplus u & \leq & \eta d^{-1 -\varepsilon} & \text{in} & \Omega\\
u & \leq & 1 & \text{on} & \p_p\Omega\\
u & \leq & 0 & \text{on} & \p_\Gamma\Omega,
\end{array}
\right.
\end{equation*}
where $\Omega$ is a parabolic Lipschitz domain in $Q_1$ in the sense of Definition \ref{defn:lipschitz_domain} with Lipschitz constant $\eta$. Then,
$$u \leq Kd^{1-\varepsilon}-t+|x'|^2 \quad \text{in} \ \Omega,$$
where $d$ is the regularized distance introduced in Lemma \ref{lem:regularized_distance}, and
$$u(re_n,0) \leq Kr^{1-\varepsilon}, \quad \forall r \in (0,1).$$
\end{lem}

\begin{proof}
We will use the comparison principle with a supersolution that has the desired growth.

Let
$$\varphi = Kd^{1-\varepsilon}-t+|x'|^2,$$
where $d$ is the regularized distance introduced in Lemma \ref{lem:regularized_distance}, and $K > 0$ is a large constant to be chosen later. 

Notice that $\varphi \geq 0$ on $\p_\Gamma\Omega$, and that $\varphi \geq 1$ on $\{t=-1\}$ and $\{|x'|=1\}$. We can also check
$$\varphi \geq K\left(\frac{x_n - \Gamma(x',t)}{2}\right)^{1-\varepsilon} - t + |x'|^2 > \frac{K(1-2\eta)}{2} \geq 1 \quad \text{on} \quad \{x_n = 1\}\cap\overline{\Omega}.$$

On the other hand, using the estimates in Lemma \ref{lem:regularized_distance},
\begin{align*}
    (\p_t-\Mplus)d^{1-\varepsilon} &= (1-\varepsilon)d^{-\varepsilon}\p_td - \sup\limits_{\lambda I \leq A \leq \Lambda I}\operatorname{Tr}(AD^2_xd^{1-\varepsilon})\\
    &= (1-\varepsilon)\left(d^{-\varepsilon}\p_td - \sup\limits_{\lambda I \leq A \leq \Lambda I}\left[\sum\limits_{i=1}^n\sum\limits_{j=1}^na_{ij}(d^{-\varepsilon}\p^2_{ij}d - \varepsilon d^{-1-\varepsilon}\p_id\p_jd)\right]\right)\\
    &=(1-\varepsilon)d^{-\varepsilon}(\p_t - \Mplus)d + \varepsilon(1-\varepsilon)d^{-1-\varepsilon}\inf\limits_{\lambda I \leq A \leq \Lambda I}\nabla_xd^\top A\nabla_xd\\
    &\geq (1-\varepsilon)\left(-C\eta + \frac{4\lambda\varepsilon}{9}\right)d^{-1-\varepsilon},
\end{align*}
where we omitted the dependence of $A$ and $a_{ij}$ on $x$ for formatting. Then,
$$(\p_t-\Mplus)\varphi \geq K(1-\varepsilon)\left(-C\eta + \frac{4\lambda\varepsilon}{9}\right)d^{-1-\varepsilon} - 1 - 2(n-1)\Lambda \geq \eta d^{-1-\varepsilon},$$
where we chose $\eta$ small, $K$ large, and used that $d < \frac{3}{2}$.

Therefore, applying the comparison principle to $u$ and $\varphi$, we obtain that $u \leq \varphi$ in $\Omega$, and in particular
$$u(re_n,0) \leq \varphi(re_n,0) = Kd(re_n,0)^{1-\varepsilon}\leq K' r^{1-\varepsilon}.$$
\end{proof}

Analogously, we can consider a subsolution.

\begin{lem}\label{lem:barrier_sub}
Let $\varepsilon \in (0,1)$. There exist sufficiently small $\eta, k, r_0 > 0$, only depending on the dimension, $\varepsilon$ and the ellipticity constants, such that the following holds. 

Let $u \geq 0$ be a solution to
\begin{equation*}
\left\{
\begin{array}{rclll}
u_t - \Mminus u & \geq & -\eta d^{\varepsilon - 1} & \text{in} & \Omega\\
u & \geq & 1 & \text{on} & \p_{\mathrm{up}} \Omega\\
\end{array}
\right.
\end{equation*}
where $\Omega$ is a parabolic Lipschitz domain in $Q_1$ in the sense of Definition \ref{defn:lipschitz_domain} with Lipschitz constant $\eta$, and
$$\p_{\mathrm{up}} \Omega := \{x_n = 1\} \cap \overline{\Omega}$$
is the top part of the boundary. Then,
$$u \geq kd^{1+\varepsilon}+t-|x'|^2 \quad \text{in} \ \Omega\cap Q_{r_0},$$
where $d$ is the regularized distance introduced in Lemma \ref{lem:regularized_distance}, and
$$u(re_n,0) \geq k r^{1+\varepsilon}, \quad \forall r \in (0,1).$$
\end{lem}

\begin{proof}
First, let
\begin{align*}
    \Omega^{(1)} &:= \left\{(x,t) \in \Omega : |x'| < \frac{1}{2}, x_n < r_0, -\frac{1}{2} < t\right\},\\
    \Omega^{(2)} &:= \left\{(x,t) \in \Omega : |x'| < \frac{1}{2}, x_n > r_0, -\frac{1}{2} < t\right\},
\end{align*}
and
$$\Omega^{(3)} := Q_1 \setminus \left\{x_n \leq \frac{r_0}{2}\right\}
,$$
with $r_0 \in (0,1)$ to be chosen later. Let also $$\p_{\mathrm{up}}\Omega^{(1)} := \{x_n = r_0\}\cap\overline{\Omega^{(1)}}.$$

Now, let $v$ be the solution to
\begin{equation*}
\left\{
\begin{array}{rclll}
v_t - \Mminus v & = & 0 & \text{in} & \Omega^{(3)}\\
v & = & 1 & \text{on} & \p_{\mathrm{up}} \Omega\\
v & = & 0 & \text{in} & \p_p\Omega^{(3)} \setminus \p_{\mathrm{up}} \Omega.
\end{array}
\right.
\end{equation*}

By the strong maximum principle, $\min\limits_{\overline{\Omega^{(2)}}}v \geq 2c_0 > 0$. On the other hand, by the comparison principle,
$$w := u + 2\eta r_0^{\varepsilon-1}(t+1) \geq v,$$
because $w \geq u$ on $\p_p\Omega^{(3)}$ and
$$(\p_t - \Mminus)w = (\p_t - \Mminus)u + 2\eta r_0^{\varepsilon-1} \geq \eta(-d^{\varepsilon-1}+2r_0^{\varepsilon-1}) \geq 0.$$
Hence, noting that $\p_{\mathrm{up}}\Omega^{(1)} \subset \overline{\Omega^{(2)}}$,
$$\min\limits_{\p_{\mathrm{up}}\Omega^{(1)}}u \geq \min\limits_{\overline{\Omega^{(2)}}}v - 2r_0^{\varepsilon-1}\eta \geq c_0 > 0,$$
and $\min\limits_{r \in [r_0,1]}u(re_n,0) \geq c_0$ as well, choosing $\eta$ small enough.

Now, let
$$\varphi(x,t) = kd^{1+\varepsilon} + t - |x'|^2,$$
where $d$ is the regularized distance introduced in Lemma \ref{lem:regularized_distance}, for the domain $\Omega$ (that coincides with the regularized distance for the domain $\Omega^{(1)}$ because the two domains lie above the same graph), and $k = \min\{c_0/(4r_0),1/32\}$.

Then, $\varphi \leq u$ on the parabolic boundary of $\Omega^{(1)}$. Indeed, since $d \leq 2r_0$ on $\p_{\mathrm{up}}\Omega^{(1)}$, $\varphi < c_0$ on $\p_{\mathrm{up}}\Omega^{(1)}$. When $t = -1/2$, since $d \leq 2$,
$$\varphi \leq \frac{1}{32}d^{1+\varepsilon}-\frac{1}{2} \leq \frac{1}{8} - \frac{1}{2} < 0,$$
and if $|x'| = 1/2$,
$$\varphi \leq \frac{1}{8}-\frac{1}{4} < 0.$$
Finally, $\varphi \leq 0$ on $\p_\Gamma\Omega$.

Hence, by an analogous computation to the proof of Lemma \ref{lem:barrier_super}, we obtain
$$\varphi_t - \Mminus\varphi \leq (1+\varepsilon)k\left(C\eta - \frac{4\lambda\varepsilon}{9}\right)d^{\varepsilon-1} + 1 + 2(n-1)\Lambda \leq -\eta d^{\varepsilon-1},$$
using that $\eta$ and $d < 2r_0$ can be chosen arbitrarily small.

Therefore, applying the comparison principle to $u$ and $\varphi$, we obtain that $u \geq \varphi$ in $\Omega^{(1)}$, and in particular in $\Omega\cap Q_{r_0}$, and also that for all $r \in (0,r_0)$
$$u(re_n,0) \geq kd^{1+\varepsilon} \geq k'r^{1+\varepsilon}.$$

On the other hand, if $r > r_0$, we can use directly that $u(re_n,0) \geq c_0$.
\end{proof}

The growth upper bound still holds when we consider a (small) right-hand side in $L^{n+1}$, thanks to Theorem \ref{thm:ABPKT}.

\begin{lem}\label{lem:barrier_super_abp}
Let $\varepsilon \in (0,1)$. There exist sufficiently small $\eta > 0$ and sufficiently large $K > 0$, only depending on the dimension, $\varepsilon$, and the ellipticity constants, such that the following holds. 

Let $d(x',x_n,t) = x_n - \Gamma(x',t)$, and let $u$ be a solution to
\begin{equation*}
\left\{
\begin{array}{rclll}
u_t - \Mplus u & \leq & \eta d^{-1 -\varepsilon} + f & \text{in} & \Omega\\
u & \leq & 1 & \text{on} & \p_p\Omega\\
u & \leq & 0 & \text{on} & \p_\Gamma\Omega\\
\end{array}
\right.
\end{equation*}
where $\|d^{-(1/(n+1)-\varepsilon)_+}f\|_{L^{n+1}(\Omega)} \leq \eta$, and $\Omega$ is a parabolic Lipschitz domain in $Q_1$ in the sense of Definition \ref{defn:lipschitz_domain} with Lipschitz constant $\eta$. Then,
$$u(re_n,0) \leq K r^{1-\varepsilon}, \quad \forall r \in (0,1).$$
\end{lem}

\begin{obs}
    Thanks to Lemma \ref{lem:regularized_distance}, we can interchange the regularized distance with $x_n - \Gamma(x',t)$ up to a constant. We will do so in the following.
\end{obs}

\begin{proof}
We will iterate Lemma \ref{lem:barrier_super} combined with Theorem \ref{thm:ABPKT}. We define the rescaled functions
$$u_j(x,t) := \frac{u(\rho^jx,\rho^{2j}t)}{\rho^{j(1-\varepsilon)}},$$
with $\rho > 0$ to be chosen later.

Now,
$$(\p_t-\Mplus)u_j \leq \rho^{j(1+\varepsilon)}(\eta(\rho^jd)^{-1-\varepsilon}+\tilde f_j) \leq \eta d^{-1-\varepsilon}+\rho^{j(1+\varepsilon)}\tilde f_j,$$
with $\tilde f_j(x,t) := f(\rho^jx,\rho^{2j}t)$.

Let $a_j := \|u_j\|_{L^\infty(\tilde\Omega)}$. By Lemma \ref{lem:barrier_super} and Theorem \ref{thm:ABPKT}, $a_0 \leq C_0$. We will show by induction that $a_j \leq C_0$ for all $j \geq 0$.

Again by Lemma \ref{lem:barrier_super} (with $\varepsilon/2$) and Theorem \ref{thm:ABPKT} applied to $u_j$,
\begin{align*}
    a_{j+1}\rho^{1-\varepsilon} &\leq K\rho^{1-\varepsilon/2}a_j+2\rho^2a_j+C\rho^{j(1+\varepsilon)}\|\tilde f_j\|_{L^{n+1}(\tilde\Omega)}\\
    &\leq K\rho^{1-\varepsilon/2}C_0+2\rho^2C_0+C\rho^{j(\varepsilon-1/(n+1))}\|f\|_{L^{n+1}(\rho^j\Omega)}\\
    &\leq \rho^{1-\varepsilon}C_0/2+2\rho^2C_0+C\|d^{-(1/(n+1)-\varepsilon)_+}f\|_{L^{n+1}(\Omega)} \leq \rho^{1-\varepsilon}C_0,
\end{align*}
choosing adequately small $\rho$ and $\eta$.

The conclusion follows by observing that given $r \in (0,1)$, for all $j$ such that $\rho^j \geq r$, $u(re_n,0) \leq a_j\rho^{j(1-\varepsilon)}$, and hence $u(re_n,0) \leq C_0(r/\rho)^{1-\varepsilon}$.
\end{proof}

Conversely, we can also add a more general right-hand side to the equation for the subsolution.

\begin{lem}\label{lem:barrier_sub_abp}
Let $\varepsilon \in (0,1)$. There exist sufficiently small $\eta, k > 0$, only depending on the dimension, $\varepsilon$, and the ellipticity constants, such that the following holds. 

Let $d(x',x_n,t) = x_n - \Gamma(x',t)$, and let $u \geq 0$ be a solution to
\begin{equation*}
\left\{
\begin{array}{rclll}
u_t - \Mminus u & \geq & -\eta d^{\varepsilon - 1} + f & \text{in} & \Omega\\
u & \geq & 1 & \text{on} & \p_{\mathrm{up}} \Omega\\
\end{array}
\right.
\end{equation*}
where $\|d^{-\varepsilon-1/(n+1)}f\|_{L^{n+1}(\Omega)} \leq \eta$, and $\Omega$ and $\p_{\mathrm{up}}\Omega$ are defined as in Lemma \ref{lem:barrier_sub}. Then,
$$u(re_n,0) \geq k r^{1+\varepsilon}, \quad \forall r \in (0,1).$$
\end{lem}

\begin{proof}
We will use a similar strategy to the proof of Lemma \ref{lem:barrier_super_abp}. We define the rescaled functions
$$u_j(x,t) := \frac{u(\rho^jx,\rho^{2j}t)}{\rho^{j(1+\varepsilon)}},$$
with $\rho > 0$ to be chosen later.

Now,
$$(\p_t-\Mminus)u_j \geq \rho^{j(1-\varepsilon)}(-\eta(\rho^jd)^{-1+\varepsilon}+\tilde f_j) \geq -\eta d^{-1+\varepsilon}+\rho^{j(1-\varepsilon)}\tilde f_j,$$
with $\tilde f_j(x,t) := f(\rho^jx,\rho^{2j}t)$.

Let $a_j := \inf\limits_{\p_{\mathrm{up}} \Omega} u_j$. By hypothesis, $a_0 = 1$. We will show by induction that $a_j \geq 1$ for all $j \geq 0$.

Again by Lemma \ref{lem:barrier_sub} (with $\varepsilon/2$) and Theorem \ref{thm:ABPKT}, applied to $u_j$,
\begin{align*}
    a_{j+1}\rho^{1+\varepsilon} &\geq k\rho^{1+\varepsilon/2}a_j-2\rho^2-C\rho^{j(1-\varepsilon)}\|\tilde f_j\|_{L^{n+1}(\tilde\Omega)}\\
    &\geq k\rho^{1+\varepsilon/2}-2\rho^2-C\rho^{-j(\varepsilon+1/(n+1))}\|f\|_{L^{n+1}(\rho^j\Omega)}\\
    &\geq 2\rho^{1+\varepsilon}-2\rho^2-C\|d^{-\varepsilon-1/(n+1)}f\|_{L^{n+1}(\Omega)} \geq \rho^{1+\varepsilon},
\end{align*}
choosing adequately small $\rho$ and $\eta$.

Finally, by Lemma \ref{lem:barrier_sub} and Theorem \ref{thm:ABPKT},
$$u_j(re_n,0) \geq k a_j - C\eta \geq k' > 0,$$
for all $r \in (\rho,1)$, provided that $\eta$ is small enough. Hence, undoing the scaling, $u(re_n,0) \geq k'r^{1+\varepsilon}$, as we wanted to prove.
\end{proof}

Combining the previous estimates, we can now give the following.

\begin{prop}\label{prop:growth}
Let $\varepsilon \in (0,1)$ and let $\L$ be a non-divergence form operator as in \eqref{eq:non-divergence_operator}. There exists  sufficiently small $\eta > 0$, only depending on the dimension, $\varepsilon$ and the ellipticity constants, such that the following holds.

Let $\Omega$ be a parabolic Lipschitz domain in $Q_1$ with Lipschitz constant $\eta$ in the sense of Definition \ref{defn:lipschitz_domain}. Let $d(x',x_n,t) = x_n - \Gamma(x',t)$, and let $u$ be a solution to
$$\left\{\begin{array}{rclll}
u_t - \L u & = & f & \text{in} & \Omega,\\
u & = & 0 & \text{on} & \p_\Gamma\Omega.
\end{array}\right.$$
Assume that $\|u\|_{L^\infty(Q_1)} \leq 1$, and $f = g + h$, with
$$\|d^{\varepsilon+1}g\|_{L^\infty(Q_1)}+\|d^{-(1/(n+1) - \varepsilon)_+}h\|_{L^{n+1}(Q_1)} \leq \eta.$$
Then,
$$|u| \leq Cd^{1-\varepsilon} \quad \text{in} \quad \Omega \cap Q_{3/4}.$$
Moreover, if $u$ is nonnegative, $m = u\left(\frac{e_n}{2},-\frac{3}{4}\right) > 0$ and
$$\|d^{-\varepsilon+1}g\|_{L^\infty(Q_1)}+\|d^{-\varepsilon-1/(n+1)}h\|_{L^{n+1}(Q_1)} \leq \eta m,$$
then,
$$u \geq cmd^{1+\varepsilon} \quad \text{in} \quad \Omega \cap Q_{3/4}.$$
The constants $C$ and $c$ are positive and depend only on the dimension, $\varepsilon$, and the ellipticity constants.
\end{prop}

\begin{proof}
For the first estimate, let $(x_0,t_0) \in \p_\Gamma\Omega\cap Q_{3/4}$, and consider the function
$$v(x,t) := u\left(x_0+\frac{1}{4}x,t_0+\frac{1}{16}t\right).$$

Then, by Lemma \ref{lem:barrier_super_abp}, $v(re_n,0) \leq K r^{1-\varepsilon}$ for all $r \in (0,1)$. Since $d$ is comparable to $x_n - \Gamma(x',t)$, it follows that $u \leq Cd^{1-\varepsilon}$ in $\Omega\cap Q_{3/4}\cap\{x_n < 1/4\}$.

Finally, notice that
$$d \geq \frac{1}{2}(x_n - \Gamma(x',t)) \geq \frac{1}{2}\left(\frac{1}{4}-2\eta\right) > \frac{1}{9} \quad \text{in} \ \Omega\cap\{x_n \geq 1/4\},$$
and the conclusion follows adjusting $C$ if necessary.

For the second estimate, let $\Omega(x_0,t_0) := \Omega \cap Q_{1/8}(x_0,t_0)$ and notice that
$$\bigcup\limits_{(x_0,t_0) \in \p_\Gamma\Omega\cap Q_{3/4}} \p_{\mathrm{up}}\Omega(x_0,t_0) \subset E := \overline{B_{7/8}'}\times[1/8-2\eta,1/8+2\eta]\times[-37/64,0],$$
where $\p_{\mathrm{up}}\Omega(x_0,t_0) := \overline{\Omega(x_0,t_0)}\cap\{x_n = x_{0,n}+1/8\}$, analogously to Lemma \ref{lem:barrier_sub}.

Then, by the interior Harnack (Theorem \ref{thm:interior_Harnack}), $u \geq c_1m$ in $E$, and by an analogous reasoning to the upper bound with Lemma \ref{lem:barrier_sub_abp} instead of Lemma \ref{lem:barrier_super_abp}, the conclusion follows.
\end{proof}

We are finally able to prove our $C^{0,\gamma}$ boundary regularity result.

\begin{proof}[Proof of Proposition \ref{prop:bdry_Calpha}]
We may assume that $\|u\|_{L^\infty(Q_1)} \leq 1$ and $K_0 = \eta$ (with $\eta$ from Proposition \ref{prop:growth}) without loss of generality after dividing by a constant. Then, by Proposition \ref{prop:growth} and Lemma \ref{lem:regularized_distance},
$$|u| \leq K((x_n - \Gamma(x',t))^\gamma \quad \text{in} \ \Omega\cap Q_{3/4}.$$

Then, we will use interior estimates in combination with Lemma \ref{lem:teo_b2} to deduce the result.

Let $p = (y',y_n,s)$ and $\rho \in (0,\frac{1}{16})$ such that $Q_{2\rho}(p) \subset \Omega\cap Q_{5/8}$, and let
$$R := \max\left\{\rho,\frac{y_n-\Gamma(y',s)}{3}\right\}.$$
Note that $Q_{2R}(p) \subset \Omega$. We distinguish two cases:

\textit{Case 1.} $R \geq \frac{1}{16}$. Then, $Q_{1/8}(p) \subset \Omega\cap Q_{3/4}$, and for all $(x',x_n,t) \in Q_{1/8}(p)$,
$$x_n \geq \Gamma(x',t) + \frac{R}{2} \geq \Gamma(x',t) + \frac{1}{32}.$$
Hence,
\begin{align*}
    u_t - \Mminus u &\geq -d^{\gamma-2}+f \geq -2^{12}+f,\\
    u_t - \Mplus u &\leq d^{\gamma-2}+f \leq 2^{12}+f,
\end{align*}
which together with the fact that $\|u\|_{L^\infty(Q_1)} \leq 1$, Theorem \ref{thm:interior_Calpha}, and a covering argument, gives
$$[u]_{C^{0,\gamma}_p(Q_\rho(p))} \leq [u]_{C^{0,\gamma}_p(Q_{1/16}(p))} \leq C.$$

\textit{Case 2.} $R < \frac{1}{16}$. Notice that if $\rho < R$, $y_n - \Gamma(y',s) = 3R$, and if $\rho = R$, using that $Q_{2\rho}(p) \subset \Omega$, $y_n - \Gamma(y',s) \geq 2\rho = 2R$. In either case,
$$2R \leq y_n - \Gamma(y',s) \leq 3R.$$

Now, for all $(x',x_n,t) \in Q_{3R/2}(p)$,
\begin{align*}
    x_n - \Gamma(x',t) &\geq y_n - \Gamma(y',s) - \frac{3}{2}R - |\Gamma(x',t)-\Gamma(y',s)| \geq \frac{R}{4},\\
    x_n - \Gamma(x',t) &\leq y_n - \Gamma(y',s) + \frac{3}{2}R + |\Gamma(x',t)-\Gamma(y',s)| \leq 5R
\end{align*}
using the parabolic Lipschitz character of $\Gamma$ and that $L_0 \leq \frac{1}{8}$.

Therefore, 
\begin{align*}
    u_t - \Mminus u &\geq -CR^{\gamma-2}+f,\\
    u_t - \Mplus u &\leq CR^{\gamma-2}+f,
\end{align*}
and $\|u\|_{L^\infty(Q_{3R/2})} \leq K(5R)^\gamma$, which combined with Theorem \ref{thm:interior_Calpha} gives
\begin{align*}
    [u]_{C^{0,\gamma}_p(Q_R(p))} &\lesssim R^{-\gamma}(5R)^\gamma + R^{n/(n+1)-\gamma}\|R^{\gamma-2}+f\|_{L^{n+1}(Q_{3R/2}(p))}\\
    &\lesssim 1 + R^{n/(n+1)-\gamma}|Q_{3R/2}|^{1/(n+1)}R^{\gamma-2}+R^{n/(n+1)-\gamma}\|f\|_{L^{n+1}(Q_{3R/2}(p))}\\
    &\lesssim 1 + 1 + \|d^{-(\gamma - n/(n+1))_+}f\|_{L^{n+1}(Q_{3R/2}(p))} \lesssim 1.
\end{align*}

The conclusion follows by Lemma \ref{lem:teo_b2}.
\end{proof}

\section{The near-linear solution}\label{sect:special_soln}

Our goal now is to find a special solution satisfying the following.
\begin{prop}\label{prop:special}
Let $\varepsilon \in (0,\alpha_0)$, and let $\L$ be a non-divergence form operator as in \eqref{eq:non-divergence_operator}. Then, there exists $\delta > 0$, only depending on $\varepsilon$, the dimension and the ellipticity constants, such that the following holds.

Let $\Omega$ be a parabolic Lipschitz domain in $Q_1$ in the sense of Definition \ref{defn:lipschitz_domain} with Lipschitz constant $\delta$, and let $d(x',x_n,t) = x_n - \Gamma(x',t)$. Then, there exists $\varphi: \Omega \to \R$ such that
$$\left\{\begin{array}{rclll}
     \varphi_t - \L\varphi & = & 0 & \text{in} & \Omega\\
     \varphi & = & 0 & \text{on} & \p_\Gamma\Omega,
\end{array}\right.$$
$\varphi \geq 0$, $\|\varphi\|_{L^\infty(Q_1)} = 1$,
$$\frac{1}{24}d^{1+\varepsilon} \leq \varphi \leq 64d^{1-\varepsilon},$$
and for all $0 < r_1 < r_2 \leq 1$,
$$\frac{\sup\limits_{Q_{r_1}}\varphi}{\sup\limits_{Q_{r_2}}\varphi} \geq \frac{1}{8}\left(\frac{r_1}{r_2}\right)^{1+\varepsilon}.$$
\end{prop}

We start by constructing solutions with a controlled growth.
\begin{lem}\label{lem:special_exists}
Let $\varepsilon \in (0,1)$, and let $\L$ be a non-divergence form operator as in \eqref{eq:non-divergence_operator}. There exists $\delta_1 \in (0,\varepsilon)$, only depending on the dimension, $\varepsilon$ and the ellipticity constants, such that the following holds.

Let $R = 2^{1/\varepsilon}$, let $\Omega$ be a parabolic Lipschitz domain in $Q_R$ in the sense of Definition \ref{defn:lipschitz_domain} with Lipschitz constant $\delta_1$, and let $d = x_n - \Gamma(x',t)$. Then, there exists $\varphi: \Omega \to \R$ such that
$$\left\{\begin{array}{rclll}
     \varphi_t - \L\varphi & = & 0 & \text{in} & \Omega\\
     \varphi & = & 0 & \text{on} & \p_\Gamma\Omega,
\end{array}\right.$$
$\varphi \geq 0$, $\|\varphi\|_{L^\infty(Q_1)} = 1$, and
$$\frac{1}{24}d^{1+\varepsilon} \leq \varphi \leq 64d^{1-\varepsilon} \quad \text{in} \ Q_R.$$
In particular, $\|\varphi\|_{L^\infty(Q_r)} \leq 128r^{1-\varepsilon}$ for all $r \in [1,R]$.
\end{lem}

\begin{proof}
First, by the same computations in Lemmas \ref{lem:barrier_super} and \ref{lem:barrier_sub},
\begin{align*}
    (\p_t-\L)d^{1-\varepsilon} &\geq (1-\varepsilon)d^{-1-\varepsilon}(-C\delta_1+C'\varepsilon) \geq 0\\
    (\p_t-\L)d^{1+\varepsilon} &\leq (1+\varepsilon)d^{-1+\varepsilon}(C\delta_1-C'\varepsilon) \leq 0,
\end{align*}
provided that $\delta_1$ is small enough. Assume without loss of generality that ${\delta_1 \in (0,1/6)}$. Then, since

$$d(x,t) \leq \frac{3}{2}(x_n - \Gamma(x',t)) \leq \frac{3}{2}(R+2\delta_1 R) \leq 2R,$$
it follows that
$$d^{1+\varepsilon} \leq (2R)^\varepsilon d \leq (2R)^{2\varepsilon}d^{1-\varepsilon}.$$

Now, let $\tilde\varphi$ be the solution to
$$\left\{\begin{array}{rclll}
\tilde\varphi_t-\L\tilde\varphi & = & 0 & \text{in} & \Omega\\
\tilde\varphi & = & (2R)^\varepsilon d & \text{on} & \p_p\Omega.
\end{array}\right.$$
By the comparison principle, it follows that
$$d^{1+\varepsilon} \leq \tilde\varphi \leq (2R)^{2\varepsilon}d^{1-\varepsilon}.$$
Then, by Lemma \ref{lem:regularized_distance},
\begin{align*}
    \|\tilde\varphi\|_{L^\infty(Q_1)} &\leq (2R)^{2\varepsilon}\|d^{1-\varepsilon}\|_{L^\infty(Q_1)} \leq 4R^{2\varepsilon}\frac{3}{2} = 24,\\
    \|\tilde\varphi\|_{L^\infty(Q_1)} &\geq \|d^{1+\varepsilon}\|_{L^\infty(Q_1)} \geq \frac{1}{4}.
\end{align*}

Let now

$$\varphi := \frac{\tilde\varphi}{\|\tilde\varphi\|_{L^\infty(Q_1)}}.$$

The first conclusion follows from the previous estimate. For the second one, notice that for $r \geq 1$,

$$\|\varphi\|_{L^\infty(Q_r)} \leq 64\|d^{1-\varepsilon}\|_{L^\infty(Q_r)} \leq 64\left(\frac{3}{2}r(1+2\delta_1)\right)^{1-\varepsilon} \leq 128r^{1-\varepsilon}.$$
\end{proof}

This special solutions satisfy the following estimate.
\begin{lem}\label{lem:special_1/2}
Let $\varepsilon \in (0,\alpha_0)$ and let $\L$ be a non-divergence form operator as in \eqref{eq:non-divergence_operator}. There exists an integer $n_0 > 1/\varepsilon$, only depending on $\varepsilon$, the dimension and the ellipticity constants, such that the following holds.

Let $R_0 = 2^{n_0}$, let $\Omega$ be a parabolic Lipschitz domain in $Q_{R_0}$ in the sense of Definition \ref{defn:lipschitz_domain} with Lipschitz constant $1/n_0$, and let $d = x_n - \Gamma(x',t)$. Let $\varphi : \Omega \to \R$ satisfy the following properties:
$$\left\{\begin{array}{rcll}
\varphi_t - \L\varphi & = & 0 & \text{in} \ \Omega\\
\varphi & = & 0 & \text{on} \ \p_\Gamma\Omega\\
\varphi & \geq & 0 &\\
\|\varphi\|_{L^\infty(Q_{2^k})} & \leq & 128\cdot 2^{k(1+1/n_0+\varepsilon)} & \forall \ k \in \{0,\ldots,n_0\} \\
\|\varphi\|_{L^\infty(Q_1)} \ & = & 1 &
\end{array}\right.$$

Then,
$$\sup\limits_{Q_{1/2}}\varphi \geq \left(\frac{1}{2}\right)^{1+1/n_0+\varepsilon}.$$
\end{lem}

\begin{proof}
Let us proceed by contradiction: assume there does not exist $n_0$ satisfying the conclusion. Then, by Lemma \ref{lem:special_exists} with $\varepsilon = 1/n_0$, there exist $n_k \uparrow \infty$, $\L_k$ non-divergence form operators, $\Omega_k$ parabolic Lipschitz domains in $Q_{R_k}$ (with Lipschitz constant $1/n_k$ and $R_k = 2^{n_k}$), and $\varphi_k : \Omega_k \to \R$ such that
$$\left\{\begin{array}{rcll}
(\p_t-\L_k)\varphi_k & = & 0 & \text{in} \ \Omega_k\\
\varphi_k & = & 0 & \text{on} \ \p_\Gamma\Omega_k\\
\varphi_k & \geq & 0 &\\
\|\varphi_k\|_{L^\infty(Q_{2^k})} & \leq & 128\cdot 2^{k(1+1/n_k+\varepsilon)} & \forall \ k \in \{0,\ldots,n_k\} \\
\|\varphi_k\|_{L^\infty(Q_1)} \ & = & 1 &
\end{array}\right.$$
while also satisfying
$$\sup\limits_{Q_{1/2}}\varphi_k < \left(\frac{1}{2}\right)^{1+1/n_k+\varepsilon}.$$

Then, by Proposition \ref{prop:bdry_Calpha}, for all $r \geq 1$
$$\|\varphi_k\|_{C^{0,\alpha}_p(Q_r)} \leq C(r),$$
for sufficiently large $k$, and therefore by Arzelà-Ascoli $\varphi_k \rightarrow \varphi_0$ locally uniformly, up to a subsequence.

Therefore, by Proposition \ref{prop:viscosity_limit}, $\varphi_0$ is a viscosity solution to
$$(\p_t-\Mplus)\varphi_0 \leq 0 \leq (\p_t-\Mminus)\varphi_0 \quad \text{in} \ \{x_n > 0\}$$
with $\varphi_0 = 0$ on $\{x_n = 0\}$, $\varphi_0 \geq 0$, $\|\varphi_0\|_{L^\infty(Q_1)} = 1$, and the growth control $\|\varphi_0\|_{L^\infty(Q_{2^k})} \leq 128\cdot 2^{k(1+\varepsilon)}$ for all $k \in \N$.

Hence, by Theorem \ref{thm:liouville}, $\varphi_0 = (x_n)_+$, contradicting the fact that
$$\frac{1}{2} = \sup\limits_{Q_{1/2}}\varphi_0 \leq \limsup\limits_{k\rightarrow\infty}\sup\limits_{Q_{1/2}}\varphi_k \leq \lim\limits_{k\rightarrow\infty} \left(\frac{1}{2}\right)^{1+1/n_k+\varepsilon} = \frac{1}{2^{1+\varepsilon}} < \frac{1}{2}.$$
\end{proof}

The next step is to iterate the inequality to obtain the following.

\begin{lem}\label{lem:special_induction}
Under the hypotheses of Lemma \ref{lem:special_1/2}, for all $0 < r_1 < r_2 \leq 1$,
$$\frac{\sup\limits_{Q_{r_1}}\varphi}{\sup\limits_{Q_{r_2}}\varphi} \geq \frac{1}{8}\left(\frac{r_1}{r_2}\right)^{1+1/n_0+\varepsilon}.$$
\end{lem}

\begin{proof}
Assume without loss of generality that $\varepsilon \in (0,\frac{1}{2})$. Let us first prove by induction that
$$\sup\limits_{Q_{2^{-k}}}\varphi \geq 2^{-k(1+\varepsilon+1/n_0)}.$$

It suffices to prove that the function
$$\bar\varphi(x,t) := \frac{\varphi(x/2,t/4)}{\sup\limits_{Q_{1/2}}\varphi}$$
also satisfies the hypotheses of Lemma \ref{lem:special_1/2}, and then the argument can be iterated. By construction, $(\p_t-\bar\L)\bar\varphi = 0$ in $Q_{2R_0}$, $\bar\varphi \geq 0$ and $\|\bar\varphi\|_{L^\infty(Q_1)} = 1$. Additionally, by Lemma \ref{lem:special_1/2}, for all $k \in \{1,\ldots,n_0+1\}$,
$$\sup\limits_{Q_{2^k}}\bar\varphi = \frac{\sup\limits_{Q_{2^{k-1}}}\varphi}{\sup\limits_{Q_{1/2}}\varphi} \leq 128\cdot 2^{(k-1)(1+1/n_0+\varepsilon)}\cdot2^{1+1/n_0+\varepsilon} = 128\cdot 2^{k(1+1/n_0+\varepsilon)}.$$

On the other hand, the reasoning  with $\bar\varphi$ implies that for any $k \in \N$,
$$\tilde\varphi(x,t) := \frac{\varphi(x/2^k,t/4^k)}{\sup\limits_{Q_{2^{-k}}}\varphi}$$
also satisfies the hypotheses of Lemma \ref{lem:special_1/2}, and hence, by the first part of the proof
$$\frac{\sup\limits_{2^{-k-m}}\varphi}{\sup\limits_{2^{-k}}\varphi} = \sup\limits_{Q_{2^{-m}}}\tilde\varphi \geq 2^{-m(1+\varepsilon+1/n_0)}.$$

Now, choose $k$ and $m$ in such a way that $2^{-k-m} \leq r_1 < 2^{-k-m+1}$ and $2^{-k-1} < r_2 \leq 2^{-k}$. Then,
$$\frac{\sup\limits_{Q_{r_1}}\varphi}{\sup\limits_{Q_{r_2}}\varphi} \geq \frac{\sup\limits_{2^{-k-m}}\varphi}{\sup\limits_{2^{-k}}\varphi} \geq 2^{-m(1+\varepsilon+1/n_0)} > \left(\frac{r_1}{4r_2}\right)^{1+\varepsilon+1/n_0} > \frac{1}{8}\left(\frac{r_1}{r_2}\right)^{1+\varepsilon+1/n_0}.$$
\end{proof}

Finally we can combine Lemma \ref{lem:special_exists} with Lemma \ref{lem:special_induction} to prove our target result.

\begin{proof}[Proof of Proposition \ref{prop:special}]
Choose $n_0$ from Lemma \ref{lem:special_1/2} with $\varepsilon/2$ instead of $\varepsilon$. Then, the function introduced in Lemma \ref{lem:special_exists} with $\varepsilon = 1/n_0$ satisfies the hypotheses of Lemma \ref{lem:special_1/2}, and the conclusion follows by Lemma \ref{lem:special_induction} (with $\varepsilon/2$ instead of $\varepsilon$).
\end{proof}

\section{Proof of the boundary Harnack}\label{sect:proof_BH}

The main ingredient in the proof of the boundary Harnack is the following expansion result.

\begin{prop}\label{prop:expansion}
Let $\alpha \in (0,\alpha_0)$, and let $\L$ be a non-divergence form operator as in \eqref{eq:non-divergence_operator}. There exists $\varepsilon_0 \in (0,1)$, only depending on $\alpha$, the dimension and the ellipticity constants, such that the following holds.

Let $\Omega$ be a parabolic Lipschitz domain in $Q_1$ in the sense of Definition \ref{defn:lipschitz_domain} with Lipschitz constant $\varepsilon_0$. Let $d(x',x_n,t) = x_n - \Gamma(x',t)$, and let $u$ be a solution to
$$\left\{\begin{array}{rclll}
u_t - \L u & = & f & \text{in} & \Omega\\ 
u & = & 0 & \text{on} & \p_\Gamma\Omega,
\end{array}\right.$$
and assume that $\|u\|_{L^\infty(Q_1)} \leq 1$, and that $f = g + h$ with
$$\|d^{1-\alpha}g\|_{L^\infty(Q_1)}+\|d^{-1/(n+1)-\alpha}h\|_{L^{n+1}(Q_1)} \leq 1.$$

Then, for each $r \in (0,1]$ there exists $K_r \in \R$ such that $|K_r| \leq C$ and
$$\|u-K_r\varphi\|_{L^\infty(Q_r)} \leq Cr^{1+\alpha},$$
where $\varphi$ is the near-linear solution introduced in Proposition \ref{prop:special} and $C$ depends only on $\alpha$, the dimension and the ellipticity constants.
\end{prop}

Before proving the expansion, we need to introduce the following growth estimate for blow-ups (cf. \cite[Lemma 4.4]{BFR18}), which is independent of the PDE and valid for general functions.

\begin{lem}\label{lem:aux2}
Let $\beta > \gamma > 0$. For every $j \in \N$, let $\Omega_j \subset \R^{n+1}$, and let ${u_j,\varphi_j : \Omega_j\to\R}$ such that $\|u_j\|_{L^\infty(Q_1)} \leq 1$, $\|\varphi_j\|_{L^\infty(Q_1)} = 1$, and, for every $0 < r_1 < r_2 < 1$,
$$\frac{\sup\limits_{Q_{r_1}}\varphi}{\sup\limits_{Q_{r_2}}\varphi} \geq c_1\left(\frac{r_1}{r_2}\right)^{\gamma}.$$

Let $K_{r,j} \in \R$ for every $r \in (0,1]$ and $j \in \N$, and assume that
$$\sup\limits_{j \in \N} |K_{r,j}| < \infty$$
and
$$\sup\limits_{r\in (0,1]}\theta(r) = \infty,$$
where
$$\theta(r) := \sup\limits_{\rho \in (r,1]}\sup\limits_{j \in \N}\rho^{-\beta}\|u_j - K_{\rho,j}\varphi_j\|_{L^\infty(Q_\rho)}.$$

Then, there exist sequences $\rho_m \downarrow 0$ and $j_m$ such that
$$\rho_m^{-\beta}\|u_{j_m}-K_{\rho_m,j_m}\varphi_{j_m}\|_{L^\infty(Q_{\rho_m})} \geq \frac{1}{2}\theta(\rho_m),$$
and
$$w_m := \frac{u_{j_m}(\rho_mx,\rho_m^2t)-K_{\rho_m,j_m}\varphi_{j_m}(\rho_mx,\rho_m^2t)}{\|u_{j_m}(\rho_mx,\rho_m^2t)-K_{\rho_m,j_m}\varphi_{j_m}(\rho_mx,\rho_m^2t)\|_{L^\infty(Q_1)}}$$
satisfies
$$\|w_m\|_{L^\infty(Q_R)} \leq CR^\beta \ \forall R \in [1,1/\rho_m).$$

Moreover, for every $0 < r_1 < r_2 < 1$,
$$\|(K_{r_2,j}-K_{r_1,j})\varphi_j\|_{L^\infty(Q_{r_2})} \leq Cr_2^\beta\theta(r_1).$$

The constant $C$ depends only on $\beta$, $\gamma$, and $c_1$.
\end{lem}

We defer the proof to Appendix \ref{sect:app}. Using Lemma \ref{lem:aux2}, we can prove the expansion:

\begin{proof}[Proof of Proposition \ref{prop:expansion}]
We divide the proof into four steps.

\textit{Step 1.} We reason by contradiction and construct a blow-up sequence. Let us prove first the following modified claim:

\textit{Claim.} For every $r \in (0,1]$, there exists $K_r$ with $|K_r| \leq C_0r^{-\alpha}$ such that
$$\|u-K_r\varphi\|_{L^\infty(Q_r)} \leq Cr^{1+\alpha},$$
with $C_0$ to be chosen later.

If we assume the claim does not hold, there are sequences $u_j, \varphi_j$, and $\Omega_j$ (with parabolic Lipschitz constant less than $1/j$), such that
$$\left\{\begin{array}{rclll}
(\p_t-\L_j)u_j & = & f_j & \text{in} & \Omega_j\\
u_j & = & 0 & \text{on} & \p_\Gamma\Omega_j,
\end{array}\right.$$
with $f_j = g_j+h_j$ such that 
$$\|d^{1-\alpha}g_j\|_{L^\infty(Q_1)}+\|d^{-1/(n+1)-\alpha}h_j\|_{L^{n+1}(Q_1)} \leq 1,$$
and
$$\|u_j - K_{r,j}\varphi_j\|_{L^\infty(Q_{r_j})} \geq jr_j^{1+\alpha},$$
where we choose
$$K_{r,j} := \frac{\int_{Q_r} u_j\varphi_j}{\int_{Q_r}\varphi_j^2}.$$

Then, by Proposition \ref{prop:special} with $\varepsilon = \alpha/2$,
$$\frac{\sup\limits_{Q_{r_1}}\varphi_j}{\sup\limits_{Q_{r_2}}\varphi_j} \geq \frac{1}{8}(r_1/r_2)^{1+\alpha/2},$$
and by Propositions \ref{prop:growth} and \ref{prop:special},
$$|K_{r,j}| \leq \frac{\left(\int_{Q_r}u_j^2\right)^{1/2}}{\left(\int_{Q_r}\varphi_j^2\right)^{1/2}} \leq \frac{Cr^{1-\alpha/2}}{cr^{1+\alpha/2}} =: C_0r^{-\alpha},$$
where we choose the constant $C_0$ from this computation.

Then, by Lemma \ref{lem:aux2} with $\gamma = 1 + \alpha/2$ and $\beta = 1 + \alpha$, there exists a sequence $\rho_m \downarrow 0$ such that
$$w_m := \frac{u_{j_m}(\rho_mx,\rho_m^2t)-K_{\rho_m,j_m}\varphi_{j_m}(\rho_mx,\rho_m^2t)}{\|u_{j_m}(\rho_mx,\rho_m^2t)-K_{\rho_m,j_m}\varphi_{j_m}(\rho_mx,\rho_m^2t)\|_{L^\infty(Q_1)}}$$
satisfies $\|w_m\|_{L^\infty(Q_1)} = 1$,
$$\|w_m\|_{L^\infty(Q_R)} \leq CR^{1+\alpha}, \ \forall R \in [1,1/\rho_m),$$
and
$$\int_{Q_1}w_m(x,t)\varphi_{j_m}(\rho_mx,\rho_m^2t) = 0$$
from the choice of $K_{\rho_m,j}$.

\textit{Step 2.} We will prove that $w_m \rightarrow (x_n)_+$ locally uniformly along a subsequence.

First, by the construction of $w_m$, we have (omitting the dependence of $f$ on $j_m$)
$$(\p_t - \Mplus)w_m \leq (\p_t - \L_m)w_m \leq \frac{2\rho_m^{1-\alpha}}{\theta(\rho_m)} |f(\rho_mx,\rho_m^2t)| \quad \text{in} \ \tilde \Omega_{j_m},$$
where $\L_m$ is the corresponding scaled operator, that has the same ellipticity constants, and
$$\tilde\Omega_{j_m} := \{(x,t) : (\rho_mx,\rho_m^2t) \in \Omega_{j_m}\}.$$
Note that $\tilde\Omega_{j_m}$ has Lipschitz constant lower or equal to $1/j_m$.

Let $d$ be the regularized distance in the domain $\Omega_{j_m}$, $\tilde d$ the regularized distance in $\tilde\Omega_{j_m}$, and let us omit the dependence of $g,h$ on $j$. Then, using Lemma \ref{lem:dist_scaling},
\begin{align*}
    \left\|\tilde d^{1-\alpha}\frac{2\rho_m^{1-\alpha}}{\theta(\rho_m)} |g(\rho_mx,\rho_m^2t)|\right\|_{L^\infty(Q_{1/\rho_m})} &\leq \left\|\frac{2\rho_m^{1-\alpha}}{\theta(\rho_m)}\left[\left(\frac{3d}{\rho_m}\right)^{1-\alpha}|g|\right](\rho_mx,\rho_m^2t)\right\|_{L^\infty(Q_{1/\rho_m})}\\
    &\leq \frac{6}{\theta(\rho_m)}\|d^{1-\alpha}|g|\|_{L^\infty(Q_1)}.
\end{align*}
Similarly, by the scaling of the $L^{n+1}$ norm,
\begin{align*}
    &\left\|\tilde d^{-1/(n+1)-\alpha}\frac{2\rho_m^{1-\alpha}}{\theta(\rho_m)} |h(\rho_mx,\rho_m^2t)|\right\|_{L^{n+1}(Q_{1/\rho_m})}\\
    &\quad\leq \left\|\frac{2\rho_m^{1-\alpha}}{\theta(\rho_m)}\left[\left(\frac{d}{3\rho_m}\right)^{-1/(n+1)-\alpha}|h|\right](\rho_mx,\rho_m^2t)\right\|_{L^{n+1}(Q_{1/\rho_m})}\\
    &\quad\leq \frac{18\rho_m^{(n+2)/(n+1)}}{\theta(\rho_m)}\|(d^{-1/(n+1)-\alpha}|h|)(\rho_mx,\rho_m^2t)\|_{L^{n+1}(Q_{1/\rho_m})}\\
    &\quad= \frac{18}{\theta(\rho_m)}\|d^{-1/(n+1)-\alpha}|h|\|_{L^{n+1}(Q_1)}.
\end{align*}

Therefore,
$$(\p_t-\Mplus)w_m \leq |g_m| + |h_m| \quad \text{in} \quad Q_{1/\rho_m}\cap\tilde\Omega_{j_m},$$
with
$$\|\tilde d^{1-\alpha}g_m\|_{L^\infty(Q_{1/\rho_m})} + \|\tilde d^{-1/(n+1)-\alpha}h_m\|_{L^{n+1}(Q_{1/\rho_m})} \leq \frac{18}{\theta(\rho_m)}.$$

Analogously,
$$(\p_t-\Mminus)w_m \geq -|g_m| - |h_m| \quad \text{in} \quad Q_{1/\rho_m}\cap\tilde\Omega_{j_m}.$$

Moreover, $w_m = 0$ on $\p_\Gamma\tilde\Omega_{j_m}$, and, for every $R \geq 1$, $\|w_m\|_{L^\infty(Q_R)} \leq CR^{1+\alpha}$ for sufficiently large $m$. Hence, by Proposition \ref{prop:bdry_Calpha},
$$\|w_m\|_{C^{0,\alpha}_p(Q_R)} \leq C(R),$$
uniformly in $m$, for $m$ large enough. Then, by Arzelà-Ascoli and Proposition \ref{prop:viscosity_limit}, we obtain that
$$w_m \rightarrow w \in C(\R^{n+1}),$$
locally uniformly along a subsequence, where $w$ is a viscosity solution of
\begin{equation*}
\left\{\begin{array}{rccclll}
w_t - \Mplus w &\leq & 0 &\leq & w_t - \Mminus w & \text{in} & \{x_n > 0\}\\
w & = & 0 & & & \text{on} & \{x_n = 0\},
\end{array}\right.
\end{equation*}
$\|w\|_{L^\infty(Q_1)} = 1$ and $\|w\|_{L^\infty(Q_R)} \leq CR^{1+\alpha}$ for all $R \geq 1$. Therefore, by Theorem \ref{thm:liouville}, $w = (x_n)_+$.

\textit{Step 3.} Let us consider the functions
$$\tilde \varphi_{j_m}(x,t) := \frac{\varphi_{j_m}(\rho_mx,\rho_m^2t)}{\|\varphi_{j_m}\|_{L^\infty(Q_{\rho_m})}}.$$

Then, $\|\tilde \varphi_{j_m}\|_{L^\infty(Q_1)} = 1$, and, by Proposition \ref{prop:special} with $\varepsilon = \alpha/2$, for all $1 \leq R \leq 1/\rho_m$,
$$\|\tilde \varphi_{j_m}\|_{L^\infty(Q_R)} \leq 8R^{1+\alpha/2}.$$

Finally, by the same arguments as in Step 2, $\tilde\varphi_{j_m} \rightarrow (x_n)_+$ locally uniformly along a subsequence.

\textit{Step 4.} We have $w_m \rightarrow (x_n)_+$ and $\tilde \varphi_m \rightarrow (x_n)_+$ locally uniformly. Now, recall that by the choice of $K_{r,j}$ in the construction of $w_m$,
$$\int_{Q_1}w_m\tilde \varphi_m = 0,$$
and passing to the limit,
$$\int_{Q_1}(x_n)_+^2 = 0,$$
which is a contradiction. Therefore, for every $r \in (0,1]$, there exists $|K_r| \leq C_0r^{-\alpha}$ such that
$$\|u - K_r\varphi\|_{L^\infty(Q_r)} \leq Cr^{1+\alpha}.$$
This is enough for $r \in (\frac{1}{2},1]$. For smaller values of $r$, observe that
\begin{align*}
    |K_r - K_{r/2}|(r/4)^{1+\alpha/2} &\leq \|(K_r - K_{r/2})\varphi\|_{L^\infty(Q_{r/2})}\\
    &\leq \|u - K_r\varphi\|_{L^\infty(Q_r)} + \|u - K_{r/2}\varphi\|_{L^\infty(Q_{r/2})} \leq Cr^{1+\alpha}.
\end{align*}
It follows that $|K_r - K_{r/2}| \leq Cr^{\alpha/2}$. Then, for $r \leq \frac{1}{2}$ we can write $r = 2^{-a}r_0$, with $r_0 \in (\frac{1}{2},1]$, and estimate
\begin{align*}
    |K_r| &\leq |K_{r_0}| + \sum\limits_{i = 0}^{a-1}|K_{2^{-i}r_0} - K_{2^{-i-1}r_0}| \leq C_0r_0^{-\alpha} + C\sum\limits_{i=0}^{a-1}(2^{-i}r_0)^{\alpha/2} \leq C.
\end{align*}
\end{proof}

Finally, we prove our main result.

\begin{proof}[Proof of Theorem \ref{thm:main}]
First, we will use a similar strategy to the proof of Proposition \ref{prop:bdry_Calpha} to estimate the Hölder seminorm of the quotient. Let $\varepsilon > 0$ in Proposition \ref{prop:special} such that $\gamma = \alpha - 7\varepsilon$. Recall that $\alpha$ is chosen in Remark \ref{obs:main_alpha}.

Let $p = (y',y_n,s)$ and $\rho \in (0,\frac{1}{16})$ such that $Q_{2\rho}(p) \subset \Omega\cap Q_{5/8}$, and let
$$R := \max\left\{\rho,\frac{y_n-\Gamma(y',s)}{3}\right\}.$$
Then, we distinguish two cases (cf. Proposition \ref{prop:bdry_Calpha}).

\textit{Case 1.} $R \geq \frac{1}{16}$. Then, $Q_{1/8}(p) \subset \Omega\cap Q_{3/4}$, and for all $(x',x_n,t) \in Q_{1/8}(p)$, $x_n \geq \Gamma(x',t) + \frac{1}{16}$, provided that the Lipschitz constant of the domain is small enough. By Proposition \ref{prop:growth}, $v \geq cm > 0$ in $Q_{1/8}(p)$. Furthermore, by Theorem \ref{thm:interior_Calpha}, $\|u\|_{C^{0,\gamma}_p(Q_{1/8}(p))} \leq C$ and $\|v\|_{C^{0,\gamma}_p(Q_{1/8}(p))} \leq C$. Therefore,
\begin{align*}
    &\left\|\frac{u}{v}\right\|_{C^{0,\gamma}_p(Q_\rho(p))} \leq \left\|\frac{u}{v}\right\|_{C^{0,\gamma}_p(Q_{1/8}(p))}\\
    &\qquad\quad\leq \frac{\|u\|_{C^{0,\gamma}_p(Q_{1/8}(p))}\|v\|_{L^\infty(Q_{1/8}(p))}+\|u\|_{L^\infty(Q_{1/8}(p))}\|v\|_{C^{0,\gamma}_p(Q_{1/8}(p))}}{\inf\limits_{Q_{1/8}(p)}v^2} \leq Cm^{-2}.
\end{align*}

\textit{Case 2.} $R < \frac{1}{16}$. Then, for all $(x',x_n,t) \in Q_{3R/2}(p)$,
$$\Gamma(x',t) + \frac{R}{2} \leq x_n \leq \Gamma(x',t) + 5R.$$

Let $\varphi$ be the special solution defined in Proposition \ref{prop:special}, centered at $(y',\Gamma(y',s),s)$. Then, by a translation of Proposition \ref{prop:expansion}, $w_1 = u - K_u\varphi$ and $w_2 = v - K_v\varphi$ satisfy
$$\|w_1\|_{L^\infty(Q_R(p))} \leq CR^{1+\alpha} \quad \text{and} \quad \|w_2\|_{L^\infty(Q_R(p))} \leq CR^{1+\alpha}.$$

Using that $d$ is comparable to $R$ in $Q_R(p)$, the right-hand side of the equation for $u$ can be estimated as
\begin{align*}
    \|f_1\|_{L^{n+1}(Q_R(p))} &\leq \|g_1\|_{L^{n+1}(Q_R(p))} + \|h_1\|_{L^{n+1}(Q_R(p))}\\
    &\lesssim R^{(n+2)/(n+1)}R^{\alpha-1}\|d^{1-\alpha}g_1\|_{L^\infty(Q_R(p))}\\
    &\quad + R^{\alpha+1/(n+1)}\|d^{-1/(n+1)-\alpha}h_1\|_{L^{n+1}(Q_R(p)))} \leq CR^{\alpha + 1/(n+1)},
\end{align*}
and analogously, in the equation for $v$, $\|f_2\|_{L^{n+1}(Q_R(p))} \leq CmR^{\alpha+1/(n+1)}$. Thus, by the interior estimates in Theorem \ref{thm:interior_Calpha}, and the growth of $v$ and $\varphi$, (see Propositions \ref{prop:growth} and \ref{prop:special}),
\begin{align*}
    [w_1]_{C^{0,\gamma}_p(Q_R(p))} &\leq CR^{-\gamma}(CR^{1+\alpha}+CR^{1+\alpha}) \leq CR^{1+7\varepsilon}\\
    [w_2]_{C^{0,\gamma}_p(Q_R(p))} &\leq CR^{-\gamma}(CR^{1+\alpha}+CmR^{1+\alpha}) \leq CR^{1+7\varepsilon}\\
    [v]_{C^{0,\gamma}_p(Q_R(p))} &\leq CR^{-\gamma}(CR^{1-\varepsilon}+CmR^{1+\alpha}) \leq CR^{1-\gamma-\varepsilon}\\
    [\varphi]_{C^{0,\gamma}_p(Q_R(p))} &\leq CR^{-\gamma}(CR^{1-\varepsilon}) \leq CR^{1-\gamma-\varepsilon}.
\end{align*}
Now, using that $u = w_1 + K_u\varphi$, we estimate first
\begin{align*}
    [w_1/v]_{C^{0,\gamma}_p(Q_R(p))} &\leq \frac{[w_1]_{C^{0,\gamma}_p(Q_R(p))}\|v\|_{L^\infty(Q_R(p))}+\|w_1\|_{L^\infty(Q_R(p))}[v]_{C^{0,\gamma}_p(Q_R(p))}}{\inf\limits_{Q_R(p)}v^2}\\
    &\leq C\frac{R^{1+7\varepsilon}R^{1-\varepsilon}+R^{1+\alpha}R^{1-\gamma-\varepsilon}}{m^2R^{2(1+\varepsilon)}} \leq Cm^{-2},
\end{align*}
where we used Proposition \ref{prop:growth} again to deduce that $v \geq cmR^{1+\varepsilon}$ in $Q_R(p)$. Then, we estimate
\begin{align*}
    [\varphi/v]_{C^{0,\gamma}_p(Q_R(p))} &\leq \frac{[v/\varphi]_{C^{0,\gamma}_p(Q_R(p))}}{\inf\limits_{Q_R(p)}(v/\varphi)^2} \leq \frac{[w_2/\varphi]_{C^{0,\gamma}_p(Q_R(p))}}{\inf\limits_{Q_R(p)}(v/\varphi)^2}\\
    &\leq \frac{[\varphi]_{C^{0,\gamma}_p}\|w_2\|_{L^\infty}+[w_2]_{C^{0,\gamma}_p}\|\varphi\|_{L^\infty}}{\inf(v/\varphi)^2\inf\varphi^2}\\
    &\leq C\frac{R^{1-\varepsilon-\gamma}R^{1+\alpha}+R^{1+7\varepsilon}R^{1-\varepsilon}}{(mR^{2\varepsilon})^2(R^{1+\varepsilon})^2} = 2Cm^{-2},
\end{align*}
where we omitted the domain in the second line to improve readability. Therefore
$$[u/v]_{C^{0,\gamma}_p(Q_R(p))} \leq [w_1/v]_{C^{0,\gamma}_p(Q_R(p))} + |K_u|[\varphi/v]_{C^{0,\gamma}_p(Q_R(p))} \leq Cm^{-2}.$$

Combining the three cases, by Lemma \ref{lem:teo_b2}, $[u/\varphi]_{C^{0,\gamma}_p(\Omega\cap Q_{1/2})} \leq Cm^{-2}.$

To obtain a bound for $\|u/\varphi\|_{C^{0,\gamma}_p(\Omega\cap Q_{1/2})}$, observe that 
$$\left\|\frac{u}{\varphi}\right\|_{L^\infty(Q_{1/2})} \leq \frac{u(e_n/4,0)}{v(e_n/4,0)} +[u/\varphi]_{C^{0,\gamma}_p(\Omega\cap Q_{1/2})}\left(\left|x-\frac{1}{4}e_n\right|+|t|^{1/2}\right) \leq \frac{1}{cm} + Cm^{-2},$$
where we used that $\|u\|_{L^\infty(Q_1)} \leq 1$ and $v(e_n/4,0) \geq cm > 0$ by the interior Harnack. Therefore, $\|u/\varphi\|_{C^{0,\gamma}_p(\Omega\cap Q_{1/2})} \leq Cm^{-2}$, as we wanted to prove.
\end{proof}

\section{Slit domains}\label{sect:slit}
Slit domains appear naturally when studying thin obstacle problems. In the case of parabolic slit domains, they appear in the time-dependent Signorini problem, and the boundary Harnack is known to hold in the homogeneous case; see \cite{PS14,DS22b}.

First, we will introduce a slightly different notation for this section. Given $x \in \R^{n+1}$, we will denote $x' = (x_1,\ldots,x_{n-1})$, i.e. $x = (x',x_n,x_{n+1})$. $B_r(x',x_n)$ will denote the $n$-dimensional ball of radius $r$ centered at $(x',x_n)$, and $B'_r(x')$ will be the one of $\R^{n-1}$. We also introduce the \textit{slit parabolic cylinders}:
$$Q_r(x,t) := B'_r(x') \times (x_n-r,x_n+r) \times (x_{n+1}-r,x_{n+1}+r) \times (t-r^2,t) \subset \R^{n+2}.$$

In this section, the domains that we will work with will be the following.
\begin{defn}\label{defn:slit_domain}
We say $\Omega$ is a parabolic Lipschitz slit domain in $Q_R$ with Lipschitz constant $L$ if $\Omega = Q_R \setminus E$, where
$$E = \big\{(x',x_n,0,t) \in Q_R \ | \ x_n \leq \Gamma(x',t) \big\},$$
and $\Gamma : B_R'\times[-R^2,0] \to \R$, with $\Gamma(0,0) = 0$ and $\|\Gamma\|_{C^{0,1}_p} \leq L$.

We will say that $E$ is the lateral boundary of $\Omega$, and write $\p_\Gamma\Omega := E$. The parabolic boundary will be defined as
$$\p_p\Omega := \p_\Gamma\Omega\cup\big(\overline{\Omega}\cap\p Q_R\cap\{t < 0\}\big).$$
\end{defn}

The goal of this section is to prove Theorem \ref{thm:slit_main}.

\subsection{Growth and boundary regularity}\label{subsect:slit_growth_reg}

In this section, we will follow the same scheme that in Section \ref{sect:bdry_growth_reg} to obtain growth estimates and regularity up to the boundary of solutions to the heat equation in slit Lipschitz domains. The final goal of the section is to prove the following:

\begin{prop}\label{prop:slit_bdry_Calpha}
Let $\gamma \in (0,\frac{1}{2})$. There exists $L_0 > 0$, depending only on $\gamma$ and the dimension, such that the following holds.

Let $\Omega$ be a parabolic slit domain in $Q_1$ as in Definition \ref{defn:slit_domain} with Lipschitz constant $L \leq L_0$. Let $u$ be a solution to
$$\left\{\begin{array}{rclll}
u_t - \Delta u & = & f & \text{in} & \Omega,\\
u & = & 0 & \text{on} & \p_\Gamma\Omega.
\end{array}\right.$$
Assume that $\|
f\|_{L^{n+2}(Q_1)} \leq K_0$. Then,
\[\|u\|_{C^{0,\gamma}_p(Q_{1/2})}\leq C(\|u\|_{L^\infty(Q_1)} + K_0).\]
The constant $C$ depends only on $\gamma$ and the dimension.
\end{prop}

We begin introducing parabolic homogeneous solutions in \textit{parabolic slit cones}, which will play the role of the powers of the regularized distance in Section \ref{sect:bdry_growth_reg}.

\begin{prop}[cf. \protect{\cite[Lemma 5.8]{FRS23}}]\label{prop:cone_solns}
There exists $\varepsilon_0 > 0$, only depending on the dimension, such that for all $\varepsilon \in (0,\varepsilon_0)$, there exist sufficiently small $\eta_-, \eta_+ > 0$, only depending on the dimension and $\varepsilon$, such that there exist unique positive solutions of
$$\p_t\varphi_- - \Delta\varphi_- = 0 \ \text{in} \ Q_1\setminus \mathcal C^-_{\eta_-}\quad \text{and}\quad \p_t\varphi_+ - \Delta\varphi_+ = 0 \ \text{in} \ Q_1\setminus\mathcal C^+_{\eta_+},$$
parabolically homogeneous of degree $\frac{1}{2} \pm \varepsilon$, i.e.
$$\varphi_-(\lambda x,\lambda^2 t) = \lambda^{\frac{1}{2}-\varepsilon}\varphi_-(x,t)\quad\text{and}\quad\varphi_+(\lambda x,\lambda^2 t) = \lambda^{\frac{1}{2}+\varepsilon}\varphi_+(x,t) \quad \forall \lambda > 0,$$
such that $\|\varphi_-\|_{L^\infty(Q_1)} = \|\varphi_+\|_{L^\infty(Q_1)} = 1$, where
$$C^+_{\eta_+} := \{x_n \leq \eta_+(|x'|+|t|^{1/2}), x_{n+1} = 0\}$$
and
$$C^-_{\eta_-} := \{x_n \leq -\eta_-(|x'|+|t|^{1/2}), x_{n+1} = 0\}.$$

Moreover, $\eta_- \rightarrow 0$ and $\eta_+ \rightarrow 0$ monotonically as $\varepsilon \rightarrow 0$, and there exists $m > 0$ such that
$$\varphi_\pm \geq m \ \text{in} \ Q_2\cap\left\{|x_{n+1}|\geq\frac{1}{n+1}\right\}.$$
\end{prop}

We defer the proof to Appendix \ref{sect:app}. Notice that $\varphi_-$ and $\varphi_+$ satisfy the following Hopf-type estimate.

\begin{lem}\label{lem:slit_hopf}
Let $\varphi_\pm$ be as in Proposition \ref{prop:cone_solns}. Then,
$$\varphi_\pm \geq c|x_{n+1}| \quad \text{in} \ Q_{1/2},$$
for a dimensional constant $c > 0$.
\end{lem}

\begin{proof}
It follows from Proposition \ref{prop:cone_solns} and Hopf's lemma.
\end{proof}

Now, we proceed with the same strategy as in Lemmas \ref{lem:barrier_super_abp} an \ref{lem:barrier_sub_abp} to obtain the desired bounds. We start with the upper bound:

\begin{lem}\label{lem:slit_super_abp}
Let $\varepsilon \in (0,1/2)$. There exist sufficiently small $\eta > 0$ and sufficiently large $K > 0$, only depending on the dimension and $\varepsilon$, such that the following holds.

Let $u$ be a solution to
\begin{equation*}
\left\{
\begin{array}{rclll}
u_t - \Delta u & \leq & f & \text{in} & \Omega\\
u & \leq & 1 & \text{on} & \p_p\Omega\\
u & \leq & 0 & \text{on} & \p_\Gamma\Omega\\
\end{array}
\right.
\end{equation*}
where $\|f\|_{L^{n+2}(\Omega)} \leq \eta$, and $\Omega$ is a parabolic slit domain in $Q_1$ in the sense of Definition \ref{defn:slit_domain} with Lipschitz constant $\eta$. Then,
$$\|u\|_{L^\infty(Q_r)} \leq K r^{1/2-\varepsilon}, \quad \forall r \in (0,1).$$
\end{lem}

\begin{proof}
We will use the comparison principle with a barrier. Assume first that $f = 0$.

Let $\varphi_-$ be as in Proposition \ref{prop:cone_solns} with $\varepsilon/2$ instead of $\varepsilon$, let $\eta = \eta_-$ and define
$$v := a\varphi_- + |x'|^2 + x_n^2 - 2nx_{n+1}^2 - 2nt.$$

First, by Lemma \ref{lem:slit_hopf} and letting $a = (2n+1)/c$,
$$a\varphi_- \geq ac|x_{n+1}| \geq (2n+1)|x_{n+1}| \quad \text{in} \ Q_1.$$

Then, $v$ satisfies $v_t - \Delta v = 0$ in $Q_1 \setminus \mathcal C_{\eta_-}^- \supset \Omega$. Moreover, on the parabolic boundary of $\Omega$ we can distinguish the following cases:

\begin{itemize}
    \item When $t = -1$, $v \geq a\varphi_- + 2n-2nx_{n+1}^2$.
    \item When $|x'| = 1$, and when $|x_n| = 1$, $v \geq a\varphi_- +1 - 2nx_{n+1}^2$.
    \item When $|x_{n+1}| = 1$, $v \geq a\varphi_- - 2n$.
    \item On $\p_\Gamma\Omega \subset \{x_{n+1} = 0\}$, $v \geq 0$.
\end{itemize}

To treat the first two cases, note that
$$a\varphi_- + 1 - 2nx_{n+1}^2 \geq 1 + (2n+1)|x_{n+1}| - 2nx_{n+1}^2 \geq 1.$$
Then, when $|x_{n+1}| = 1$,
$$a\varphi_- - 2n \geq 1.$$

Therefore, by the comparison principle, $u \leq v$ in $\Omega$. To include the right-hand side, notice that we can always write $u = u_0 + u_f$, where
\[
\left\{
\begin{array}{rclll}
(\p_t - \Delta)u_0 & = & 0 & \text{in} & \Omega\\
u_0 & = & u & \text{on} & \p_p \Omega
\end{array}
\right.
\ \text{and}\
\left\{
\begin{array}{rclll}
(\p_t-\Delta)u_f & = & f & \text{in} & \Omega\\
u_f & = & 0 & \text{on} & \p_p \Omega.
\end{array}
\right.
\]

Then, by Theorem \ref{thm:ABPKT}, and applying the reasoning above to $u_0$,
$$u = u_0 + u_f \leq v + C\|f\|_{L^{n+2}(Q_1)} \leq v + C\eta.$$

Now, we do an iteration scheme as in Lemma \ref{lem:barrier_super_abp}. For that, let us define the rescaled functions
$$u_j(x,t) := \frac{u(\rho^jx,\rho^{2j}t)}{\rho^{j(1/2-\varepsilon)}},$$
with $\rho > 0$ to be chosen later. Now,
$$(\p_t-\Delta)u_j = \rho^{j(3/2+\varepsilon)}f_j,$$
with $f_j(x,t) := f(\rho^jx,\rho^{2j}t)$. Let $b_j := \|u_j\|_{L^\infty(Q_1)}$. Writing $u = u_0+u_f$ again,
$$b_0 \leq \|u_0\|_{L^\infty(Q_1)} + \|u_f\|_{L^\infty(Q_1)} \leq \|v\|_{L^\infty(Q_1)} + C\|f\|_{L^{n+2}(Q_1)} \leq a + 4n + 2 + C\eta.$$

We will show by induction that $b_j \leq b_0$ for all $j \geq 0$. Indeed, by the first part of the proof and induction hypothesis,
$$u_j \leq v + C\|\rho^{j(3/2+\varepsilon)}f_j\|_{L^{n+2}(Q_1)},$$
and hence, using that $v$ is a sum of terms with at least parabolic homogeneity $\frac{1-\varepsilon}{2}$,
\begin{align*}
    b_{j+1}\rho^{1/2-\varepsilon} &\leq \|v\|_{L^\infty(Q_\rho)} + C\rho^{j(3/2+\varepsilon)}\|f_j\|_{L^{n+2}(Q_1)}\\
    &\leq b_0\rho^{(1-\varepsilon)/2}+\rho^{j(3/2+\varepsilon-(n+3)/(n+2))}\|f\|_{L^{n+2}(Q_{\rho^j})}\\
    &\leq b_0\rho^{(1-\varepsilon)/2}+\|f\|_{L^{n+2}(Q_1)} \leq b_0\rho^{(1-\varepsilon)/2} + \eta \leq b_0\rho^{1/2-\varepsilon},
\end{align*}
after choosing sufficiently small $\eta$ and $\rho$.

Finally, given $r \in (0,1)$, let $j$ be such that $\rho^{j+1} < r \leq \rho^j$. Therefore,
$$\|u\|_{L^\infty(Q_r)} \leq \|u\|_{L^\infty(Q_{\rho^j})} = b_j\rho^{j(1/2-\varepsilon)} < b_0(r/\rho)^{1/2-\varepsilon}.$$
\end{proof}

We also deduce a lower bound:

\begin{lem}\label{lem:slit_sub_abp}
Let $\varepsilon \in (0,1/6)$ and $\mu \in (0,1)$. There exist sufficiently small $\eta, k > 0$, only depending on the dimension, $\mu$ and $\varepsilon$, such that the following holds.

Let $u \geq 0$ be a solution to
\begin{equation*}
\left\{
\begin{array}{rclll}
u_t - \Delta u & \geq & f & \text{in} & \Omega\\
u & \geq & 1 & \text{on} & \p_{\mathrm{up}}\Omega\\
\end{array}
\right.
\end{equation*}
where $\|f\|_{L^{n+2}(\Omega)} \leq \eta$, $\Omega$ is a parabolic slit domain in $Q_1$ in the sense of Definition \ref{defn:slit_domain} with Lipschitz constant $\eta$, and
$$\p_{\mathrm{up}}\Omega := \{x_n = 1\} \cap \overline{\Omega}.$$

Then,
$$u(re,0) \geq k r^{1/2+\varepsilon}, \quad \forall r \in \left(0,\frac{1}{4}\right),$$
for all unit vectors $e = \cos(\theta)e_n + \sin(\theta)e_{n+1}$ with $\cos(\theta) \geq -1+\mu$.
\end{lem}

\begin{proof}
The proof is very similar to that of Lemma \ref{lem:slit_super_abp}. Assume first that $f = 0$.

Let
$$\Omega^{(1)} := \left\{(x,t) \in \Omega : |x'| < \frac{1}{2}, |x_n| < \frac{1}{2}, |x_{n+1}| < \frac{1}{\sqrt{2(4n+13)}}, -\frac{1}{2} < t\right\}$$
and
$$\p_{\mathrm{up}}\Omega^{(1)} := \overline{\Omega^{(1)}}\cap\left\{|x_{n+1}| = \frac{1}{\sqrt{2(4n+13)}}\right\}.$$
By an analogous reasoning to the proof of Lemma \ref{lem:barrier_sub}, $u \geq c_0$ on $\p_{\mathrm{up}}\Omega^{(1)}$, and $u(re,0) \geq c_0$ for all $r \in (\frac{1}{\sqrt{2(4n+13)}},\frac{1}{2})$, given that $\eta$ is small enough. However, here $c_0$ depends on $\mu$.

Let $\varphi_+$ be as in Proposition \ref{prop:cone_solns} with $\varepsilon/2$ instead of $\varepsilon$, let $\eta \leq \eta_+$ and define
$$v := c_0\left(\frac{1}{2}\varphi_+ + 2t + (4n+13)x_{n+1}^2 - 4|x'|^2 - 16\left(x_n+\frac{1}{4}\right)_-^2\right).$$

Then, $v$ satisfies $v_t - \Delta v \leq 0$ in $Q_1 \setminus \mathcal C_{\eta_+}^+$. Moreover, on the parabolic boundary of $\Omega^{(1)}$ we can distinguish the following cases (recall that $\|\varphi_+\|_{L^\infty(Q_1)} = 1$ and that $|x_{n+1}| \leq 1/\sqrt{2(4n+13)}$).

\begin{itemize}
    \item When $t = -1/2$, $v \leq c_0(1/2-1+1/2) = 0$.
    \item When $|x'| = 1/2$, and when $x_n = -1/2$, $v \leq c_0(1/2+1/2-1) = 0$.
    \item When $x_n = 1/2$, $v \leq c_0(1/2+1/2) = c_0$.
    \item When $|x_{n+1}| = 1/\sqrt{2(4n+13)}$, $v \leq c_0(1/2+1/2) = c_0$.
    \item On $\p_\Gamma\Omega \subset C_{\eta_+}^+$, $v \leq 0$.
\end{itemize}

Therefore, by the comparison principle, $u \geq v$ in $\overline{\Omega^{(1)}}$. The right-hand side can be included in the same way as in Lemma \ref{lem:slit_super_abp}, giving
$$u \geq v - C\|f\|_{L^{n+2}(Q_1)} \geq v - C\eta \quad \text{in} \ \overline{\Omega^{(1)}}.$$

Now, we do an iteration scheme as in Lemma \ref{lem:barrier_sub_abp}. For that, let us define the rescaled functions
$$u_j(x,t) := \frac{u(\rho^jx,\rho^{2j}t)}{\rho^{j(1/2+\varepsilon)}},$$
with $\rho > 0$ to be chosen later. Now,
$$(\p_t-\Delta)u_j = \rho^{j(3/2-\varepsilon)}f_j,$$
with $f_j(x,t) := f(\rho^jx,\rho^{2j}t)$. Let 
$$b_j := \inf\limits_{\p_{\mathrm{up}}\Omega_j}u_j,$$
where $\Omega_j$ is the appropriate scaled domain of $u_j$ in $Q_1$. By hypothesis, $b_0 \geq 1$. We will show by induction that $b_j \geq 1$ for all $j \geq 0$.

Indeed, by the first part of the proof and induction hypothesis,
$$u_j \geq v - C\|\rho^{j(3/2-\varepsilon)}f_j\|_{L^{n+2}(Q_1)},$$
and hence, using the parabolic homogeneity of $\varphi_+$, and using the same scaling arguments as in Lemma \ref{lem:slit_super_abp}
\begin{align*}
    b_{j+1}\rho^{1/2+\varepsilon} &\geq c_0\frac{\inf\limits_{\p_{\mathrm{up}}\Omega_j}\varphi_+}{2}\rho^{(1+\varepsilon)/2} - (4n+14)c_0\rho^2 - C\rho^{j(3/2-\varepsilon)}\|f_j\|_{L^{n+2}(\Omega_j^{(1)})}\\
    &\geq c_0c\rho^{(1+\varepsilon)/2} - C\rho^2-C\rho^{j(3/2-\varepsilon-(n+3)/(n+2))}\eta \geq \rho^{1/2+\varepsilon},
\end{align*}
for sufficiently small $\eta$ and $\rho$.

Finally, by the first part of the proof,
$$u_j \geq v - C\eta \quad \text{in} \ \overline{\Omega^{(1)}},$$
and then
$$u_j \geq \frac{c_0}{2}\varphi_+ - C\eta \quad \text{in} \ \{t = 0, x' = 0, x_n \geq -1/4\},$$
which in turn implies $u_j(re,0) \geq c(\mu)r^{1/2+\varepsilon} - C\eta \geq kr^{1/2+\varepsilon}$ for all $r \in [1/(4\rho),1/4)$. The conclusion follows undoing the scaling.
\end{proof}

Combining Lemmas \ref{lem:slit_super_abp} and \ref{lem:slit_sub_abp}, we can deduce a growth estimate for all solutions.

\begin{prop}\label{prop:slit_growth}
Let $\varepsilon \in (0,1/6)$. There exists sufficiently small $\eta > 0$, only depending on $\varepsilon$ and the dimension, such that the following holds.

Let $\Omega$ be a parabolic slit domain in $Q_1$ in the sense of Definition \ref{defn:slit_domain}. Let $u$ be a solution to
$$\left\{\begin{array}{rclll}
u_t - \Delta u & = & f & \text{in} & \Omega,\\
u & = & 0 & \text{on} & \p_\Gamma\Omega.
\end{array}\right.$$
Assume that $\|u\|_{L^\infty(Q_1)} \leq 1$, and $\|f\|_{L^{n+2}(Q_1)} \leq \eta$.

Then,
$$|u| \leq C|(x_n - \Gamma(x',t),x_{n+1})|^{1/2-\varepsilon} \quad \text{in} \quad \Omega \cap Q_{3/4}.$$
Moreover, if $u$ is nonnegative, $m = u\left(\frac{e_n}{2},-\frac{3}{4}\right) > 0$, $\|f\|_{L^{n+2}(Q_1)} \leq \eta m$, then
$$u \geq cm|(x_n - \Gamma(x',t),x_{n+1})|^{1/2+\varepsilon} \quad \text{in} \quad \Omega \cap \mathcal{C} \cap Q_{3/4},$$
where $\mathcal{C}$ is the "cone" defined as
$$\mathcal{C} := \left\{(x,t) \in \R^{n+2} \ | \ x_n - \Gamma(x',t) + 10|x_{n+1}| \geq 0\right\}.$$
The constants $C$ and $c$ are positive and depend only on the dimension, $\varepsilon$, and the ellipticity constants.
\end{prop}

\begin{proof}
We will follow the same strategy as in Proposition \ref{prop:growth}. The proof of the upper bound is exactly the same, using Lemma \ref{lem:slit_super_abp} instead of Lemma \ref{lem:barrier_super_abp}.

For the second estimate, let $\Omega(x_0,t_0) := \Omega \cap Q_{1/8}(x_0,t_0)$ and notice that
$$\bigcup\limits_{(x_0,t_0) \in \p_\Gamma\Omega\cap Q_{3/4}} \p_{\mathrm{up}}\Omega(x_0,t_0) \subset E := \overline{B_{7/8}'}\times[1/8-2\eta,1/8+2\eta]\times[-1/8,1/8]\times[-37/64,0],$$
where $\p_{\mathrm{up}}\Omega(x_0,t_0) := \overline{\Omega(x_0,t_0)}\cap\{x_n = x_{0,n}+1/8\}$, analogously to Lemma \ref{lem:slit_sub_abp}. The rest follows as in Proposition \ref{prop:growth}, but here we use the interior Harnack to see that $u \geq c_1m$ in
$$F := Q_{3/4} \setminus \{x_n < 1/32 - 2\eta, \ |x_{n+1}| < 1/32\}$$
instead of $E$. Here we use $\eta < 1/256$ to ensure that there is a uniform positive distance from $F$ to the boundary of $\Omega$.
\end{proof}

Finally, we can deduce the Hölder regularity up to the boundary.

\begin{proof}[Proof of Proposition \ref{prop:slit_bdry_Calpha}]

We will use the same strategy as in Proposition \ref{prop:bdry_Calpha}. Let us assume without loss of generality that $\|u\|_{L^\infty(Q_1)} \leq 1$ and $K_0 = \eta$ (with $\eta$ from Proposition \ref{prop:slit_growth}). Then, by Proposition \ref{prop:slit_growth},
$$|u| \leq C(|x_n - \Gamma(x',t)|+|x_{n+1}|)^\gamma \quad \text{in} \ \Omega\cap Q_{3/4}.$$

Then, we will use interior estimates in combination with Lemma \ref{lem:teo_b2}. Let $p = (y',y_n,y_{n+1},s)$ and $\rho \in (0,\frac{1}{16})$ such that $Q_{2\rho}(p) \subset \Omega\cap Q_{5/8}$, and let
$$R := \max\left\{\rho,\frac{|y_{n+1}|}{3},\frac{|y_n-\Gamma(y',s)|}{3}\right\}.$$

We distinguish four cases:

\textit{Case 1.} $R \geq \frac{1}{16}$, and $y_n > \Gamma(y',s) - 3R$ or $|y_{n+1}| \geq \frac{1}{8}$. Then, $Q_{1/8}(p) \subset \Omega\cap Q_{3/4}$. Then, by Theorem \ref{thm:interior_Calpha} combined with the fact that $\|u\|_{L^\infty(Q_1)} \leq 1$ and $\|f\|_{L^{n+2}(Q_1)} \leq 1$, we obtain
$$[u]_{C^{0,\gamma}_p(Q_\rho(p))} \leq[u]_{C^{0,\gamma}_p(Q_{1/8}(p))} \leq C.$$

\textit{Case 2.} $R \geq \frac{1}{16}$, $y_n = \Gamma(y',s) - 3R$ and $|y_{n+1}| < \frac{1}{8}$. Then, $Q_{1/8}(p) \subset Q_{3/4}$ and
$$Q_{1/8}(p)\cap\{x_{n+1}=0\} \subset \{x_n \leq \Gamma(x',t), x_{n+1} = 0\}.$$

Hence, $u \equiv 0$ on $Q_{1/8}(p)\cap\{x_{n+1}=0\}$, and we can decompose $u = u_1 + u_2$, with $u_1 = u\chi_{\{x_{n+1}>0\}}$ and $u_2 = u\chi_{\{x_{n+1}<0\}}$.

Now, let $\tilde u_1$ and $\tilde u_2$ be the odd reflections of $u_1$ and $u_2$ across $\{x_{n+1} = 0\}$, and note that they satisfy $(\p_t - \Delta)\tilde u_i = f_i$ in $Q_{1/8}(p)$, with $f_i$ being the appropriate reflection of $f$. Therefore, by Theorem \ref{thm:interior_Calpha} again as in \textit{Case 1}, $[u]_{C^{0,\gamma}_p(Q_\rho(p))} \leq C$.

\textit{Case 3.} $R < \frac{1}{16}$, and $y_n > \Gamma(y',s) - 3R$ or $|y_{n+1}| \geq 2R$. Then, $Q_{2R}(p) \subset \Omega \cap Q_{3/4}$. Moreover, since $|y_{n+1}| \leq 3R$ and $|y_n - \Gamma(y',s)| \leq 3R$, $|u| \leq K'R^\gamma$ in $Q_{2R}(p)$. Therefore,
\begin{align*}
    [u]_{C^{0,\gamma}_p(Q_R(p))} &\leq C(R^{-\gamma}\|u\|_{L^\infty(Q_{2R}(p))} + R^{(n+1)/(n+2)-\gamma}\|f\|_{L^{n+2}(Q_{2R}(p))})\\
    &\leq C(K'+R^{(n+1)/(n+2)-\gamma}\|f\|_{L^{n+2}(Q_1)}) \leq C'.
\end{align*}

\textit{Case 4.} $R < \frac{1}{16}$, $y_n = \Gamma(y',s) - 3R$ and $|y_{n+1}| < 2R$. We proceed by an odd reflection as in \textit{Case 2} and then use the estimates of \textit{Case 3}.
\end{proof}

\subsection{Special solution}\label{subsect:slit_special}

Our goal now is to construct a special solution which is \textit{almost homogeneous with parabolic homogeneity $\frac{1}{2}$} (cf. Proposition \ref{prop:special} for the same result in \textit{one-sided} Lipschitz domains).

\begin{prop}\label{prop:slit_special}
Let $\varepsilon \in (0,\varepsilon_0)$, with $\varepsilon_0$ from Proposition \ref{prop:cone_solns}. Then, there exist $\delta > 0$ and $C > 0$, only depending on $\varepsilon$ and the dimension, such that the following holds.

Let $\Omega$ be a parabolic slit domain in $Q_1$ in the sense of Definition \ref{defn:slit_domain} with Lipschitz constant $\delta$. Let $\varphi_\pm$ be the parabolically homogeneous solutions introduced in Proposition \ref{prop:cone_solns}, in a way that $\eta_\pm > \delta$ so that
$$Q_1\setminus C^+_{\eta^+} \subset \Omega \subset Q_1 \setminus C^-_{\eta-}.$$

Then, there exists $\varphi : \Omega \to \R$ such that
$$\left\{\begin{array}{rclll}
     \varphi_t - \Delta\varphi & = & 0 & \text{in} & \Omega\\
     \varphi & = & 0 & \text{on} & \p_\Gamma\Omega,
\end{array}\right.$$
$\varphi \geq 0$, $\varphi$ is even in $x_{n+1}$, $\|\varphi\|_{L^\infty(Q_1)} = 1$,
$$\frac{1}{C}\varphi_+ \leq \varphi \leq C\varphi_-,$$
and for all $0 < r_1 < r_2 \leq 1$,
$$\frac{\sup\limits_{Q_{r_1}}\varphi}{\sup\limits_{Q_{r_2}}\varphi} \geq \frac{1}{8}\left(\frac{r_1}{r_2}\right)^{1/2+\varepsilon}.$$
\end{prop}

First, we prove the following almost positivity property for supersolutions to the heat equation.

\begin{lem}\label{lem:almost_positivity}
Let $E \subset \{x_{n+1} = 0\}\cap Q_2$ be a closed set, and let $w$ be a supersolution to
$$\left\{\begin{array}{rclll}
w_t - \Delta w & \geq & 0 & \text{in} & Q_2 \setminus E\\ 
w & = & 0 & \text{on} & E\\
w & \geq & 1 & \text{in} & Q_2\cap\{|x_{n+1}| \geq \frac{1}{n+1}\}\\
w & \geq & -1 & \text{in} & Q_2.
\end{array}\right.$$

Then,
$$w \geq 0 \quad \text{in} \ Q_1.$$
\end{lem}

\begin{proof}
We will proceed by comparison with a barrier. Let $(y,s) \in Q_1$. If $y_{n+1} \geq \frac{1}{n+1}$, there is nothing to prove. Otherwise, consider the set
$$\Omega := B_1(y',y_n)\times\left(-\frac{1}{n+1},\frac{1}{n+1}\right)\times(s-1,s) \setminus E.$$

By construction, $\Omega \subset Q_2 \setminus E$. Now, consider
$$v = w + 2|(x',x_n) - (y',y_n)|^2+4(s-t)-2(n+1)|x_{n+1}|^2,$$
which by construction is also a supersolution for the heat equation in $\Omega$. Moreover, on the parabolic boundary we can distinguish the following cases:
\begin{itemize}
    \item When $t = s - 1$, $v \geq -1 + 4 - \frac{2}{n+1} \geq 0$.
    \item When $|(x',x_n) - (y',y_n)| = 1$, $v \geq -1 + 2 - \frac{2}{n+1} \geq 0$.
    \item When $|x_{n+1}| = \frac{1}{n+1}$, $v \geq 1 - \frac{2}{n+1} \geq 0$.
    \item On $E$, $v \geq w \geq 0$.
\end{itemize}

Therefore, by the comparison principle, $v(y,s) \geq 0$, and it follows that 
$$w(y,s) = v(y,s) + 2(n+1)|y_{n+1}|^2 \geq 0.$$
\end{proof}

As a consequence, we can deduce the following.

\begin{lem}\label{lem:phi+_leq_phi-}
There exist $C, \varepsilon_0 > 0$, only depending on the dimension, such that for all $\varepsilon \in (0,\varepsilon_0)$,
$$\varphi_+ \leq C\varphi_- \quad \text{in} \ Q_1,$$
where $\varphi_\pm$ are the parabolically homogeneous solutions introduced in Proposition \ref{prop:cone_solns}.
\end{lem}

\begin{proof}
Let
$$w := \frac{2}{m}\varphi_- - \frac{1}{2}\varphi_+$$
with $m$ from Proposition \ref{prop:cone_solns}.

Since $\{\varphi_- = 0\} \subset \{\varphi_+ = 0\}$, $w$ is a supersolution for the heat equation in $Q_2 \setminus \{\varphi_- = 0\}$. Moreover, since $\|\varphi_+\|_{L^\infty(Q_2)} \leq 2^{1/2+\varepsilon}$ by homogeneity,
$w \geq -1$ in $Q_2$, and $w \geq 1$ whenever $|x_{n+1}| \geq \frac{1}{n+1}$.

Hence, by Lemma \ref{lem:almost_positivity}, $w \geq 0$ in $Q_1$ and the conclusion follows.
\end{proof}

We will also need a Liouville theorem for slit domains:

\begin{thm}\label{thm:slit_liouville}
Let $\alpha \in (0,\frac{1}{2})$, and let $u$ be a solution to
$$\left\{\begin{array}{rcccl}
(\p_t - \Delta)u & = & 0 & \text{in} & \left(\R^{n+1} \setminus \{x_n \leq 0, x_{n+1} = 0\}\right)\times(-\infty,0)\\
u & = & 0 & \text{on} & \{x_n \leq 0, x_{n+1} = 0\}\\
u(x',x_n,x_{n+1},t) & = & u(x',x_n,-x_{n+1},t) & \text{in} & \R^{n+1}\times(-\infty,0)
\end{array}\right.$$
with the growth control
$$\|u\|_{L^\infty(Q_R)} \leq C(1+R^{1+\alpha}), \ \forall R \geq 1.$$
Then, $u = a\varphi_0$ for some $a \in \R$, where
$$\varphi_0(x,t) := \operatorname{Re}(\sqrt{x_n + ix_{n+1}}).$$
\end{thm}

\begin{proof}
    We will proceed as in \cite[Theorem 4.11]{FR17}. Let $\gamma \in (0,\frac{1}{2})$ such that $3\gamma > 1 + \alpha$, and let
    $$u_r(x,t) := \frac{u(rx,r^2t)}{r^{1+\alpha}}$$
    for $r \geq 1$. Notice how $u_r$ also satisfies the hypotheses.

    Now, by Proposition \ref{prop:slit_bdry_Calpha},
    $$[u]_{C^{0,\gamma}(Q_{2r})} = r^{1+\alpha-\gamma}[u_r]_{C^{0,\gamma}(Q_2)} \leq Cr^{1+\alpha-\gamma}\|u_r\|_{L^\infty(Q_4)} \leq Cr^{1+\alpha-\gamma}.$$
    Then, given $h \in B_1$ such that $h_n = h_{n+1} = 0$ and $\tau \in (-1,1)$, let 
    $$u^{(1)}(x,t) := u(x+h,t+\tau) - u(x,t)$$
    and notice that it is also an even solution to the heat equation in the same domain. Hence,
    $$\|u^{(1)}\|_{L^\infty(Q_r)} \leq Cr^{1+\alpha-\gamma}.$$
    Now we can repeat the procedure starting with $u^{(i)}$ instead of $u$ to obtain that
    $$\|u^{(2)}\|_{L^\infty(Q_r)} \leq Cr^{1+\alpha-2\gamma}$$
    and
    $$\|u^{(3)}\|_{L^\infty(Q_r)} \leq Cr^{1+\alpha-3\gamma},$$
    where $u^{(i+1)} := u^{(i)}(\cdot+h,\cdot+\tau) - u^{(i)}$. Letting $r \rightarrow \infty$ in the last expression, we obtain that $u^{(3)}$ is identically zero, and therefore $u(\cdot,x_n,x_{n+1},\cdot)$ is a third order polynomial, but from the growth condition on $u$ we deduce that actually $u$ is of the form
    $$u(x,t) = \phi(x_n,x_{n+1}) + \psi(x_n,x_{n+1})\cdot x' + \zeta(x_n,x_{n+1})t.$$
    Since the domain is translation invariant in the $x'$ and $t$ directions, we deduce that $\phi,\psi,\zeta$ are two-dimensional solutions to the heat equation that do not depend on time, i.e., harmonic functions in $\R^2 \setminus (-\infty,0)$. Hence, they are of the form
    $$\sum\limits_{k = 0}^\infty a_k \operatorname{Re}((x_n+ix_{n+1})^{k+1/2}) + \sum\limits_{k=1}^\infty b_k \operatorname{Im}((x_n+ix_{n+1})^k).$$
    Finally, from the growth and the even symmetry of $u$, we deduce that $\phi \equiv a_0\operatorname{Re}((x_n+ix_{n+1})^{1/2})$ and $\psi\equiv\zeta\equiv0$, as we wanted to prove.
\end{proof}

\begin{obs}
    The assumption of $u$ being even in the $x_{n+1}$ direction is necessary to discard linear terms of the form $b_1x_{n+1}$. This is motivated by the study of free boundary problems such as the parabolic Signorini problem, where the solution being even is a natural assumption (see for instance \cite{DGPT17}).
\end{obs}

Now, we are ready to replicate the strategy of Section \ref{sect:special_soln}. First, we construct solutions with a controlled growth.
\begin{lem}\label{lem:slit_special_exists}
Let $\varepsilon \in (0,\varepsilon_0)$, with $\varepsilon_0$ from Proposition \ref{prop:cone_solns}. There exist $C > 0$, only depending on the dimension, and $\delta_1 \in (0,\varepsilon)$, only depending on the dimension and $\varepsilon$, such that the following holds.

Let $R = 2^{1/\varepsilon}$ and let $\Omega$ be a parabolic slit domain in $Q_R$ in the sense of Definition \ref{defn:slit_domain} with Lipschitz constant $\delta_1$. Then, there exists $\varphi: \Omega \to \R$ such that
$$\left\{\begin{array}{rclll}
     \varphi_t - \Delta\varphi & = & 0 & \text{in} & \Omega\\
     \varphi & = & 0 & \text{on} & \p_\Gamma\Omega,
\end{array}\right.$$
$\varphi \geq 0$, $\varphi$ is even in $x_{n+1}$, $\|\varphi\|_{L^\infty(Q_1)} = 1$, and
$$\frac{1}{C}\varphi_+ \leq \varphi \leq C\varphi_- \quad \text{in} \ Q_R.$$
In particular, $\|\varphi\|_{L^\infty(Q_r)} \leq Cr^{1/2-\varepsilon}$ for all $r \in [1,R]$.
\end{lem}

\begin{proof}
First, by construction and Proposition \ref{prop:cone_solns}, if $\delta_1$ is small enough $\varphi_+$ is a subsolution and $\varphi_-$ is a supersolution for the heat equation in $\Omega$.

Now, from Lemma \ref{lem:phi+_leq_phi-} and the different parabolic homogeneities of $\varphi_+$ and $\varphi_-$, we deduce that
$$\varphi_+ \leq CR^{2\varepsilon}\varphi_- = 4C\varphi_- \quad \text{in} \ Q_R.$$
Note that $C$ does not depend on $\varepsilon$.

Therefore, by the comparison principle there exists $\tilde\varphi \geq 0$, a solution to the heat equation in $\Omega$, vanishing on $\p_\Gamma\Omega$, that satisfies
$$\varphi_+ \leq \tilde\varphi \leq 4C\varphi_- \quad \text{in} \ Q_R.$$

By Proposition \ref{prop:cone_solns}, $1 \leq \|\tilde\varphi\|_{L^\infty(Q_1)} \leq 4C$, and therefore
$$\varphi := \frac{\tilde\varphi}{\|\tilde\varphi\|_{L^\infty(Q_1)}}$$
satisfies the first estimate. The second estimate follows directly from the parabolic scaling of $\varphi_-$.
\end{proof}

Now, the proof continues as in the one sided case.

\begin{proof}[Proof of Proposition \ref{prop:slit_special}]
We follow the same strategy as in the proof of Proposition \ref{prop:special}. First, Lemma \ref{lem:slit_special_exists} replaces Lemma \ref{lem:special_exists}. Then, an analogue to Lemma \ref{lem:special_1/2} can be proved by the same type of blow-up argument. To do so, we use Proposition \ref{prop:slit_bdry_Calpha} for the boundary regularity, and Theorem \ref{thm:slit_liouville} to classify the blow-up limit.

The conclusion follows by an inductive argument as in Lemma \ref{lem:special_induction}, and combining the estimates as in the proof of Proposition \ref{prop:special}.
\end{proof}

\subsection{Expansion in slit domains}\label{subsect:slit_expansion}

The following proposition follows the lines of Proposition \ref{prop:expansion}, adapted to slit domains.

\begin{prop}\label{prop:slit_expansion}
Let $\alpha \in (0,\frac{1}{2})$. There exists $\varepsilon_0 \in (0,1)$, only depending on $\alpha$ and the dimension such that the following holds.

Let $\Omega$ be a parabolic slit domain in $Q_1$ in the sense of Definition \ref{defn:slit_domain} with Lipschitz constant $\varepsilon_0$. Let $u$ be a solution to
$$\left\{\begin{array}{rclll}
u_t - \Delta u & = & f & \text{in} & \Omega\\ 
u & = & 0 & \text{on} & \p_\Gamma\Omega,
\end{array}\right.$$
where $u$ is even in $x_{n+1}$, $\|u\|_{L^\infty(Q_1)} \leq 1$, and $\|f\|_{L^q(Q_1)} \leq 1$ with ${q = (n+3)/(1-\alpha)}$.

Then, for each $r \in (0,1]$ there exists $K_r \in \R$ such that $|K_r| \leq C$ and
$$\|u-K_r\varphi\|_{L^\infty(Q_r)} \leq Cr^{1+\alpha},$$
where $\varphi$ is the special solution introduced in Proposition \ref{prop:slit_special} and $C$ depends only on $\alpha$ and the dimension.
\end{prop}

\begin{proof}
We follow the same four steps as in Proposition \ref{prop:expansion}.

Steps 1 and 4 are identical. Steps 2 and 3 have to be modified in the same way. We will only write the modified Step 2.

\textit{Modified Step 2.} We prove that $w_m \rightarrow w$ locally uniformly along a subsequence, where $w$ is a solution to
\begin{equation}\label{eq:slit_limit_w}
\left\{\begin{array}{rclll}
w_t - \Delta w & = & 0 & \text{in} & \R^{n+2} \setminus \{x_n \leq 0, x_{n+1} = 0\}\\
w & = & 0 & \text{on} & \{x_n \leq 0, x_{n+1} = 0\}.
\end{array}\right.
\end{equation}

Then, by the construction of $w_m$ (omitting the dependence of $f$ on $j_m$),
$$|(\p_t - \Delta)w_m| \leq \frac{2\rho_m^{1-\alpha}}{\theta(\rho_m)}f(\rho_m x, \rho_m^2t),$$
and estimating the right-hand side
$$\left\|\frac{2\rho_m^{1-\alpha}}{\theta(\rho_m)}f(\rho_m x,\rho_m^2 t)\right\|_{L^q(Q_1)} = \frac{2\rho_m^{1-\alpha}}{\theta(\rho_m)}\rho_m^{-(n+3)/q}\|f\|_{L^q(Q_{\rho_m})} \leq \frac{2\|f\|_{L^q(Q_1)}}{\theta(\rho_m)} \leq \frac{2}{\theta(\rho_m)},$$
where we used that $(n+3)/q = 1 - \alpha$.

Moreover, $w_m = 0$ on the appropriate rescaling of $\p_\Gamma\Omega_{j_m}$, and, for every $R \geq 1$, $\|w_m\|_{L^\infty(Q_R)} \leq CR^{1+\alpha}$ for sufficiently large $m$. Hence, by Proposition \ref{prop:slit_bdry_Calpha},
$$\|w_m\|_{C^{0,1/4}_p(Q_R)} \leq C(R),$$
uniformly in $m$, for $m$ large enough. Then, by Arzelà-Ascoli and  Proposition \ref{prop:viscosity_limit}, we obtain that
$$w_m \rightarrow w \in C(\R^{n+2}),$$
locally uniformly along a subsequence, where $w$ is a viscosity solution to \eqref{eq:slit_limit_w} and $\|w\|_{L^\infty(Q_R)} \leq CR^{1+\alpha}$ for all $R \geq 1$. Therefore, by Theorem \ref{thm:slit_liouville} and the fact that $\|w\|_{L^\infty(Q_1)} = 1$, $w = \varphi_0$. In the rest of the proof, $\varphi_0$ plays the role of $(x_n)_+$ in the proof of Proposition \ref{prop:expansion}.

\end{proof}

\subsection{The Boundary Harnack in slit domains}\label{subsect:slit_main}

\begin{proof}[Proof of Theorem \ref{thm:slit_main}]
The strategy of the proof is the same as in Theorem \ref{thm:main}. Let $\alpha = 1 - \frac{n+3}{q}$, and $\varepsilon \in (0,\varepsilon_0)$ (with $\varepsilon_0$ from Proposition \ref{prop:slit_special}) such that $\gamma \leq \frac{1}{2}+\alpha-17\varepsilon$.

Let $R_0 \in (0,\frac{1}{48})$ to be chosen later, $p = (y',y_n,s)$ and $\rho \in (0,R_0)$ such that $Q_{2\rho}(p) \subset \Omega\cap Q_{5/8}$. Assume without loss of generality that $y_{n+1} \geq 0$, and let
$$R := \max\left\{\rho,\frac{y_{n+1}}{8},\frac{|y_n-\Gamma(y',s)|}{8}\right\}.$$
Then, we distinguish four cases (cf. Proposition \ref{prop:slit_bdry_Calpha}).

\textit{Case 1.} $R \geq R_0$, and $y_n > \Gamma(y',s) - 8R$ or $y_{n+1} \geq 2R$. Then, $Q_{2R}(p) \subset \Omega\cap Q_{3/4}$. Now, let us consider two subcases:
\begin{itemize}
    \item If $y_{n+1} < 2R$, $y_n = \Gamma(y',s) + 8R$. Therefore, for all $(x,t) \in Q_R(p)$,
    $$x_n \geq \Gamma(x',t) + (7-2c_0)R > \Gamma(x',t) + 6R.$$
    \item If $y_{n+1} \geq 2R$, since $y_n \geq \Gamma(y',s) - 8R$, for all $(x,t) \in Q_R(p)$,
    $$x_n \geq \Gamma(x',t) - (9+2c_0)R > \Gamma(x',t) - 10R \geq \Gamma(x',t) - 10x_{n+1}.$$
\end{itemize}
In both cases it holds that, for all $(x,t) \in Q_R(p)$,
$$x_n - \Gamma(x',t) + 10|x_{n+1}| > 0,$$
and
$$|x_n - \Gamma(x',t)|+|x_{n+1}| > 2R.$$

Hence, by Proposition \ref{prop:slit_growth}, $v \geq cm > 0$ in $Q_R(p)$. Furthermore, by Proposition \ref{prop:slit_bdry_Calpha}, $\|u\|_{C^{0,\gamma}_p(Q_R(p))} \leq C$ and $\|v\|_{C^{0,\gamma}_p(Q_R(p))} \leq C$. Therefore,
$$\left\|\frac{u}{v}\right\|_{C^{0,\gamma}_p(Q_\rho(p))} \leq \frac{\|u\|_{C^{0,\gamma}_p(Q_\rho(p))}\|v\|_{L^\infty(Q_\rho(p))}+\|u\|_{L^\infty(Q_\rho(p))}\|v\|_{C^{0,\gamma}_p(Q_\rho(p))}}{\inf\limits_{Q_\rho(p)}v^2} \leq Cm^{-2}.$$

\textit{Case 2.} $R \geq R_0$, $y_n = \Gamma(y',s) - 8R$ and $y_{n+1} < 2R_0$. Let 
$$E := Q_{1/8}^+(y',y_n,0,s) := Q_{1/8}(y',y_n,0,s)\cap\{x_{n+1} > 0\} \subset Q_{3/4}.$$
Then,
$$\overline{E}\cap\{x_{n+1}=0\} \subset \{x_n \leq \Gamma(x',t), \ x_{n+1} = 0\}.$$
Moreover, $\rho < R_0 < \frac{1}{48}$, and hence $Q_\rho(p) \subset Q_{1/16}^+(y',y_n,0,s)$.

We will apply Theorem \ref{thm:main} with the following functions defined in $Q_1^+$:
$$\tilde u(x,t) := u\left((y',y_n,0)+\frac{x}{8},s+\frac{t}{64}\right)\quad \text{and}\quad \tilde v(x,t) := \frac{v\left((y',y_n,0)+\frac{x}{8},s+\frac{t}{64}\right)}{\|v\|_{L^\infty(Q_{1/8}^+(y',y_n,0,s))}}.$$

First, $\|\tilde u\|_{L^\infty(Q_1^+)} \leq 1$, $v > 0$ and $\|\tilde v\|_{L^\infty(Q_1^+)} = 1$, and the domain is a half-space so it has Lipschitz constant $0$. Then, about the right-hand side of the equation,
$$\tilde u_t - \Delta \tilde u = \frac{1}{64}f_1\left((y',y_n,0)+\frac{x}{8},s+\frac{t}{64}\right),$$
and hence
$$\left\|\tilde u_t - \Delta \tilde u\right\|_{L^q(Q_1^+)} \leq \frac{8^{(n+3)/q}}{64}\|f\|_{L^q(Q_1)} \leq 1.$$

Moreover,
$$\tilde v\left(\frac{e_{n+1}}{2},-\frac{3}{4}\right) = \frac{v((y',y_n,\frac{1}{16}),-\frac{3}{256})}{\|v\|_{L^\infty(Q_{1/8}^+(y',y_n,0,s))}} \geq \frac{c_2m}{\|v\|_{L^\infty(Q_{1/8}^+(y',y_n,0,s))}} > 0,$$
by Proposition \ref{prop:slit_growth}.

Therefore, we can apply Theorem \ref{thm:main} to obtain
$$\left\|\frac{\tilde u}{\tilde v}\right\|_{C^{0,\gamma}_p(Q_{1/2}^+)} \leq C\|v\|_{L^\infty(Q_{1/8}^+(y',y_n,0,s))}^2m^{-2},$$
and hence
$$\left\|\frac{u}{v}\right\|_{C^{0,\gamma}_p(Q_\rho(p))} \leq C\|v\|_{L^\infty(Q_{1/16}^+(y',y_n,0,s))}m^{-2} \leq Cm^{-2}.$$

\textit{Case 3.} $R < R_0$, and $y_n > \Gamma(y',s) - 8R$ or $y_{n+1} \geq 2R$. Then, $Q_{2R}(p) \subset \Omega\cap Q_{3/4}$. Analogously to \textit{Case 1}, we can apply Proposition \ref{prop:slit_growth} to obtain $v \geq cmR^{1/2+\varepsilon}$ in $Q_R(p)$. Then, the right-hand side of the equation for $u$ can be estimated in $L^{n+2}$ as
\begin{align*}
    \|f_1\|_{L^{n+2}(Q_{2R}(p))} &\leq R^{(n+3)(1/(n+2)-1/q)}\|f_1\|_{L^q(Q_{2R}(p))}\\
    &\leq R^{1/(n+2)+\alpha}\|f_1\|_{L^q(Q_1)} \leq R^{1/(n+2)+\alpha},
\end{align*}
and analogously $\|f_2\|_{L^{n+2}(Q_{2R}(p))} \leq c_0mR^{1/(n+2)+\alpha}$.

Now, let $\varphi$ be the special solution introduced in Proposition \ref{prop:slit_special}, centered at $(y',\Gamma(y',s),0,s)$. Then, $w_1 = u - K_u\varphi$ and $w_2 = v - K_v\varphi$ satisfy
$$\|w_1\|_{L^\infty(Q_{2R}(p))} \leq CR^{1+\alpha} \quad \text{and} \quad \|w_2\|_{L^\infty(Q_{2R}(p))} \leq CR^{1+\alpha}$$
by a translation of Proposition \ref{prop:slit_expansion}.

Finally, we proceed as in the proof of Theorem \ref{thm:main}. By the interior estimates in Theorem \ref{thm:interior_Calpha}, and the growth of $v$ and $\varphi$ (see Propositions \ref{prop:slit_growth} and \ref{prop:slit_special}), and using that $u = w_1 + K_u\varphi$, we estimate
\begin{align*}
    [w_1/v]_{C^{0,\gamma}_p(Q_R(p))} &\leq \frac{[w_1]_{C^{0,\gamma}_p(Q_R(p))}\|v\|_{L^\infty(Q_R(p))}+\|w_1\|_{L^\infty(Q_R(p))}[v]_{C^{0,\gamma}_p(Q_R(p))}}{\inf\limits_{Q_R(p)}v^2}\\
    &\leq C\frac{R^{1+\alpha-\gamma}R^{1/2-\varepsilon}+R^{1+\alpha}R^{1/2-\gamma-\varepsilon}}{m^2R^{2(1/2+\varepsilon)}} \leq Cm^{-2}
\end{align*}
and
\begin{align*}
    [\varphi/v]_{C^{0,\gamma}_p(Q_R(p))} &\leq \frac{[\varphi]_{C^{0,\gamma}_p}\|w_2\|_{L^\infty}+[w_2]_{C^{0,\gamma}_p}\|\varphi\|_{L^\infty}}{\inf(v/\varphi)^2\inf\varphi^2}\\
    &\leq C\frac{R^{1/2-\varepsilon-\gamma}R^{1+\alpha}+R^{1+\alpha-\gamma}R^{1/2-\varepsilon}}{(mR^{2\varepsilon})^2(R^{1/2+\varepsilon})^2} \leq 2Cm^{-2},
\end{align*}
where we omitted the domain to improve readability, and therefore
$$[u/v]_{C^{0,\gamma}_p(Q_R(p))} \leq [w_1/v]_{C^{0,\gamma}_p(Q_R(p))} + |K_u|[\varphi/v]_{C^{0,\gamma}_p(Q_R(p))} \leq Cm^{-2}.$$

\textit{Case 4.} $R < R_0$, $y_n = \Gamma(y',s) - 8R$ and $y_{n+1} < 2R$. Let
$$E := Q_{6R}^+(y',y_n,0,s) \subset Q_{3/4}.$$
Then,
$$\overline{E}\cap\{x_{n+1} = 0\} \subset \{x_n \leq \Gamma(x',t)\}.$$
Moreover, $y_{n+1} + \rho < 3R$, and then $Q_\rho(p) \subset Q_{3R}^+(y',y_n,0,s)$.
As in \textit{Case 3}, we can apply Proposition \ref{prop:slit_expansion} to obtain $w_1 = u - K_u\varphi$ and $w_2 = v - K_v\varphi$ satisfying
$$\|w_1\|_{L^\infty(E)} \leq CR^{1+\alpha} \quad \text{and} \quad \|w_2\|_{L^\infty(E)} \leq CR^{1+\alpha}.$$
Assume without loss of generality that $C \geq 36$. Now, let $x_0 = (y',y_n,0)$ and let
$$\tilde w_1(x,t) := \frac{w_1(x_0+6Rx,s+36R^2t)}{CR^{1+\alpha}} \quad \text{and} \quad \tilde v(x,t) := \frac{v(x_0+6Rx,s+36R^2t)}{\|v\|_{L^\infty(E)}}.$$
Since $y_n = \Gamma(y',s) - 8R$, $(y',y_n,3R,s) \in E$ satisfies
$$y_n - \Gamma(y',s) + 10\cdot3R \geq 0,$$
and by Proposition \ref{prop:slit_growth}, $v(y',y_n,3R,s) \geq cmR^{1/2+\varepsilon}$, and thus $\|v\|_{L^\infty(E)} \geq cmR^{1/2+\varepsilon}$.
Therefore, $\|w_1\|_{L^\infty(Q_1^+)} \leq 1$, $\|v\|_{L^\infty(Q_1^+)} = 1$, and the right-hand sides satisfy
$$\|(\p_t-\Delta)\tilde w_1\|_{L^q(Q_1^+)} = \frac{36R^2}{CR^{1+\alpha}}(6R)^{-(n+3)/q}\|f_1\|_{L^q(E)} \leq 1$$
and
\begin{align*}
    \|(\p_t-\Delta)\tilde v\|_{L^q(Q_1^+)} &= \frac{36R^2}{\|v\|_{L^\infty(E)}}(6R)^{-(n+3)/q}\|f_2\|_{L^q(E)}\\
    &\leq \frac{36R^2(6R)^{\alpha-1}}{cmR^{1/2+\varepsilon}}c_0m \leq c_0(cmR^{2\varepsilon}),
\end{align*}
provided that $R \leq R_0$ is small enough, where we also used that ${(n+3)/q=1-\alpha}$. Finally, by Proposition \ref{prop:slit_growth}, $v(x_0+3Re_n,s-27R^2) \geq cmR^{1/2+\varepsilon}$ which implies $\tilde v(\frac{e_n}{2},-\frac{3}{4}) \geq cmR^{2\varepsilon}$, and also $\|v\|_{L^\infty(E)} \leq CR^{1/2-\varepsilon}$. Thus, by Theorem \ref{thm:main},
$$[\tilde w_1/\tilde v]_{C^{0,\gamma}_p(Q_1^+)} \leq C(m)R^{-4\varepsilon},$$
and undoing the scaling,
$$[w_1/v]_{C^{0,\gamma}_p(E)} \leq \frac{CR^{1+\alpha}}{\|v\|_{L^\infty(E)}}R^{-\gamma}[\tilde w_1/\tilde v]_{C^{0,\gamma}_p(Q_1^+)} \leq C(m)R^{1/2+\alpha-\gamma-5\varepsilon} \leq C(m).$$

Consider now
$$\tilde w_2(x,t) := \frac{w_2(x_0+6Rx,s+36R^2t)}{CR^{1+\alpha}} \quad\text{and}\quad \tilde\varphi(x,t) := \frac{\varphi(x_0+6Rx,s+36R^2t)}{\|\varphi\|_{L^\infty(E)}}.$$
By the parabolic homogeneity of $\varphi$, we get $\tilde\varphi(\frac{e_n}{2},-\frac{3}{4}) \geq cR^{2\varepsilon}$; see Proposition \ref{prop:slit_special}. We also have that 
$$\|(\p_t-\Delta)\tilde w_2\|_{L^q(Q_1^+)} \leq c_0m$$
by the same reasoning as with $\tilde w_1$.

By Theorem \ref{thm:main},
$$[\tilde w_2/\tilde\varphi]_{C^{0,\gamma}_p(Q_1^+)} \leq C(m)R^{-4\varepsilon},$$
and undoing the scaling,
$$[w_2/\varphi]_{C^{0,\gamma}_p(E)} \leq \frac{CR^{1+\alpha}}{\|\varphi\|_{L^\infty(E)}}R^{-\gamma}[\tilde w_2/\tilde \varphi]_{C^{0,\gamma}_p(Q_1^+)} \leq C(m)R^{1/2+\alpha-\gamma-5\varepsilon}.$$
On the other hand, by Corollary \ref{cor:main_carleson},
$$c(m)R^{4\varepsilon} \leq \frac{\tilde v}{\tilde\varphi} \leq C(m)R^{-4\varepsilon} \ \text{in} \  Q_1^+,$$
which after undoing the scaling (by Proposition \ref{prop:slit_growth}) becomes
$$c(m)R^{6\varepsilon} \leq \frac{v}{\varphi} \leq C(m)R^{-6\varepsilon} \ \text{in} \  E.$$

Hence, we can compute
$$[\varphi/v]_{C^{0,\gamma}_p(E)} \leq \frac{[v/\varphi]_{C^{0,\gamma}_p(E)}}{\inf\limits_E(v/\varphi)^2} = \frac{[w_2/\varphi]_{C^{0,\gamma}_p(E)}}{\inf\limits_E(v/\varphi)^2}$$
and applying the previous estimates
$$[\varphi/v]_{C^{0,\gamma}_p(E)} \leq C(m)R^{1/2+\alpha-\gamma-17\varepsilon} \leq C(m).$$

Finally, as in \textit{Case 3},
$$[u/v]_{C^{0,\gamma}_p(Q_R(p))} \leq [w_1/v]_{C^{0,\gamma}_p(Q_R(p))} + |K_u|[\varphi/v]_{C^{0,\gamma}_p(Q_R(p))} \leq C(m),$$
as we wanted to prove.
\end{proof}

\section{Applications to free boundary problems}\label{sect:FB}

\subsection{$C^{1,\alpha}$ free boundary regularity for the parabolic obstacle problem}

The argument in this proof is standard, we write it for the sake of completeness.

\begin{proof}[Proof of Corollary \ref{cor:classic_parabolic_obstacle}]
Let $e \in \R^{n+1}$ be a unit vector. Then, since $u \in C^1$, $u_e$ is a solution to
$$\left\{
\begin{array}{rclll}
\p_t u_e - \Delta u_e & = & f_e & \text{in} & \{u > 0\}\\
u_e & = & 0 & \text{on} & \p\{u > 0\}.
\end{array}
\right.$$

Let now $r > 0$ to be chosen later, and define the functions
$$w_1 := \frac{u_e(rx,r^2t)}{\max\{\|u_e\|_{L^\infty(Q_r)},Cr\}} \quad \text{and} \quad w_2 := \frac{u_n(rx,r^2t)}{\|u_n\|_{L^\infty(Q_r)}}.$$
Then, $\|w_1\|_{L^\infty(Q_1)} \leq 1$, $\|w_2\|_{L^\infty(Q_1)} = 1$ and $w_2 > 0$. Furthermore, $\|u_n\|_{L^\infty(Q_r)} \leq Cr$ by the $C^{1,1}_x$ regularity of $u$, and it follows that
$$w_2\left(\frac{e_n}{2},-\frac{3}{4}\right) \geq \frac{cd(re_n/2,-3r^2/4)}{Cr} \geq \frac{c}{4C},$$
using that $d(re_n/2,-3r^2/4) \geq r/4$ if the Lipschitz constant of the domain is small enough. Finally,
$$\|(\p_t - \Delta)w_1\|_{L^q(Q_1)} \leq \frac{r^{2-(n+2)/q}\|\nabla f\|_{L^q(Q_r)}}{Cr} \leq \frac{\|\nabla f\|_{L^q(Q_1)}}{C}r^{1-(n+2)/q}$$
and
$$\|(\p_t - \Delta)w_2\|_{L^q(Q_1)} \leq \frac{r^{2-(n+2)/q}\|\nabla f\|_{L^q(Q_r)}}{u_n(re_n/2,-r^2/2)} \leq \frac{4\|\nabla f\|_{L^q(Q_1)}}{c}r^{1-(n+2)/q}.$$

Therefore, choosing $r > 0$ small enough (independent of $e$) we can apply Theorem \ref{thm:main} to $w_1$ and $w_2$ and obtain that $w_1/w_2 \in C^{0,\alpha}_p(Q_{1/2}\cap\{w_2 > 0\})$. Now letting $e = e_i$ for all vectors of the coordinate basis, we obtain that
$$\frac{(\nabla u,u_t)}{u_n} \in C^{0,\alpha}_p(Q_{r/2}\cap\{u > 0\}) \subset C^{0,\alpha/2}(Q_{r/2}\cap\{u > 0\}).$$
Notice also that the modulus of this function is bounded below by $1$.

Now, the normal vector to the level sets $\{u = t\}$ for $t > 0$ can be written as
$$\hat n = \frac{(\nabla u,u_t)}{\sqrt{|\nabla u|^2+u_t^2}} = \dfrac{\frac{(\nabla u,u_t)}{u_n}}{\left|\frac{(\nabla u,u_t)}{u_n}\right|} \in C^{0,\alpha/2}(Q_{r/2}\cap\{u > 0\}),$$
hence $\{u = t\}$ is a $C^{1,\alpha/2}$ hypersurface, and taking the limit as $t \downarrow 0$ (uniformly because $u \in C^1$), we obtain that $\p\{u > 0\}$ is $C^{1,\alpha/2}$ as well.
\end{proof}

\subsection{$C^{1,\alpha}$ free boundary regularity for the parabolic Signorini problem}

In the case of slit domains, the proof has to be slightly modified to account for the different scaling of the solution.

\begin{proof}[Proof of Corollary \ref{cor:parabolic_signorini}]
Let $e \in \R^{n+2}$ be a unit vector. Then, since $u \in C^1$, $u_e$ is a solution to
$$\left\{
\begin{array}{rclll}
\p_t u_e - \Delta u_e & = & f_e & \text{in} & Q_1 \setminus \Lambda(u)\\
u_e & = & 0 & \text{on} & \Lambda(u).
\end{array}
\right.$$

Let now $r > 0$ to be chosen later, and define the functions
$$w_1 := \frac{u_e(rx,r^2t)}{\max\{\|u_e\|_{L^\infty(Q_r)},Cr^{1/2}\}} \quad \text{and} \quad w_2 := \frac{u_n(rx,r^2t)}{\|u_n\|_{L^\infty(Q_r)}}.$$
Then, $\|w_1\|_{L^\infty(Q_1)} \leq 1$, $\|w_2\|_{L^\infty(Q_1)} = 1$ and $w_2 > 0$. Furthermore, ${\|u_n\|_{L^\infty(Q_r)} \leq Cr^{1/2}}$ by the $C^{3/2}_x$ regularity of $u$, and it follows that
$$w_2\left(\frac{e_n}{2},-\frac{3}{4}\right) \geq \frac{cd(re_n/2,-3r^2/4)^{1/2}}{Cr^{1/2}} \geq \frac{c}{2C},$$
using that $d(re_n/2,-3r^2/4) \geq r/4$ if the Lipschitz constant of the domain is small enough. Finally,
$$\|(\p_t - \Delta)w_1\|_{L^q(Q_1)} \leq \frac{r^{2 - (n+3)/q}\|\nabla f\|_{L^q(Q_r)}}{Cr^{1/2}} \leq \frac{\|\nabla f\|}{C}r^{3/2 - (n+3)/q}$$
and
$$\|(\p_t - \Delta)w_2\|_{L^q(Q_1)} \leq \frac{r^{2 - (n+3)/q}\|\nabla f\|_{L^q(Q_r)}}{u_n(re_n/2,-r^2/2)} \leq \frac{4\|\nabla f\|}{c}r^{3/2 - (n+3)/q}.$$

Therefore, choosing $r > 0$ small enough (independent of $e$) we can apply Theorem \ref{thm:slit_main} to $w_1$ and $w_2$ and obtain that $w_1/w_2 \in C^{0,\alpha}_p(Q_{1/2}\cap\{w_2 > 0\})$. From here the argument goes on as in the proof of Corollary \ref{cor:classic_parabolic_obstacle}.
\end{proof}

\section{The elliptic boundary Harnack with right-hand side}\label{sect:elliptic}
Applying similar reasoning as in Sections \ref{sect:bdry_growth_reg}, \ref{sect:special_soln}, and \ref{sect:proof_BH}, but using elliptic instead of parabolic theory, one can arrive to the following result, that generalizes the right-hand sides considered in \cite{AS19} and \cite{RT21} for non-divergence operators.

It is noteworthy that even for the Laplacian, this is the first optimal regularity result for quotients of solutions in domains that are less regular than $C^1$.

First, let us define a Lipschitz domain in the elliptic setting.
\begin{defn}\label{defn:elliptic_domains}
We say $\Omega$ is a Lipschitz domain in $B_R$ with Lipschitz constant $L$ if $\Omega$ is the epigraph of a Lipschitz function $\Gamma : B_R' \to \R$, with $\Gamma(0,0) = 0$:
$$\Omega = \big\{(x',x_n) \in B'_R\times(-R,R) \ | \ x_n > \Gamma(x') \big\}, \quad \|\Gamma\|_{C^{0,1}} \leq L.$$
In this context, we will denote the lateral boundary
$$\p_\Gamma\Omega := \big\{(x,x_n) \in B'_R\times(-R,R) \ | \ x_n = \Gamma(x') \big\},$$
which is a subset of the topological boundary of $\Omega$.
\end{defn}

\begin{obs}
    To extend the concept of regularized distance to \textit{one-sided} elliptic Lipschitz domains, a simple approach is to establish a correspondence between elliptic domains and time-independent parabolic domains by adding a dummy variable.
\end{obs}

Finally, it is worth highlighting that the key difference in the proof is the change in scaling between the parabolic ABP estimate, Theorem \ref{thm:ABPKT}, and its elliptic counterpart, which we state below.

\begin{thm}[\protect{\cite[Theorem 3.2]{CC95}}]\label{thm:ABP}
Let $\L$ be a non-divergence form operator as in \eqref{eq:non-divergence_operator} and let $u \in W^{2,n}_\loc$ be a solution to $\L u = f$ in $B_r$, with $f \in L^{n}(B_r)$.

Then,
$$\sup\limits_{B_r} u \leq \sup\limits_{\p B_r} u^+ + Cr\|f\|_{L^{n}(B_r)},$$
where $C$ depends only on the dimension and the ellipticity constants.
\end{thm}

Now we are ready to state our main elliptic result.

\begin{thm}\label{thm:elliptic_main}
Let $0 < \gamma < \alpha < \alpha_0$, $m \in (0,1]$ and let $\L$ be a non-divergence form operator as in \eqref{eq:non-divergence_operator}. There exists $c_0 \in (0,1)$, only depending on $\alpha$, $\gamma$, the dimension and the ellipticity constants, such that the following holds.

Let $\Omega$ be a Lipschitz domain in $B_1$ in the sense of Definition \ref{defn:elliptic_domains} with Lipschitz constant $L \leq c_0$. Let $u$ and $v$ be solutions to
$$\left\{\begin{array}{rclll}
\L u & = & f_1 & \text{in} & \Omega\\ 
u & = & 0 & \text{on} & \p_\Gamma\Omega
\end{array}\right.\quad\text{and}\quad\left\{\begin{array}{rclll}
\L v & = & f_2 & \text{in} & \Omega\\ 
v & = & 0 & \text{on} & \p_\Gamma\Omega,
\end{array}\right.$$
and assume that $\|u\|_{L^\infty(B_1)} \leq 1$, $\|v\|_{L^\infty(B_1)} = 1$, $v > 0$, $v\left(\frac{e_n}{2}\right) \geq m$, and that $f_i = g_i + h_i$ with
$$\|d^{1-\alpha}g_1\|_{L^\infty(B_1)}+\|d^{-\alpha}h_1\|_{L^n(B_1)} \leq 1$$
and
$$\|d^{1-\alpha}g_2\|_{L^\infty(B_1)}+\|d^{-\alpha}h_2\|_{L^n(B_1)} \leq c_0m,$$
where $d$ is the regularized distance introduced in Remark \ref{obs:elliptic_regularized_distance}.

Then,
$$\left\|\frac{u}{v}\right\|_{C^{0,\gamma}(\Omega\cap B_{1/2})} \leq C,$$ where $C$ depends only on $m$, $\alpha$, $\gamma$, the dimension and the ellipticity constants.
\end{thm}

\begin{obs}\label{obs:elliptic_regularized_distance}
The function space considered for the right-hand side is the most general allowed by our proof. Notice how the weighted $L^{n+1}$ norm of the parabolic result translates into a weighted $L^n$ norm in the elliptic setting, due to the different scaling of the ABP estimate.

By a similar argument to Proposition \ref{prop:dist_interpolation}, Theorem \ref{thm:elliptic_main} allows for $f_i \in L^q$ with $q = n/(1-\alpha)$, generalizing \cite{RT21}. It also allows for $|f_i| \leq c_0md^{\alpha - 1}$ as in \cite{AS19}.
\end{obs}

\begin{proof}
It follows from the proof of Theorem \ref{thm:main} and the previous lemmas, using Theorem \ref{thm:ABP} instead of Theorem \ref{thm:ABPKT}. Notice that if we consider the solution to an elliptic problem as a stationary solution for the parabolic problem and try to apply Theorem \ref{thm:main} directly, we obtain a weaker result.
\end{proof}

If we can interchange the roles of $u$ and $v$, we derive a corollary in a similar manner to the parabolic case.

\begin{cor}\label{cor:elliptic_carleson}
Let $\alpha \in (0,\alpha_0)$, $m \in (0,1]$ and let $\L$ be a non-divergence form operator as in \eqref{eq:non-divergence_operator}. There exists $c_0 \in (0,1)$, only depending on $\alpha$, the dimension and the ellipticity constants, such that the following holds.

Let $\Omega$ be a Lipschitz domain in $B_1$ in the sense of Definition \ref{defn:elliptic_domains} with Lipschitz constant $L \leq c_0$. Let $u$ and $v$ be positive solutions to
$$\left\{\begin{array}{rclll}
\L u & = & f_1 & \text{in} & \Omega\\ 
u & = & 0 & \text{on} & \p_\Gamma\Omega
\end{array}\right.\quad\text{and}\quad\left\{\begin{array}{rclll}
\L v & = & f_2 & \text{in} & \Omega\\ 
v & = & 0 & \text{on} & \p_\Gamma\Omega,
\end{array}\right.$$
and assume that $\|u\|_{L^\infty(B_1)} = \|v\|_{L^\infty(B_1)} = 1$, $v\left(\frac{e_n}{2}\right) \geq m$, $u\left(\frac{e_n}{2}\right) \geq m$, and that $f_i = g_i + h_i$ with
$$\|d^{1-\alpha}g_i\|_{L^\infty(B_1)}+\|d^{-\alpha}h_i\|_{L^n(B_1)} \leq c_0m,$$
where $d$ is the regularized distance introduced in Remark \ref{obs:elliptic_regularized_distance}.

Then,
$$\frac{1}{C} \leq \frac{u}{v} \leq C \quad \text{in} \ \Omega\cap B_{1/2},$$ 
where $C$ depends only on $m$, $\alpha$, the dimension and the ellipticity constants.
\end{cor}

\begin{obs}
    The analogous theorems hold for the right notion of \textit{elliptic slit domains} with a right-hand side with the same conditions as in Theorem \ref{thm:elliptic_main}.
\end{obs}

\section{Proof of Corollary \ref{cor:hopf_parabolic_rhs}}\label{sect:hopf}
Using the boundary Harnack, we can combine it with the Hopf lemma to get a Hopf-type estimate for solutions of parabolic equations with a right-hand side. Let us start by defining Dini continuity and the interior $C^{1,\mathrm{Dini}}$ condition.

\begin{defn}
    We say $\omega : [0,\infty) \rightarrow [0,\infty)$ is a Dini modulus of continuity if it is nondecreasing and there exists $r_0 > 0$ such that
    $$\int_0^{r_0}\omega(r)\frac{\mathrm{d}r}{r} < \infty.$$
\end{defn}

\begin{defn}\label{defn:interior_C1_Dini}
    Given $\Omega$ a parabolic Lipschitz domain in $Q_R$, we say $\Omega$ satisfies the interior parabolic $C^{1,\mathrm{Dini}}$ condition at $0$ if (possibly after a rotation) there exists $r_0 > 0$ and a Dini modulus of continuity $\omega$ such that
    $$\{(x',x_n,t) \in Q_{r_0} \ | \ x_n > (|x'|+|t|^{1/2})\omega(|x'|+|t|^{1/2})\} \subset \Omega.$$
\end{defn}

Our starting point will be the following boundary point lemma for parabolic $\mathcal{C}^{1,\mathrm{Dini}}$ domains. We were surprised to not find it in the literature, so we provide it here for completeness. The proof follows the steps in \cite{LZ23} and relies on a standard iteration scheme.

\begin{thm}\label{thm:hopf_parabolic}
Let $\L$ be a non-divergence form operator as in \eqref{eq:non-divergence_operator}. Let $\Omega$ be a parabolic Lipschitz domain in $Q_1$ in the sense of Definition \ref{defn:lipschitz_domain}, and assume that it satisfies the interior $C^{1,\mathrm{Dini}}$ condition at $0$.

Let $u$ be a positive solution to
$$\left\{\begin{array}{rclll}
u_t - \L u & = & 0 & \text{in} & \Omega\\ 
u & = & 0 & \text{on} & \p_\Gamma\Omega,
\end{array}\right.$$
and assume that $u\left(\frac{e_n}{2},-\frac{3}{4}\right) = 1$. Then, for all $r \in (0,\delta)$,
$$u(re_n,0) \geq cr,$$
where $c > 0$ and $\delta$ depend only on the dimension, the ellipticity constants, and the modulus of continuity of the domain.
\end{thm}

We start with an auxiliary lemma for sequences.

\begin{lem}\label{lem:aux_seq}
Let $C > 0$, and let $\{a_k\}$ and $\{w_k\}$ be sequences of positive real numbers such that $a_{k+2} = a_{k+1} - Cw_ka_k$. Then, if
$$\sum w_k \leq \frac{1}{2C}\left(2 + \frac{a_1}{a_2}\right)^{-1},$$
then $a_k > a_2/6$ for all $k \geq 2$.
\end{lem}

\begin{proof}
    First, notice that we may assume $C = 1$ without loss of generality. Then,
    $$a_3 = a_2 - w_1a_1 \geq a_2\left(1 - w_1\frac{a_1}{a_2}\right) > a_2\left(1 - \frac{1}{2}\right) = \frac{a_2}{2}.$$
    Furthermore, if we assume that $2a_{k+1} > a_k$,
    $$a_{k+2} = a_{k+1}\left(1 - w_k\frac{a_k}{a_{k+1}}\right) > a_{k+1}\left(1 - \frac{a_k}{4a_{k+1}}\right) > \frac{a_{k+1}}{2}.$$
    Hence, by induction we see that $2a_{k+1} > a_k$ for all $k \geq 2$.

    Now, by iterating the recurrence we also have, for $k \geq 3$,
    $$a_k > a_3(1 - 2w_3)(1 - 2w_4)\cdot\ldots\cdot(1 - 2w_{k-1}),$$
    and since $2w_k < 1/2$ for all $k$,
    $$a_k > a_3e^{-4w_3}e^{-4w_4}\cdot\ldots\cdots e^{-4w_{k-1}} > a_3e^{-4\sum w_j} > \frac{a_3}{e} > \frac{a_2}{6},$$
    where we used that $e^{-2x} < (1 - x)$ for all $x \in (0,1/2)$.
\end{proof}

Now we are ready to prove the main statement of the section.

\begin{proof}[Proof of Theorem \ref{thm:hopf_parabolic}]
First, from the interior Dini condition we get that there exists $r_0 > 0$ such that
$$\omega(r_0) \leq c_0 \quad \text{and} \quad \int_0^{r_0}\omega(r)\frac{\mathrm{d}r}{r} \leq c_0,$$
where $c_0 > 0$ is a small constant to be chosen later. After scaling, and using the parabolic Harnack inequality, we may assume $r_0 = 1$ and $u(e_n/2,-1/2) = 1$.

Then, we denote
$$\Omega_r^+ := \Omega \cap Q_r \cap \{x_n > 0\}.$$

By the Hopf lemma for flat boundaries applied to the set $\{x_n > \omega(1)\}$, we obtain
$u \geq c_1(x_n - \omega(1))$ in $Q_{1/2}\cap\{x_n > \omega(1)\}$, and using that $u \geq 0$,
$$u \geq c_1x_n - c_1\omega(1) \quad \text{in} \ \Omega_{1/2}^+.$$
Let $a_1 := c_1$ and $b_1 := c_1\omega(1)$. We will prove by induction that
$$u \geq a_kx_n - b_k \quad \text{in} \ \Omega_{2^{-k}}^+,$$
for some positive sequences $a_k$ and $b_k$ satisfying the recurrence relations
$$\begin{cases}
    a_{k+1} &= \ a_k - 2^kCb_k\\
    b_{k+1} &= \ 2^{-k}\omega(2^{-k})a_k.
\end{cases}$$

Indeed, assume it true for a certain $k$ and let $v = u - a_kx_n$, which is $\L$-caloric in $\Omega_{2^{-k}}^+$. On the one hand, $v \geq -b_k$ by induction hypothesis. On the other hand, $v \geq -a_kx_n$ because $u \geq 0$.

Now, let us estimate $v$ in $\Omega_{2^{-k-1}}^+$ from below. To do so, let us define $w_1$ and $w_2$ as
$$\left\{\begin{array}{rclll}
(\p_t - \L)w_1 & = & 0 & \text{in} & Q_{2^{-k}}\cap\{x_n > 0\}\\ 
w_1 & = & -b_k & \text{on} & \p_p Q_{2^{-k}}\cap\{x_n > 0\}\\
w_1 & = & 0 & \text{on} & \{x_n = 0\},
\end{array}
\right.$$
and
$$\left\{\begin{array}{rclll}
(\p_t - \L)w_2 & = & 0 & \text{in} & \Omega\cap Q_{2^{-k}}\\ 
w_2 & = & 0 & \text{on} & \p_p (\Omega\cap Q_{2^{-k}}) \setminus \p_\Gamma\Omega\\
w_2 & = & -a_kx_n & \text{on} & \p_\Gamma\Omega.
\end{array}
\right.$$
Now,
\begin{itemize}
    \item On $\p_p (\Omega\cap Q_{2^{-k}}) \setminus \p_\Gamma\Omega$, $v \geq -b_k = w_1 + w_2$,
    \item On $\p_\Gamma\Omega\cap\{x_n > 0\}$, $v \geq -a_kx_n = w_2 > w_1 + w_2$,
    \item On $\{x_n = 0\}\cap\Omega$, $v \geq 0 = w_1 > w_1 + w_2$.
\end{itemize}
Hence, by the comparison principle, $v \geq w_1 + w_2$ in $\Omega_{2^{-k}}^+$.

Then, we estimate $w_1$ by the boundary Lipschitz regularity of solutions (for flat boundaries), scaling and linearity to obtain
$$w_1 \geq -Cb_k(2^kx_n) \quad \text{in} \ Q_{2^{-k-1}}\cap\{x_n > 0\}$$
and we estimate $w_2$ with the maximum principle with
$$w_2 \geq -a_k\sup\limits_{x \in \p_\Gamma\Omega\cap Q_{2^{-k}}}\{x_n\} \geq -2^{-k}\omega(2^{-k})a_k.$$
Putting everything together one obtains
$$u \geq a_kx_n + w_1 + w_2 \geq (a_k - 2^kCb_k)x_n - 2^{-k}\omega(2^{-k})a_k \quad \text{in} \ \Omega_{2^{-k-1}}^+.$$

Moreover, notice that $a_{k+2} = a_{k+1} - C\omega(2^{-k})a_k$, 
$$a_2 = a_1 - 2Cb_1 = c_1(1 - 2C\omega(1)) \geq c_1(1 - 2Cc_0) \geq \frac{c_1}{2},$$
and
$$\sum\omega(2^{-k}) \leq \sum\frac{1}{\ln(2)}\int_{2^{-k}}^{2^{-k+1}}\omega(r)\frac{\mathrm{d}r}{r} \leq \frac{c_0}{\ln(2)}.$$
Hence, choosing $c_0$ small enough we can apply Lemma \ref{lem:aux_seq} and obtain $a_k \geq c_1/12$ for all $k$. On the other hand, for $k \geq 2$,
$$b_k = 2^{-k+1}a_{k-1}\omega(2^{-k+1}) \leq 2^{-k+1}c_1c_0,$$
using that $a_k$ is decreasing and $\omega$ is nondecreasing. Now if we choose $c_0 \leq 1/96$, we have $b_k \leq 2^{-k-2}a_k$ for all $k \geq 2$, and then for all $r \in [2^{-k-1},2^{-k})$,
$$u(re_n,0) \geq a_kr - b_k \geq a_k(r - 2^{-k-2}) \geq \frac{a_kr}{2} \geq \frac{c_1}{24}r,$$
and the conclusion follows.
\end{proof}

Finally, after combining it with the boundary Harnack and the near-linear solution from Section \ref{sect:special_soln}, we can prove our Hopf lemma for equations with right-hand side.

\begin{proof}[Proof of Corollary \ref{cor:hopf_parabolic_rhs}]
    Let $\varphi$ be the special solution defined in Proposition \ref{prop:special}, and assume that $\varphi(e_n/2,-1/2) = 1$ after normalizing. From Theorem \ref{thm:hopf_parabolic}, for all $r \in (0,\delta)$ and some $c > 0$,
    $$\varphi(re_n,0) \geq cr.$$
    
    Then, divide $u$ by a constant so that $\|u\|_{L^\infty(Q_1)} = 1$. Now we can apply Corollary \ref{cor:main_carleson} to $\varphi$ and $u$ to obtain
    $$\frac{\varphi}{u} \leq C \quad \text{in} \ Q_{1/2},$$
    and hence 
    $$u(re_n,0) \geq C^{-1}\varphi(re_n,0) \geq C^{-1}cr,$$
    for all $r \in (0,\min\{1/2,\delta\})$.
\end{proof}

\appendix

\section{Auxiliary results}\label{sect:app}

\subsection{The space of the right-hand sides}

We will prove an interpolation inequality between weighted $L^p$ spaces that seems classical but we were not able to find in the literature.

\begin{lem}\label{lem:lp_interpolation_1D}
Let $\alpha \in (0,1)$, $p \geq 1$, and let $q = (p+1)/(1-\alpha)$. Let  $f \in L^q((0,1))$. Then,
$$\inf\limits_{\lambda > 0}\left[\lambda + \left(\int_0^1(|f| - \lambda x^{\alpha-1})^p_+x^{-1-p\alpha}\mathrm{d}x\right)^{\frac{1}{p}}\right] \leq 2\|f\|_{L^q((0,1))}.$$
In particular,
$$L^q \subset L^p((0,1);x^{-\frac{1}{p}-\alpha})+L^\infty((0,1);x^{1-\alpha}).$$
\end{lem}

\begin{proof}
Assume without loss of generality that $f \geq 0$. Then, let us do the change of variables $t = (px^p)^{-1}$, and let us also define $h : (\frac{1}{p},\infty) \to \R$ as the function satisfying
$$f(x) = x^{\alpha-1}h(t).$$

On the one hand,
$$\|h\|_{L^q}^q = \int_\frac{1}{p}^\infty h(t)^q\mathrm{d}t = \int_0^1(f(x)x^{1-\alpha})^q\frac{\mathrm{d}x}{x^{p+1}} = \int_0^1f(x)^q\mathrm{d}x = \|f\|_{L^q}^q.$$

On the other hand,
\begin{align*}
    \int_0^1(f(x)-\lambda x^{\alpha-1})^p_+x^{-1-p\alpha}\mathrm{d}x &= \int_\frac{1}{p}^\infty(h(t)-\lambda)^p_+x^{p(\alpha-1)}x^{-1-p\alpha}x^{p+1}\mathrm{d}t\\
    &= \int_\frac{1}{p}^\infty(h(t)-\lambda)_+^p\mathrm{d}t = \|(h - \lambda)_+\|_{L^p}^p.
\end{align*}

Therefore, it suffices to prove that, given $1 \leq p \leq q$,
$$\inf\limits_{\lambda > 0}\left[\lambda + \|(h - \lambda)_+\|_{L^p}\right] \leq 2\|h\|_{L^q}.$$

For that purpose, we estimate
$$\|(h-\lambda)_+\|_{L^p}^p \leq \int(h-\lambda)_+^p \leq \int \chi_{\{h > \lambda\}}\frac{h^q}{\lambda^{q-p}} \leq \lambda^{p-q}\int h^q = \lambda^{p-q}\|h\|_{L^q}^q.$$

Then, if we consider $\lambda = \|h\|_{L^q}$,
$$\inf\limits_{\lambda > 0}\left[\lambda + \|(h-\lambda)_+\|_{L^p}\right] \leq \inf\limits_{\lambda > 0}\left[\lambda + \lambda^{1-q/p}\|h\|_{L^q}^{q/p}\right] \leq 2\|h\|_{L^q}.$$
\end{proof}

As a consequence, we can interpolate between the weighted Lebesgue spaces used in the one-sided boundary Harnack.

\begin{prop}\label{prop:dist_interpolation}
Let $\alpha \in (0,1)$, $q = (n+2)/(1-\alpha)$, and let $\Omega$ be a parabolic Lipschitz domain in $Q_1$ in the sense of Definition \ref{defn:lipschitz_domain}. Let $f \in L^q(\Omega)$.

Then, there exist $g, h : \Omega \to \R$ such that $f = g + h$ and
$$\|d^{1-\alpha}g\|_{L^\infty(\Omega)}+\|d^{-1/(n+1)-\alpha}h\|_{L^{n+1}(\Omega)} \leq 2\|f\|_{L^q(\Omega)},$$
where $d(x',x_n,t) = x_n - \Gamma(x',t)$.
\end{prop}

\begin{proof}
Consider the bi-Lipschitz change of variables ${(y',y_n,s) = (x',x_n - \Gamma(x',t),t)}$. Then, it suffices to apply Lemma \ref{lem:lp_interpolation_1D} with $p = n+1$ in the variable $y_n$ and integrate in $y'$ and $s$.
\end{proof}

\subsection{The regularized distance}
We will give some ideas on how the construction of the regularized distance is done.

\begin{proof}[Sketch of the proof of Lemma \ref{lem:regularized_distance}]
We follow the construction in \cite[Section IV.5]{Lie96}. Let $\varphi \in C^\infty(B_1)$ and $\eta \in C^\infty((0,1))$ be nonnegative cutoff functions with
$$\int_{\R^n}\varphi = \int_{\R}\eta = 1.$$
Let $A = 4\sqrt{L^2+1}$ and $K = \|\eta'\|_{L^1(0,1)}$. We then define
$$F(x, t, \rho) = x_n - \int_{\R^n}\int_{\R}\Gamma\left(x' - \frac{\rho}{A}y,t - \frac{\rho^2}{2(1+K)^2A^2}s\right)\eta(s)\varphi(y)\mathrm{d}s\mathrm{d}y.$$
Recall that $\p_\Gamma\Omega = \{x_n = \Gamma(x',t)\}$.

Then, it can be shown that for each $(x,t) \in \Omega\cap Q_1$, there is a unique $\rho$ such that $F(x,t,\rho) = \rho$, and we choose $d(x,t)$ to be equal to this $\rho$.

From the proof in \cite[Section IV.5]{Lie96}, we obtain
\begin{gather*}
    \frac{1}{2}(x_n - \Gamma(x',t)) \leq d \leq \frac{3}{2}(x_n - \Gamma(x',t)),\\
    \p_nd \geq \frac{2}{3} \quad \text{and} \quad |\nabla_x d| \leq \frac{A}{2}.
\end{gather*}

To obtain the last estimate, one needs to proceed as in \cite[Theorem 3.1]{Lie85}, notice that since $L \leq 1$, $A \in [4,4\sqrt{2}]$ and it can be absorbed into $C_2$, and since $\Gamma$ is not parabolic $C^1$ but only parabolic Lipschitz, one needs to repeat the computations done with the modulus of continuity of $\nabla_x \Gamma$ substituting it by appropriate expressions concerning the regularity of $\Gamma$.

Indeed,
$$|\nabla_x\Gamma(x_1,t_1) - \nabla_x\Gamma(x_2,t_2)| \leq \xi(|x_1-x_2|) + \xi(|t_1-t_2|^{1/2})$$
becomes
$$|\nabla_x\Gamma(x_1,t_1) - \nabla_x\Gamma(x_2,t_2)| \leq 2\|\nabla_x g\|_{L^\infty} = 2L,$$
and
$$|\Gamma(x,t_1) - \Gamma(x,t_2)| \leq |t_1-t_2|^{1/2}\xi(|t_1-t_2|^{1/2})$$
becomes
$$|\Gamma(x,t_1) - \Gamma(x,t_2)| \leq L|t_1-t_2|^{1/2}.$$

After these changes, carrying out the rest of the computations in the proof of \cite[Theorem 3.1]{Lie85} one can deduce that
$$|\p_td| + |D_x^2d| \leq \frac{C_2L}{d}.$$
\end{proof}

\subsection{Blow-up construction}

Let us prove how we can construct the blow-up.

\begin{proof}[Proof of Lemma \ref{lem:aux2}]
First, by the definition of $\theta$, $\theta(r) < \infty$ for each $r > 0$ because $\|u_j\|_{L^\infty(Q_1)} \leq 1$, $\|\varphi_j\|_{L^\infty(Q_1)} = 1$, the $K_{r,j}$ are bounded for a fixed $r$, and by hypothesis $\lim\limits_{r\rightarrow0}\theta(r) = \infty$.

Then, for every positive integer $m$, there exist $\rho_m \geq 1/m$ and $j_m$ such that
$$\rho_m^{-\beta}\|u_{j_m}-K_{\rho_m,j_m}\varphi_{j_m}\|_{L^\infty(Q_{\rho_m})} \geq \frac{1}{2}\theta(1/m) \geq \frac{1}{2}\theta(\rho_m).$$
Let us choose $\rho_m \downarrow 0$ as follows: if $\theta(1/(m+1)) = \theta(1/m)$, we take $\rho_{m+1} = \rho_m$, and if $\theta(1/(m+1)) > \theta(1/m)$, then there is a suitable $\rho_{m+1} \in [1/(m+1),1/m)$.

To compute the growth of the $w_m$, we first need to estimate $\|(K_{2r,j} - K_{r,j})\varphi_j\|_{L^\infty(Q_r)}$. Indeed, using the definition of $\theta$,
\begin{align*}
 \frac{\|K_{2r,j}\varphi_j - K_{r,j}\varphi_j\|_{L^\infty(Q_r)}}{r^\beta\theta(r)}
    &\leq \frac{2^\beta\theta(2r)}{\theta(r)}\frac{\|K_{2r,j}\varphi_j - u_j\|_{L^\infty(Q_{2r})}}{(2r)^\beta\theta(2r)} + \frac{\|u_j - K_{r,j}\varphi_j\|_{L^\infty(Q_r)}}{r^\beta\theta(r)}\\
    &\leq 2^\beta+1.
\end{align*}
Hence, $\|(K_{2r,j} - K_{r,j})\varphi_j\|_{L^\infty(Q_r)} \leq Cr^\beta\theta(r)$. Analogously, for any $\mu \in [1,2]$, 
$$\|(K_{\mu r,j} - K_{r,j})\varphi_j\|_{L^\infty(Q_r)} \leq Cr^\beta\theta(r).$$
Furthermore, given $1 \leq a \leq b$,
\begin{align*}
    \|(K_{2ar,j} - K_{ar,j})\varphi_j\|_{L^\infty(Q_{br})} &\leq |K_{2ar,j}-K_{ar,j}|\|\varphi_j\|_{L^\infty(Q_{br})}\\
    &\leq c_1^{-1}|K_{2ar,j}-K_{ar,j}|(b/a)^{\gamma}\|\varphi_j\|_{L^\infty(Q_{ar})}\\
    &\leq c_1^{-1}(b/a)^{\gamma}\|(K_{2ar,j} - K_{ar,j})\varphi_j\|_{L^\infty(Q_{ar})}\\
    &\leq C(ar)^\beta(b/a)^{\gamma}\theta(ar) \leq Cr^\beta a^{\beta-\gamma}b^{\gamma}\theta(r).
\end{align*}

Our next step is the following computation. If $R = 2^NR_0 \leq 1/r$ with $R_0 \in [1,2)$,
\begin{align*}
    \|(K_{Rr,j} - K_{r,j})\varphi_j\|_{L^\infty(Q_{Rr})} &\leq \|(K_{R_0r,j} - K_{r,j})\varphi_j\|_{L^\infty(Q_{Rr})}\\
    &\quad + \sum\limits_{m = 0}^{N-1} \|(K_{2^{m+1}R_0r,j} - K_{2^mR_0r,j})\varphi_j\|_{L^\infty(Q_{Rr})}\\
    &\leq CR^\gamma r^\beta\theta(r) + CR^{\gamma}r^\beta\theta(r)\sum\limits_{m=0}^{N-1}(2^mR_0)^{\beta-\gamma}\\
    &\leq CR^\gamma r^\beta\theta(r) + CR^\gamma r^\beta\theta(r)R^{\beta-\gamma} \leq C(Rr)^\beta\theta(r).
\end{align*}

Finally, for any $1 \leq R \leq 1/\rho_m$,
\begin{align*}
    \|w_m\|_{L^\infty(Q_R)} &= \frac{\|u_{j_m} - K_{\rho_m,j_m}\varphi_{j_m}\|_{L^\infty(Q_{R\rho_m})}}{\|u_{j_m} - K_{\rho_m,j_m}\varphi_{j_m}\|_{L^\infty(Q_{\rho_m})}}\\
    &\leq \frac{2\|u_{j_m} - K_{\rho_m,j_m}\varphi_{j_m}\|_{L^\infty(Q_{R\rho_m})}}{\rho_m^\beta\theta(\rho_m)}\\
    &\leq \frac{2R^\beta\|u_{j_m} - K_{R\rho_m,j_m}\varphi_{j_m}\|_{L^\infty(Q_{R\rho_m})}}{(R\rho_m)^\beta\theta(\rho_m)}\\
    &\quad + \frac{2\|K_{R\rho_m,j_m}\varphi_{j_m} - K_{\rho_m,j_m}\varphi_{j_m}\|_{L^\infty(Q_{R\rho_m})}}{\rho_m^\beta\theta(\rho_m)}\\
    &\leq \frac{2R^\beta\theta(R\rho_m)}{\theta(\rho_m)} + CR^\beta \leq CR^\beta.
\end{align*}
\end{proof}

\subsection{Homogeneous solutions in the complement of \textit{thin cones}} We will prove Proposition \ref{prop:cone_solns}. Let us change a bit the notation for convenience of the proof. If $\varphi$ is a positive solution to
$$\p_t\varphi - \Delta\varphi = 0 \ \text{in} \ Q_1\setminus C_\eta,$$
where
$$C_\eta := \{x_n \leq \eta(|x'|+|t|^{1/2}), \ x_{n+1} = 0\},$$
and $\varphi$ satisfies $\varphi(\lambda x,\lambda^2 t) = \lambda^\kappa\varphi(x,t)$, for some $\kappa$, then $\kappa$ is uniquely determined (as a function of $\eta$).

Indeed, we can write $\varphi(x,t) = |t|^{\kappa/2}\phi(x/|t|^{1/2})$, and $\phi$ solves the following eigenvalue problem for the Ornstein-Uhlenbeck operator (see \cite[Lemma 5.8]{FRS23}):
\begin{equation}\label{eq:OU_eigenvalue}
    \left\{\begin{array}{rclll}
\L_{OU}\phi + \frac{\kappa}{2}\phi & = & 0 & \text{in} & \R^{n+1} \setminus \tilde C_\eta\\
\phi & = & 0 & \text{on} & \tilde C_\eta,
\end{array}\right.
\end{equation}
where
$$\tilde C_\eta = \{x_n \leq \eta(|x'|+1), \ x_{n+1} = 0\}$$
and
$$\L_{OU}\phi(x) := \Delta\phi(x) - \frac{x}{2}\cdot\nabla\phi(x) = e^{|x|^2/4}\operatorname{Div}(e^{-|x|^2/4}\nabla\phi).$$

Since $\phi$ is positive, it is the first eigenfunction for $\L_{OU}$ in this domain, and therefore by the Rayleigh quotient characterization,
\begin{equation}\label{eq:RL}
\frac{\kappa}{2} = \inf\limits_{u \in C^{0,1}_c(\R^{n+1} \setminus \tilde C_\eta), \ \|u\|_{L^2_w} = 1}\int|\nabla u|^2e^{-|x|^2/4},
\end{equation}
where 
$$\|u\|_{L^2_w}^2 := \int u^2e^{-|x|^2/4}$$
and the infimum is attained by a unique function $\phi_\eta \in L^2_w$ by standard arguments. 

Now we are ready to start the proof of Proposition \ref{prop:cone_solns}. First we will show the stability of minimizers of the Rayleigh quotient, in the following sense:
\begin{lem}\label{lem:OU_minimizer_stability}
Let $\eta \in (-\frac{1}{3},\frac{1}{3})$ and let $\kappa = \kappa(\eta)$ as in \eqref{eq:RL}. Then, for all $\varepsilon > 0$ there exists $\delta > 0$ such that if $\|u\|_{L^2_w} = 1$, $u$ vanishes on $\tilde C_\eta$ and
$$\int|\nabla u|^2e^{-|x|^2/4} < \frac{\kappa}{2} + \delta,$$
then
$$\|u - \phi_\eta\|_{L^2_w} < \varepsilon.$$
\end{lem}

\begin{proof}
    Consider the spectral decomposition of $-\L_{OU}$ in the domain $\R^{n+1} \setminus \tilde C_\eta$, so that the eigenvalues are
    $$0 < \frac{\kappa}{2} < \kappa_2 \leq \ldots$$
    and the eigenfunctions are
    $$\{\phi_\eta, \phi_{\eta,2}, \ldots\}$$
    and form an orthonormal basis with respect to the weighted scalar product
    $$\langle f,g\rangle = \int fge^{-|x|^2/4}.$$
    Then, if we write $u = \alpha\phi_\eta + \sum c_i\phi_{\eta,i}$,
    $$\int|\nabla u|^2e^{-|x|^2/4} = -\int u\L_{OU}u = \alpha^2\frac{\kappa}{2} + \sum c_i^2\kappa_i < \frac{\kappa}{2} + \delta.$$
    On the other hand,
    $$1 = \|u\|_{L^2_w}^2 = \alpha^2 + \sum c_i^2,$$
    and then
    $$\sum c_i^2\kappa_i < \sum c_i^2\frac{\kappa}{2} + \delta \ \Rightarrow \ \sum c_i^2\left(\kappa_i - \frac{\kappa}{2}\right) < \delta,$$
    so it follows that
    $$\|u - \phi_\eta\|_{L^2_w}^2 = 2\sum c_i^2 < \frac{\delta}{\kappa_2 - \kappa/2} < \varepsilon,$$
    as required.
\end{proof}

Then, we see the monotonicity and continuity of the eigenvalue with respect to the domain.
\begin{lem}\label{lem:OU_vap_cont}
Let $\kappa : (-\frac{1}{3},\frac{1}{3}) \rightarrow \R$ as in \eqref{eq:RL}. Then, $\kappa$ is strictly increasing and continuous.
\end{lem}

\begin{proof}
    First we prove the monotonicity. Let $-\frac{1}{3} < \eta_1 < \eta_2 < \frac{1}{3}$. Then, $\tilde C_{\eta_1} \subset \tilde C_{\eta_2}$, so $\R^{n+1} \setminus \tilde C_{\eta_2} \subset \R^{n+1} \setminus \tilde C_{\eta_1}$, and since the infimum in the Rayleigh quotient is taken over more functions in the case of $\eta_1$, we get $\kappa(\eta_1) \leq \kappa(\eta_2)$.

    Now, if $\kappa(\eta_1) = \kappa(\eta_2)$, this means that $\phi_{\eta_2}$ is a solution to \eqref{eq:OU_eigenvalue} with $\eta_1$. But $\phi_{\eta_2}$ is identically zero in $\tilde C_{\eta_2}\setminus\tilde C_{\eta_1}$, and hence we have a solution to an elliptic equation that vanishes in an open subset of the domain, contradicting the strong maximum principle. Therefore it cannot be $\kappa(\eta_1) = \kappa(\eta_2)$ and it must be $\kappa(\eta_1) < \kappa(\eta_2)$.

    On the other hand, to prove continuity we will obtain an upper bound for $\kappa(\eta_2)$ in terms of $\kappa(\eta_1)$ by deforming the domain and the solution for $\eta_1$ to get a competitor.

    Let 
    $$\alpha = \arctan(\eta_2) - \arctan(\eta_1) < 2\arctan(1/3) < \pi/3,$$
    and define $\tau : \mathbb{S}^1 \rightarrow \mathbb{S}^1$ as
    $$\tau(\theta) = \begin{cases}
    \left(1+\frac{3\alpha}{\pi}\right)\theta & |\theta| \leq \frac{\pi}{3},\\
    \theta + \alpha\operatorname{sgn}(\theta) & \frac{\pi}{3} < |\theta| \leq \frac{2\pi}{3},\\
    \left(1-\frac{3\alpha}{\pi}\right)\theta + 3\alpha\operatorname{sgn}(\theta) & \frac{2\pi}{3} < |\theta|.
    \end{cases}$$
    Then, let $\rho : \R^2 \rightarrow \R^2$ be defined in polar coordinates as $\rho(r,\theta) = (r,\tau(\theta))$, and let $J : \R^2 \rightarrow \R^2$ defined as $J(x,y) = (y,x)$. Thus,
    $$\psi(x_1,\ldots,x_{n-1},x_n,x_{n+1}) := \phi_{\eta_1}(x_1,\ldots,J(\rho(J(x_{n-1},x_n))),x_{n+1})$$
    is a positive function that vanishes on $\tilde C_{\eta_2}$, and we can get an upper bound for $\kappa(\eta_2)$ by computing the value of its Rayleigh quotient, i.e.
    $$\frac{\kappa(\eta_2)}{2} \leq \frac{1}{\|\psi\|_{L^2_w}}\int|\nabla\psi|^2e^{-|x|^2/4}.$$
    Observe that the deformation of space introduced by $\rho$ changes volume in a factor of $1 + \frac{3\alpha}{\pi}$, $1$ or $1 - \frac{3\alpha}{\pi}$. Hence, by scaling we obtain the following estimates:
    $$\|\psi\|_{L^2_w} \geq \left(1-\frac{3\alpha}{\pi}\right)\|\phi_{\eta_1}\|_{L^2_w} = 1 - \frac{3\alpha}{\pi}$$
    and
    $$\int |\nabla\psi|^2e^{-|x|^2/4} \leq \left(1 - \frac{3\alpha}{\pi}\right)^{-1}\int |\nabla\phi_{\eta_1}|^2e^{-|x|^2/4} \leq \left(1 - \frac{3\alpha}{\pi}\right)^{-1}\frac{\kappa(\eta_1)}{2},$$
    and combining them we get
    $$\kappa(\eta_1) < \kappa(\eta_2) \leq \left(1 - \frac{3}{\pi}(\arctan(\eta_2) - \arctan(\eta_1))\right)^{-2}\kappa(\eta_1),$$
    which already implies that $\kappa$ is continuous.
\end{proof}

To conclude, we write the following:

\begin{proof}[Proof of Proposition \ref{prop:cone_solns}]
Let $\kappa : (-\frac{1}{3},\frac{1}{3}) \rightarrow \R$ as in Lemma \ref{lem:OU_vap_cont}, that is strictly increasing and continuous, and observe that $\kappa(0) = 1$ because $\varphi_0(x,t) := \operatorname{Re}(\sqrt{x_n + ix_{n+1}})$ is the solution to the original parabolic problem for $\eta = 0$.

Then, there exists $\varepsilon_0 > 0$ such that $\kappa^{-1} : (1 - 2\varepsilon_0,1+2\varepsilon_0) \rightarrow [-\frac{1}{4},\frac{1}{4}]$ is well defined, continuous and strictly increasing. Moreover, for any $\varepsilon \in (-\varepsilon_0,\varepsilon_0)$, let
$$\varphi_\varepsilon(x,t) := \frac{\tilde\varphi_\varepsilon}{\|\tilde\varphi_\varepsilon\|_{L^\infty(Q_1)}} := \frac{|t|^{1/2 + \varepsilon}\phi_\eta(x/|t|^{1/2})}{\||t|^{1/2 + \varepsilon}\phi_\eta(x/|t|^{1/2})\|_{L^\infty(Q_1)}},$$
where $\eta = \kappa^{-1}(1+2\varepsilon)$. To see that $\varphi_\varepsilon$ is well defined, we need to check that $\tilde\varphi_\varepsilon \in L^\infty(Q_1)$ and that it is not identically zero in $Q_1$.

First, recall that $\tilde\varphi_\varepsilon$ is a positive solution to the heat equation in $\R^{n+2} \setminus C_\eta$ that vanishes on $C_\eta$, so in particular it is a subsolution in the full space. Now, $\tilde\varphi_\varepsilon(\cdot,-1) \equiv \phi_\eta$, and then, for all $(x,t) \in Q_{1/2}$,
\begin{align*}
    \tilde\varphi_\varepsilon(x,t) &\leq C\int\phi_\eta(y)e^{-|x-y|^2/4} \leq C\left(\int e^{-|x-y|^2/6}\right)^{1/2}\left(\int\phi_\eta(y)^2e^{-|x-y|^2/3}\right)^{1/2}\\
    &\leq C\left(\int\phi_\eta(y)^2e^{-|y|^2/4}\right)^{1/2} = C,
\end{align*}
where we used that for all $x \in B_{1/2}$,
$$-\frac{|x-y|^2}{3} \leq C - \frac{|y|^2}{4}.$$
Hence, by homogeneity, $\|\tilde\varphi_\varepsilon\|_{L^\infty(Q_1)} \leq 2^{1/2+\varepsilon}C$.

On the other hand, since for all $-\frac{1}{3} < \eta_1 < \eta_2 < \frac{1}{3}$, $\phi_{\eta_2}$ vanishes on $\tilde C_{\eta_1}$, by Lemmas \ref{lem:OU_minimizer_stability} and \ref{lem:OU_vap_cont},
$$\lim\limits_{\eta_2 \rightarrow \eta_1}\|\phi_{\eta_2} - \phi_{\eta_1}\|_{L^2_w} = 0.$$

Now, let $E = B_1\cap\{|x_{n+1}| \geq \frac{1}{4(n+1)}\}$. We claim that, for all $\eta \in [-\frac{1}{4},\frac{1}{4}]$, $\|\phi_\eta\|_{L^\infty(E)} \geq c > 0$, where $c$ is independent of $\eta$. Let us prove it by contradiction. If not, there would exist $\{\eta_k\}$ such that $\|\phi_{\eta_k}\|_{L^\infty(E)} < 1/k$, and therefore, after choosing a subsequence, $\eta_{k_m} \rightarrow \eta_0 \in [-\frac{1}{4},\frac{1}{4}]$, and by continuity $\phi_{\eta_{k_m}} \rightarrow \phi_{\eta_0}$ strongly in $L^2_w$, and hence $\phi_{\eta_0} \equiv 0$ in $E$, contradicting the strong maximum principle.

Furthermore, by the parabolic interior Harnack inequality, $\tilde\varphi_\varepsilon \geq c'$ in the set $Q_{1/2}\cap\{x_{n+1} \geq \frac{1}{4(n+1)}\}$, and by homogeneity $\tilde\varphi_\varepsilon \geq c'$ also in $Q_2\cap\{|x_{n+1}| \geq \frac{1}{n+1}\}$.

Therefore, for all $\varepsilon \in (-\varepsilon_0,\varepsilon_0)$, $\tilde\varphi_\varepsilon \geq c'$ in $Q_2\cap\{|x_{n+1}| \geq \frac{1}{n+1}\}$ and ${\|\tilde\varphi_\varepsilon\|_{L^\infty(Q_1)} \leq C}$, and thus the conclusion follows.
\end{proof}

\textbf{Data availability statement.}
Data sharing not applicable to this article as no datasets were generated or analysed during the current study.

\textbf{Conflict of interest statement.}
The author has no competing interests to declare that are relevant to the content of this article.

\end{document}